\documentclass{amsart}

\usepackage{amsfonts}
\usepackage{graphicx}
\usepackage{hyperref}
\usepackage[ngerman, english]{babel}
\usepackage{datetime}
\usepackage{enumitem} 
\usepackage{mathtools}
\usepackage{cite}
\usepackage{url}
\usepackage[a4paper,margin=20ex]{geometry}
\usepackage{xpatch}
\usepackage{amssymb}
\usepackage{amsfonts}
\usepackage{amsmath}
\usepackage{microtype}
\usepackage{hyperref}
\usepackage[dvipsnames]{xcolor}
\usepackage{tikz}
\usepackage{pgfplots}
\definecolor{citation}{rgb}{0.2,0.58,0.2} 
\definecolor{formula}{rgb}{0,0,1}
\hypersetup{colorlinks,breaklinks,
	citecolor=[rgb]{0.2,0.58,0.2},
	urlcolor=[rgb]{0,0.5,0.5},
	linkcolor=[rgb]{0,0,1}}

\xpatchcmd{\proof}
{\itshape}
{\bfseries}
{}
{}

\newcommand\blfootnote[1]{%
  \begingroup
  \renewcommand\thefootnote{}\footnote{#1}%
  \addtocounter{footnote}{-1}%
  \endgroup
}

\makeatletter
\newcommand*{\rom}[1]{\expandafter\@slowromancap\romannumeral #1@}
\makeatother
\def\Xint#1{\mathchoice
{\XXint\displaystyle\textstyle{#1}}%
{\XXint\textstyle\scriptstyle{#1}}%
{\XXint\scriptstyle\scriptscriptstyle{#1}}%
{\XXint\scriptscriptstyle\scriptscriptstyle{#1}}%
\!\int}
\def\XXint#1#2#3{{\setbox0=\hbox{$#1{#2#3}{\int}$ }
\vcenter{\hbox{$#2#3$ }}\kern-.582\wd0}}

\def\dashint{\Xint-}

\newtheorem*{defin}{Definition}
\newtheorem{thm}{Theorem}[section]
\newtheorem{cor}[thm]{Corollary}
\newtheorem{lem}[thm]{Lemma}
\newtheorem{rem}[thm]{Remark}

\newtheorem{prop}[thm]{Proposition}

\numberwithin{equation}{section}

\title{Regularity theory for nonlocal equations with VMO coefficients}
\author{Simon Nowak}
\address{Universit\"at Bielefeld, Fakult\"at f\"ur Mathematik, Postfach 100131, D-33501 Bielefeld, Germany}
\email{simon.nowak@uni-bielefeld.de}
\begin{document}

\maketitle
\begin{abstract}
We prove higher regularity for nonlinear nonlocal equations with possibly discontinuous coefficients of VMO-type in fractional Sobolev spaces. While for corresponding local elliptic equations with VMO coefficients it is only possible to obtain higher integrability, in our nonlocal setting we are able to also prove a substantial amount of higher differentiability, so that our result is in some sense of purely nonlocal type. By embedding, we also obtain higher H\"older regularity for such nonlocal equations.
\end{abstract}
\pagestyle{headings}

\section{Introduction} 
\subsection{Setting} \label{setting}
In this work, we are dealing with nonlinear nonlocal integro-differential equations of the form \blfootnote{Supported by SFB 1283 of the German Research Foundation.}
\begin{equation} \label{nonlocaleq}
	L_A^\Phi u =  f \text{ in } \Omega \subset \mathbb{R}^n,
\end{equation}
where $\Omega \subset \mathbb{R}^n$ is a domain (= open set) and $f:\Omega \to \mathbb{R}$ is a given function, while $A:\mathbb{R}^n \times \mathbb{R}^n \to \mathbb{R}$ is a coefficient and $\Phi:\mathbb{R} \to \mathbb{R}$ is a nonlinearity with properties to be specified below. Moreover, for some fixed $s \in (0,1)$ the nonlocal operator $L_A^\Phi$ is formally defined by
\begin{equation} \label{no}
	L_A^\Phi u(x) := p.v. \int_{\mathbb{R}^n} \frac{A(x,y)}{|x-y|^{n+2s}} \Phi(u(x)-u(y))dy, \quad x \in \Omega.
\end{equation}
For the sake of simplicity, throughout the paper we assume that $n>2s$.
Moreover, we assume that the coefficient $A$ is measurable and that there exists a constant $\Lambda \geq 1$ such that
\begin{equation} \label{eq1}
	\Lambda^{-1} \leq A(x,y) \leq \Lambda \text{ for almost all } x,y \in \mathbb{R}^n.
\end{equation}
In addition, we require $A$ to be symmetric, that is,
\begin{equation} \label{symmetry}
	A(x,y)=A(y,x) \text{ for almost all } x,y \in \mathbb{R}^n.
\end{equation}
We define $\mathcal{L}_0(\Lambda)$ as the class of all such measurable coefficients $A$ that satisfy the conditions (\ref{eq1}) and (\ref{symmetry}).
Furthermore, we require that the nonlinearity $\Phi$ satisfies $\Phi(0)=0$ and the following Lipschitz continuity and monotonicity assumptions, namely
\begin{equation} \label{PhiLipschitz}
	|\Phi(t)-\Phi(t^\prime)| \leq \Lambda |t-t^\prime| \text{ for all } t,t^\prime \in \mathbb{R}
\end{equation}
and
\begin{equation} \label{PhiMonotone}
	\left (\Phi(t)-\Phi(t^\prime) \right )(t-t^\prime) \geq \Lambda^{-1} (t-t^\prime)^2 \text{ for all } t,t^\prime \in \mathbb{R},
\end{equation}
where for simplicity we use the same constant $\Lambda \geq 1$ as in (\ref{eq1}). The above conditions are for instance satisfied by any $C^1$ function $\Phi$ with $\Phi(0)=0$ such that the image of the first derivative $\Phi^\prime$ of $\Phi$ is contained in $[\Lambda^{-1},\Lambda]$.
Consider the fractional Sobolev space
$$W^{s,2}(\mathbb{R}^n)= \left \{u \in L^2(\mathbb{R}^n) \mathrel{\Big|} \int_{\mathbb{R}^n} \int_{\mathbb{R}^n} \frac{|u(x)-u(y)|^2}{|x-y|^{n+2s}}dydx < \infty \right \}$$
and denote by $W^{s,2}_c(\Omega)$ the set of all functions that belong to $W^{s,2}(\mathbb{R}^n)$ and are compactly supported in $\Omega$. We are now in the position to define weak solutions of the equation (\ref{nonlocaleq}) as follows.
\begin{defin}
	Given $f \in L^\frac{2n}{n+2s}_{loc}(\Omega)$, we say that $u \in W^{s,2}(\mathbb{R}^n)$ is a weak solution of the equation $L_A^\Phi u = f$ in $\Omega$, if 
	\begin{equation} \label{weaksolx1}
		\int_{\mathbb{R}^n} \int_{\mathbb{R}^n} \frac{A(x,y)}{|x-y|^{n+2s}} \Phi(u(x)-u(y))(\varphi(x)-\varphi(y))dydx = \int_{\Omega} f \varphi dx \quad \forall \varphi \in W^{s,2}_c(\Omega).
	\end{equation}
\end{defin}
We remark that the right-hand side of (\ref{weaksolx1}) is finite in view of using H\"older's inequality with H\"older conjugates $\frac{2n}{n+2s}$ and $\frac{2n}{n-2s}$ and the fractional Sobolev embedding (see Proposition \ref{Sobemb}).

In our main results, we require $A$ to be of vanishing mean oscillation close to the diagonal in the following sense.

\begin{defin}
	Let $\delta>0$ and $A \in \mathcal{L}_0(\Lambda)$. We say that $A$ is $\delta$-vanishing in a ball $B \subset \mathbb{R}^n$, if for any $r>0$ and all $x_0,y_0 \in B$ with $B_r(x_0) \subset B$ and $B_r(y_0) \subset B$, we have $$ \dashint_{B_r(x_0)} \dashint_{B_r(y_0)} |A(x,y)-\overline A_{r,x_0,y_0}|dydx \leq \delta ,$$
	where $\overline A_{r,x_0,y_0}:= \dashint_{B_r(x_0)} \dashint_{B_r(y_0)} A(x,y)dydx$. \par
	Moreover, we say that $A$ is $(\delta,R)$-BMO in a domain $\Omega \subset \mathbb{R}^n$ and for some $R>0$, if for any $z \in \Omega$ and any $0<r\leq R$ with $B_r(z) \Subset \Omega$, $A$ is $\delta$-vanishing in $B_r(z)$. \par
	Finally, we say that $A$ is VMO in $\Omega$, if for any $\delta>0$, there exists some $R>0$ such that $A$ is $(\delta,R)$-BMO in $\Omega$.
\end{defin}
If $A$ belongs to the classical space of functions with vanishing mean oscillation $\textnormal{VMO}(\mathbb{R}^{2n})$ (see e.g.\ \cite[section 2.1.1]{PS}, \cite{DiF} or \cite{Sarason}), then $A$ is also VMO in $\mathbb{R}^n$ in the above sense. Nevertheless, our assumption that $A$ is VMO in $\Omega$ is more general, as it roughly speaking only means that $A$ is of vanishing mean oscillation in some arbitrarily small open neighbourhood of the diagonal in $\Omega \times \Omega$, while away from the diagonal in $\Omega \times \Omega$ and outside of $\Omega \times \Omega$ the behaviour of $A$ is allowed to be more general. In particular, if $A$ is continuous in an open neighbourhood of the diagonal in $\Omega \times \Omega$, then $A$ is clearly VMO in $\Omega$. 
Nevertheless, continuity close to the diagonal is not essential, as there are plenty of VMO functions that are discontinuous. For example, assuming that $\Omega$ contains the origin, if for some $\alpha \in (0,1)$ we have
\begin{equation} \label{ex}
	\begin{aligned}
		A(x,y)=
		\begin{cases}
			\textnormal{sin} \left (|\textnormal{log}(|x|+|y|)|^\alpha \right )+2 & \text{ if } x \neq 0 \text{ or } y \neq 0 \\
			0 & \text{ if } x=y=0
		\end{cases}
	\end{aligned}
\end{equation}
or
\begin{equation} \label{ex1}
	\begin{aligned}
		A(x,y)=
		\begin{cases}
			\textnormal{sin} \left (\textnormal{log}|\textnormal{log}(|x|+|y|)| \right )+2 & \text{ if } x \neq 0 \text{ or } y \neq 0 \\
			0 & \text{ if } x=y=0
		\end{cases}
	\end{aligned}
\end{equation}
in an open neighbourhood of $\textnormal{diag}(\Omega \times \Omega)$, then $A$ is VMO in $\Omega$. However, in both cases $A$ is discontinuous at $x=y=0$.
\subsection{Main results} \label{mr}
Our first main result is concerned with Sobolev regularity.

\begin{thm} \label{mainint5z}
	Let $\Omega \subset \mathbb{R}^n$ be a domain, $s \in (0,1)$ and $p \in [2,\infty)$. Moreover, fix some $t$ such that
	\begin{equation} \label{trangexz}
		s \leq t < \min  \left \{2s \left (1-\frac{1}{p} \right) ,1-\frac{2-2s}{p} \right \} = \begin{cases} 2s \left (1-\frac{1}{p} \right), & \text{ if } s \leq 1/2 \\ 1-\frac{2-2s}{p}, & \text{ if } s > 1/2 \end{cases} =:t_{sup}.
	\end{equation} If $A \in \mathcal{L}_0(\Lambda)$ is VMO in $\Omega$ and if $\Phi$ satisfies the conditions (\ref{PhiLipschitz}) and (\ref{PhiMonotone}) with respect to $\Lambda$, 
	then for any weak solution $u \in W^{s,2}(\mathbb{R}^n)$
	of the equation
	$$
	L_A^\Phi u = f \text{ in } \Omega,
	$$
	we have the implication $$f \in L^\frac{np}{n+(2s-t)p}_{loc}(\Omega) \implies u \in W^{t,p}_{loc}(\Omega).$$
\end{thm}

\begin{rem} \label{mainrem} \normalfont
	In fact, in order to arrive at the conclusion that $u \in W^{t,p}_{loc}(\Omega)$ for some $t$ and some $p$ in Theorem \ref{mainint5z}, it is actually enough to assume that $A$ is $(\delta,R)$-BMO in $\Omega$ for some arbitrarily small $R>0$ and some small enough $\delta>0$ depending only on $p,t,n,s$ and $\Lambda$, see Theorem \ref{mainint5} below. This is in line with corresponding results for local elliptic equations, see e.g.\ \cite{ByunLp}.
\end{rem}

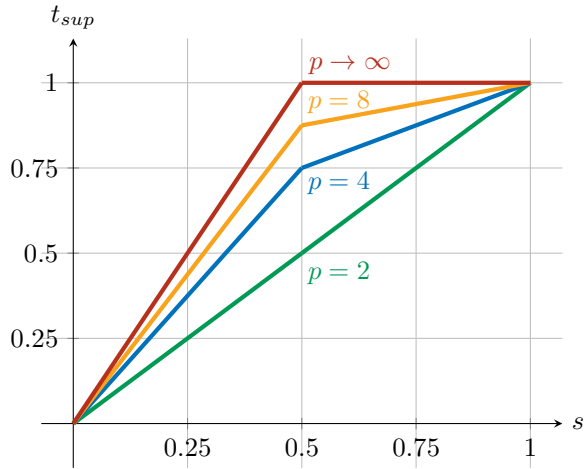
\begin{figure} \label{Fig1}
	\centering
	\begin{tikzpicture}
		\begin{axis}[grid = major,clip = true, 
			clip mode=individual, axis x line = middle, 
			axis y line = middle, xmin=0,ymin=0,xmax=1,ymax=1, samples=50,xlabel=$s$,xlabel style={at=(current axis.right of origin), anchor=west},ylabel=$t_{sup}$,ylabel style={at=(current axis.above origin), anchor=south},xtick = {0,0.25,0.5,0.75,1},ytick = {0,0.25,0.5,0.75,1}, enlarge y limits={rel=0.13}, enlarge x limits={rel=0.07}]
			\addplot[ForestGreen,ultra thick, domain=0:0.5] {x} [yshift=1pt,xshift=14pt] node[below,pos=1] {$p=2$};
			\addplot[ForestGreen,ultra thick,domain=0.5:1] {x};
			\addplot[RoyalBlue,ultra thick, domain=0:0.5] {1.5*x} [yshift=2.5pt,xshift=14pt] node[below,pos=1] {$p=4$};
			\addplot[RoyalBlue,ultra thick,domain=0.5:1] {0.5*x+0.5};
			\addplot[YellowOrange,ultra thick, domain=0:0.5] {1.75*x} [yshift=0.5pt,xshift=14pt] node[above,pos=1] {$p=8$};
			\addplot[YellowOrange,ultra thick,domain=0.5:1] {0.25*x+0.75};
			\addplot[BrickRed,ultra thick, domain=0:0.5] {2*x}[yshift=-1pt,xshift=18pt] node[above,pos=1] {$p \to \infty$};
			\addplot[BrickRed,ultra thick,domain=0.5:1] {1};
		\end{axis}
	\end{tikzpicture}
	\caption{Higher differentiability in relation to $s$ and $p$.}
\end{figure}
An interesting feature of Theorem \ref{mainint5z} is that the differentiability gain indicated by the number $t_{sup}$ depends on the gain of integrability by the relation (\ref{trangexz}). This relation is visualized in the above figure. \par
In particular, on the one hand we observe that in the case when $p$ is close to $2$, that is, in the case of a small gain of integrability, Theorem 1.1 also implies only a small gain of differentiability. On the other hand, in the limit case when $p \to \infty$ we obtain differentiability in the whole range $s \leq t< \min \left \{2s,1 \right \}$, which we expect to be sharp in the case when $A$ is merely VMO. 
An interesting question is if also in the case of smaller values of $p$ the differentiability gain in Theorem \ref{mainint5z} can be improved beyond $t_{sup}$ to the full range $s \leq t< \min \left \{2s,1 \right \}$, or if counterexamples that contradict such an improvement can be constructed. This is because such a gain of differentiability was in fact observed in the recent paper \cite{MSY}. However, in \cite{MSY} this improved regularity is only proved in the linear case when $\Phi(t)=t$ and under some H\"older continuity assumption on $A$, which in particular does not include many examples of discontinuous VMO coefficients like (\ref{ex}) and (\ref{ex1}), see section \ref{pr} for more details. \par
In Theorem \ref{mainint5z}, we stated the result in terms of the higher integrability exponent $p$ at which we arrive, which has the advantage that the statement of Theorem \ref{mainint5z} is relatively clean. However, an interesting question is given by how much higher integrability and differentiability we gain if we instead prescribe the integrability of the source function $f$. This question leads to the following reformulation of Theorem \ref{mainint5z}.
\begin{thm} \label{mainint5zx}
	Let $\Omega \subset \mathbb{R}^n$ be a domain, $s \in (0,1)$ and $f \in L^q_{loc}(\Omega)$ for some $q \in \left (\frac{2n}{n+2s},\infty \right )$. In addition, assume that $A \in \mathcal{L}_0(\Lambda)$ is VMO in $\Omega$ and that $\Phi$ satisfies the conditions (\ref{PhiLipschitz}) and (\ref{PhiMonotone}) with respect to $\Lambda$. Then for any weak solution $u \in W^{s,2}(\mathbb{R}^n)$
	of the equation
	$L_A^\Phi u = f \text{ in } \Omega,$ the following is true. Fix some $t$ such that $s \leq t < 1$.
	\begin{itemize}
		\item If $t$ satisfies
		\begin{equation} \label{trangexy}
			2s-\frac{n}{q}<t < \begin{cases} 2s \left (1- \frac{n}{(n+2s)q} \right ), & \text{ if } s \leq 1/2 \\ 1-\frac{(2-2s)(n+q-2sq)}{(n+2-2s)q}, & \text{ if } s > 1/2, \end{cases}
		\end{equation}
		then we have $u \in W^{t,p}_{loc}(\Omega)$, where $p=\frac{nq}{n-(2s-t)q}$.
		\item If $t$ satisfies 
		\begin{equation} \label{equation}
			t \leq 2s-\frac{n}{q},
		\end{equation}
		then we have $u \in W^{t,p}_{loc}(\Omega)$ for any $p \in (1,\infty)$.
	\end{itemize}
\end{thm}

Note that in the first case of Theorem \ref{mainint5zx} we always have $\frac{nq}{n-(2s-t)q} > 2$, so that we always gain integrability beyond the initial integrability exponent $2$ as well as differentiability beyond the initial differentiability parameter $s$. \par Moreover, we note that in the case when $2s-\frac{n}{q}<1$, it is relatively easy to see that we always have
$$2s-\frac{n}{q} < \begin{cases} 2s \left (1- \frac{n}{(n+2s)q} \right ), & \text{ if } s \leq 1/2 \\ 1-\frac{(2-2s)(n+q-2sq)}{(n+2-2s)q}, & \text{ if } s > 1/2, \end{cases}$$
so that in this case the range of $t$ given by (\ref{trangexy}) is always non-empty. \par Also, we remark that in the case when $2s-\frac{n}{q} \geq 1$, Theorem \ref{mainint5zx} implies that $u$ belongs to $W^{t,p}_{loc}(\Omega)$ for any $t$ in the range $s \leq t < 1$ and any $p \in (1,\infty)$. \par
By embedding, Theorem \ref{mainint5zx} also implies the following higher H\"older regularity result.
\begin{thm} \label{C2sreg1}
	Let $\Omega \subset \mathbb{R}^n$ be a domain, $s \in (0,1)$ and $f \in L^q_{loc}(\Omega)$ for some $q>\frac{n}{2s}$. If $A \in \mathcal{L}_0(\Lambda)$ is VMO in $\Omega$ and $\Phi$ satisfies the conditions (\ref{PhiLipschitz}) and (\ref{PhiMonotone}) with respect to $\Lambda$, then for any weak solution $u \in W^{s,2}(\mathbb{R}^n)$ of the equation
	$$
	L_A^\Phi u = f \text{ in } \Omega,
	$$
	we have
	\begin{equation} \label{Hold}
		\begin{aligned}
			u \in
			\begin{cases}
				C^{2s-\frac{n}{q}}_{loc}(\Omega), & \textnormal{ if } 2s-\frac{n}{q}<1 \\
				C^\alpha_{loc}(\Omega) \quad \forall \alpha \in (0,1), & \textnormal{ if } 2s-\frac{n}{q} \geq 1.
			\end{cases}
		\end{aligned}
	\end{equation}
\end{thm}
While as mentioned it is up to further investigation if the differentiability gain in Theorem \ref{mainint5z} and Theorem \ref{mainint5zx} is optimal, we nevertheless expect the H\"older regularity in Theorem \ref{C2sreg1} to be sharp in the case of VMO coefficients or even continuous coefficients, since even the mentioned improved gain of differentiability along the Sobolev scale in the range $s \leq t< \left \{2s,1 \right \}$ would still only lead to the same amount of H\"older regularity obtained in Theorem \ref{C2sreg1}.

\subsection{Local elliptic equations with VMO coefficients}
For the sake of comparison, let us briefly discuss corresponding regularity results for local elliptic equations in divergence form of the type
\begin{equation} \label{localeq}
	\textnormal{div}(B \nabla u)= f \quad \text{in } \Omega,
\end{equation}
where the matrix of coefficients $B=\{b_{ij}\}_{i,j=1}^n$ is assumed to be uniformly elliptic and bounded. In the linear case when $\Phi(t)=t$, the equation (\ref{localeq}) can in some sense be thought of as a local analogue of the nonlocal equation (\ref{nonlocaleq}) corresponding to the limit case $s=1$. For some rigorous results in this direction, we refer to \cite{GKV}. It is known that if the coefficients $b_{ij}$ belong to $\textnormal{VMO}(\Omega)$ and $f \in L^\frac{np}{n+p}_{loc}(\Omega)$ for some $p>2$, then weak solutions $u \in W^{1,2}_{loc}(\Omega)$ of the equation (\ref{localeq}) belong to $W^{1,p}_{loc}(\Omega)$, see e.g.\ \cite{DiF,ByunLp,IwanSbod} and also \cite{Kin,AM,BD,DongKim1} for more general developments in this direction. This corresponds to our Theorem \ref{mainint5z} in the case when $t=s$. On the other hand, in order to gain any amount of differentiability along the Sobolev scale in the context of local equations, a corresponding amount of differentiability has to be imposed on the coefficients, so that in the case of VMO coefficients in general no differentiability gain at all is attainable. Therefore, the additional differentiability gain in Theorem \ref{mainint5z} is in some sense a purely nonlocal phenomenon. \par 
This nonlocal differential stability effect is also visible in the context of H\"older regularity, although in this case it is somewhat more subtle to recognize it. In fact, by embedding the above $W^{1,p}$ regularity result implies that for any weak solution $u \in W^{1,2}_{loc}(\Omega)$ of (\ref{localeq}) with $f \in L^q_{loc}(\Omega)$ for some $q > \frac{n}{2}$ we indeed have
\begin{align*}
	u \in
	\begin{cases}
		C^{2-\frac{n}{q}}_{loc}(\Omega), & \textnormal{ if } q<n \\
		C^\alpha_{loc}(\Omega) \quad \forall \alpha \in (0,1), & \textnormal{ if } q\geq n,
	\end{cases}
\end{align*}
which at first sight directly corresponds to Theorem \ref{C2sreg1}. However, there is an important difference in the case when $q$ is large, which is due to the differentiability gain in Theorem \ref{mainint5zx}. In order to illustrate this difference, note that in the case when $f \in L^\infty_{loc}(\Omega)$, for any weak solution $u \in W^{1,2}_{loc}(\Omega)$ of (\ref{localeq}) we have $C^\alpha_{loc}(\Omega)$ for any $\alpha \in (0,1)$. Since in some sense the order of the equation (\ref{nonlocaleq}) is $s$ times the order of the equation (\ref{localeq}), one might therefore be tempted to guess that weak solutions $u \in W^{s,2}(\mathbb{R}^n)$ of (\ref{nonlocaleq}) should in general not exceed $C^s$ regularity. However, Theorem \ref{C2sreg1} shows that any such weak solution to (\ref{nonlocaleq}) indeed belongs to $C^\alpha_{loc}(\Omega)$ for any $0<\alpha<\min \big \{2s,1 \big\}$ whenever $f \in L^\infty_{loc}(\Omega)$, exceeding $C^s$ regularity. In particular, in the case when $s \geq 1/2$, such weak solutions to nonlocal equations with VMO coefficients and locally bounded right-hand side enjoy the same amount of H\"older regularity as weak solutions to corresponding local equations with VMO coefficients, despite the fact that the order of such nonlocal equations is lower.

\subsection{Previous results} \label{pr}
In recent years, the regularity theory for weak solutions to nonlocal equations of the type (\ref{nonlocaleq}) has seen a great amount of progress, in particular concerning regularity results of purely nonlocal type, in the sense that as above the obtained regularity is better than one might expect when considering corresponding results for local elliptic equations. \par
Regarding such results for general coefficients $A \in \mathcal{L}_0(\Lambda)$, in \cite{selfimpro} and \cite{Schikorra} it is demonstrated that weak solutions to nonlocal equations of the type (\ref{nonlocaleq}) are slightly higher differentiable and higher integrable along the scale of fractional Sobolev spaces, which is a phenomenon not shared by local elliptic equations of the type (\ref{localeq}) with merely measurable coefficients, where it is only possible to obtain higher integrability. \par
Concerning higher Sobolev regularity of purely nonlocal type, in \cite{MSY} the authors in particular show that in the linear case when $\Phi(t)=t$, if $\Omega=\mathbb{R}^n$ and if the mapping $x \mapsto A(x,y)$ is globally H\"older continuous with some arbitrary H\"older exponent, then the statement of Theorem \ref{mainint5z} holds for $t$ in the improved range $s \leq t<\min \big \{2s,1 \big\}$. As we already discussed briefly in section \ref{mr}, an interesting question is therefore if the regularity obtained in \cite{MSY} can be replicated in our general setting of possibly nonlinear equations with VMO coefficients posed on general domains $\Omega \subset \mathbb{R}^n$, in particular since many examples of discontinuous VMO coefficients like (\ref{ex}) or (\ref{ex1}) are not covered by the H\"older continuity assumption in \cite{MSY}. \par
On the other hand, regarding higher H\"older regularity, by the Sobolev embedding the mentioned Sobolev regularity result in \cite{MSY} implies exactly the same amount of H\"older regularity given by (\ref{Hold}) from Theorem \ref{C2sreg1} under the mentioned assumptions imposed in \cite{MSY}. In other words, although in comparison to \cite{MSY} in general we obtain less differentiability along the Sobolev scale, we nevertheless gain enough differentiability in order to obtain the same amount of H\"older regularity by embedding in our general setting. A similar H\"older regularity result was obtained in \cite{Fall}, again in the case of linear equations, but allowing for coefficients that are merely continuous. Concerning higher H\"older regularity for possibly nonlinear equations, in \cite{MeH} it is in particular proved that if $\Phi$ satisfies the assumptions (\ref{PhiLipschitz}) and (\ref{PhiMonotone}) and $A$ is continuous in $\Omega$, then weak solutions $u$ of (\ref{nonlocaleq}) belong to $C^\alpha_{loc}(\Omega)$ for any $0<\alpha<\min \big \{2s-\frac{n}{q},1 \big\}$ whenever $f \in L^q_{loc}(\Omega)$ for some $q> \frac{n}{2s}$, which almost matches the regularity (\ref{Hold}) obtained in Theorem \ref{C2sreg1}. In addition, while in comparison with Theorem \ref{C2sreg1} the result in \cite{MeH} does not include general discontinuous coefficients of VMO-type, it in fact holds for a slightly larger class of coefficients than simply continuous ones, including in particular coefficients that are translation invariant inside of $\Omega$. Nevertheless, our approach can easily be modified in order to prove our main results under the assumption on $A$ from \cite{MeH}, see Remark \ref{endremark}. In addition, the H\"older regularity result in \cite{MeH} holds for a slightly larger class of weak solutions called local weak solutions, essentially only assuming that $u \in W^{s,2}_{loc}(\Omega)$ and the finiteness of the nonlocal tails of $u$, see \cite{MeH}. While we believe that our approach can be modified in order to generalize our main results to this setting of local weak solutions, we decided not to insist on this point, in particular since this would require also a revision of the previous work \cite{selfimpro}. \par
Let us also mention that in \cite{MeN} (see also \cite{Me}), Theorem \ref{mainint5z} was proved in the case when $t=s$, that is, without the additional differentiability gain, under essentially the same assumptions on $A$ and $\Phi$ as in \cite{MeH}. \par 
More results concerning Sobolev regularity for nonlocal equations are for example proved in \cite{BL,Cozzi,Warma,Grubb,KassMengScott,MP,YLKS,DongKim}, while some more results on H\"older regularity are proved in \cite{BLS,Fall1,NonlocalGeneral,CSa,finnish,Kassmann,CK,Chaker,Stinga,Silvestre,CDF,Peral}. Furthermore, for various regularity results regarding nonlocal equations similar to (\ref{nonlocaleq}) in the more general setting of measure data, we refer to \cite{mdata}.

\subsection{Approach} \label{app}
Our approach is mainly influenced by techniques introduced in \cite{CaffarelliPeral} and \cite{selfimpro}. Namely, in \cite{CaffarelliPeral} techniques were developed allowing to prove higher integrability of the gradient $\nabla u$ of weak solutions to local equations with VMO coefficients of the type (\ref{localeq}), which corresponds to the $W^{1,p}$ regularity theory briefly discussed in section \ref{mr}. \par 
The approach can be summarized as follows. The first step is to use the assumption that the coefficients $b_{ij}$ are VMO in order to locally approximate the gradient of some weak solution $u$ of (\ref{localeq}) by the gradient of a weak solution $v$ to a suitable homogeneous equation with constant coefficients. In order to include discontinuous coefficients of VMO type into the analysis, one uses the fact that $\nabla u$ is known to satisfy a $L^{2+\gamma}_{loc}$ estimate for some small $\gamma>0$, which can be proved in the general setting of merely bounded measurable coefficients by means of so-called Gehring-type lemmas.
One then exploits the fact that the approximate solution $v$, which in the local case up to a change of coordinates is simply a harmonic function, is already known to satisfy a local Lipschitz estimate in order to transfer some regularity from $v$ to $u$. This transfer of regularity is achieved by covering the level sets of the Hardy-Littlewood maximal function of $|\nabla u|^2$ of the form $\left \{\mathcal{M}(|\nabla u|^2) > \lambda^2 \right \}$ by dyadic cubes that are chosen by means of an exit time argument and form a so-called Calder\'on-Zygmund covering, essentially meaning that the cubes in the covering have in some sense good density properties with respect to the level set that is covered by them. Combined with fact that $\nabla u$ can be approximated by the gradient of a harmonic and therefore very regular function, these good density properties are then exploited in order gain control of the measures of the cubes in the covering by means of so-called good-$\lambda$ inequalities. By standard arguments from measure theory, this then allows to prove the desired higher integrability of $\nabla u$, which then implies the desired $W^{1,p}_{loc}$ estimate. \par 
Adapting this approach in order to prove higher Sobolev regularity for nonlocal equations of the type (\ref{nonlocaleq}) comes with a number of obstacles. In particular, a main challenge in the nonlocal context is to find a suitable replacement for the gradient $\nabla u$ which is used in the local context. In \cite{Me} and \cite{MeN}, the above approach was executed for weak solutions $u$ to nonlocal equations of the type (\ref{nonlocaleq}) with the local gradient replaced by the nonlocal gradient-type operator
\begin{equation} \label{sgrad}
	\nabla^s u(x)= \left ( \int_{\mathbb{R}^n} \frac{(u(x)-u(y))^2}{|x-y|^{n+2s}}dy \right )^{\frac{1}{2}}.
\end{equation}
In view of an alternative characterization of Bessel potential spaces, the obtained higher integrability of $\nabla^s u$ then leads to $W^{s,p}_{loc}$ regularity, which corresponds to the case of no differentiability gain as in the setting of local equations. However, as no local higher integrability estimate for small exponents is known for $\nabla^s u$ for weak solutions $u$ to (\ref{nonlocaleq}), the main result in \cite{MeN} does not include the case of VMO coefficients. In addition, while this result corresponds to the $W^{1,p}_{loc}$ estimate obtained in the setting of local equations, considering the gradient-type operator (\ref{sgrad}) does not lead to any higher differentiability. \par 
In order to also gain higher differentiability and include the case of VMO coefficients, we instead use another nonlocal-type gradient operator which is inspired by \cite{selfimpro}.
Fix some $\theta \in \left (0,\frac{1}{2} \right)$. We define a Borel measure $\mu$ on $\mathbb{R}^{2n}$ as follows. For any measurable set $E \subset \mathbb{R}^{2n}$, set
\begin{equation} \label{mudef}
	\mu(E):= \int_{E} \frac{dxdy}{|x-y|^{n-2\theta}}.
\end{equation}
Moreover, for any function $u:\mathbb{R}^n \to \mathbb{R}$ and $(x,y) \in \mathbb{R}^{2n}$ with $x \neq y$, we define the function 
\begin{equation} \label{Udef}
	U(x,y):=\frac{|u(x)-u(y)|}{|x-y|^{s+\theta}}.
\end{equation}
For any domain $\Omega \subset \mathbb{R}^n$, we then clearly have $u \in W^{s,2}(\Omega)$ if and only if $u \in L^2(\Omega)$ and $U \in L^2(\Omega \times \Omega,\mu)$, so that $U$ and $\mu$ are in some sense in duality.
Regarding larger exponents, by a simple computation for any $p > 2$ and $s_\theta:=s+\theta \left (1-\frac{2}{p} \right )>s$, we have 
\begin{equation} \label{muequiv}
	u \in W^{s_\theta,p}(\Omega) \quad \text{if and only if} \quad u \in L^p(\Omega) \text{ and } U \in L^p(\Omega \times \Omega,\mu).
\end{equation}
Therefore, in contrast to the gradient-type operator $\nabla^s$, by proving higher integrability of the gradient-type function $U$ with respect to the measure $\mu$, we do not only gain regularity along the integrability scale of fractional Sobolev spaces, but also a substantial amount of higher differentiability! However, proving this higher integrability of $U$ in the case when $A$ is merely VMO in $\Omega$ comes with a number of additional difficulties. In order to accomplish this, we combine nonlocal adaptations of the approximation and covering techniques from \cite{CaffarelliPeral} with adaptations of some further covering and combinatorial techniques from \cite{selfimpro}. \par
First of all, in order to include equations with VMO coefficients, we need a local higher integrability result for $U$ for small exponents, which is proved in \cite{selfimpro} for $\theta>0$ small enough, which is sufficient for our purposes. \par Furthermore, in contrast to the functions $\nabla u$ and $\nabla^s u$ which are defined on $\mathbb{R}^n$, the function $U$ is defined on $\mathbb{R}^{2n}$. In particular, the level sets of the maximal function with respect to $\mu$ of $U^2$, that is, the sets of the form $\left \{\mathcal{M}(U^2) > \lambda^2 \right \}$, are subsets of $\mathbb{R}^{2n}$ instead of $\mathbb{R}^{n}$. Therefore, in this setting we need to run an exit time argument in $\mathbb{R}^{2n}$ instead of $\mathbb{R}^{n}$ in order to cover the level set of $U$ by Calder\'on-Zygmund cubes in $\mathbb{R}^{2n}$. In other words, every dyadic cube $\mathcal{K}$ in the corresponding Calder\'on-Zygmund covering is of the form $\mathcal{K}=K_1 \times K_2$, where $K_1$ and $K_2$ are dyadic cubes in $\mathbb{R}^n$. A major technical issue that arises at this point is that for cubes $\mathcal{K}=K_1 \times K_2$ that are far away from the diagonal in the sense that $\textnormal{dist}(K_1,K_2)$ is large, the information that $u$ solves a nonlocal equation of the type (\ref{nonlocaleq}) cannot be used effectively. For this reason, we additionally construct an auxiliary diagonal cover consisting of diagonal balls $\mathcal{B}=B \times B$ that once again have nice density properties with respect to the level set $\left \{\mathcal{M}(U^2) > \lambda^2 \right \}$. Since close to the diagonal the information given by the equation can be used much more efficiently, we construct this auxiliary cover in a way such that the exit time at which the balls are chosen is somewhat smaller than the corresponding exit time at which the corresponding Calder\'on-Zygmund cubes are chosen, so that the balls in the auxiliary cover tend to be somewhat larger than the corresponding Calder\'on-Zygmund cubes. All in all, roughly speaking we have $$\left \{\mathcal{M}(U^2) > \lambda^2 \right \} \subset \bigcup \mathcal{B} \cup \bigcup \mathcal{K},$$ 
where the balls $\mathcal{B}$ are diagonal balls with good density properties and the cubes $\mathcal{K}$ are Calder\'on-Zygmund cubes that are far away from the diagonal. \par
The measures of the balls in the auxiliary diagonal cover can then be estimated by approximating $U$ by a corresponding function $V$ in small enough balls, which is given as in (\ref{Udef}) with $u$ replaced by a weak solution $v$ of a corresponding equation of the form $L^\Phi_{\widetilde A}v=0$, where the coefficient $A$ is locally replaced by a suitable constant, while the global behaviour of $A$ has to be left unchanged, since our assumption that $A$ is VMO is local in nature. This leads to the issue that proving a strong enough estimate for $v$, enabling us to transfer enough regularity to $u$, is much more difficult than in the setting of linear local equations, where as mentioned, the approximate solution is effectively simply a harmonic function. Nevertheless, in \cite{MeH} it is proved that such weak solutions $v$ to $L^\Phi_{\widetilde A}v=0$ satisfy a $C^{s+\theta}_{loc}$ estimate in the restricted range $0<\theta < \min \{s,1-s\}$. This H\"older estimate directly implies that $V$ satisfies an $L^\infty_{loc}$ estimate, which is sufficient in order to control the measures of the balls in the auxiliary diagonal cover. \par 
As already indicated, the task of controlling the measures of the off-diagonal Calder\'on-Zygmund cubes requires additional ideas, since far from the diagonal the information provided by the equation is only of very limited use. In order to bypass this problem in the context of proving higher integrability and differentiability of $u$ for small exponents, in \cite[Lemma 5.3]{selfimpro} it was noted that on cubes that are far away from the diagonal, $L^2$-reverse H\"older-type inequalities hold for $U$ without relying on the equation.
Since we want to prove higher integrability also for large exponents, we overcome this problem by noticing that such reverse H\"older-type inequalities for off-diagonal cubes also hold for larger exponents. However, as in \cite[Lemma 5.3]{selfimpro}, these reverse H\"older-type inequalities come with additional diagonal correction terms involving diagonal cubes that do not belong to the original Calder\'on-Zygmund covering, leading to serious difficulties. These difficulties are bypassed by an involved combinatorial argument inspired by a corresponding one in \cite{selfimpro}, enabling us to control also the measures of the off-diagonal Calder\'on-Zygmund cubes. \par 
By combining the estimates for the measures of the diagonal balls and off-diagonal cubes, we are then able to estimate the measure of the level set $\left \{\mathcal{M}(U^2) > \lambda^2 \right \}$ for $\lambda$ large enough, which by a standard application of Fubini's theorem and standard properties of the Hardy-Littlewood maximal function implies the desired $L^p_{loc}$ estimate for $U$ in the form of an a priori estimate, which is then used in order to prove the desired regularity by standard smoothing techniques based on mollifiers.

\subsection{Outline of the paper}
The paper is organized as follows. In section \ref{fracSob}, we formally introduce the fractional Sobolev spaces $W^{s,p}$ and also another type of fractional Sobolev spaces, namely the Bessel potential spaces $H^{s,p}$. \par In section \ref{dm}, we turn to proving some simple properties of the measure $\mu$ introduced in the previous section \ref{app}, while in section \ref{HL} we define the Hardy-Littlewood maximal function with respect to the measure $\mu$ and mention some important properties of it. \par 
In section \ref{pe}, we then discuss some preliminary estimates for nonlocal equations which are essentially known. More precisely, in section \ref{hr} we briefly recall the mentioned higher H\"older regularity result from \cite{MeH}, while in section \ref{kms} we recall the mentioned Sobolev regularity result for small exponents contained in \cite{selfimpro}. In section \ref{fl1}, we then prove a result about $H^{2s,p}$ estimates for the homogeneous Dirichlet problem involving the fractional Poisson-type equation $(-\Delta)^s g=f$, where $(-\Delta)^s$ is the fractional Laplacian. This estimate allows us to focus on proving regularity for nonlocal equations of the type $L_A^\Phi u=(-\Delta)^s g$ instead of (\ref{nonlocaleq}), since once we are able to transfer a sufficient amount of regularity from $g$ to $u$, Theorem \ref{mainint5z} follows by first transferring the regularity from $f$ to some solution $g$ of $(-\Delta)^s g=f$ and then from $g$ to weak solutions $u$ of (\ref{nonlocaleq}). \par 
The rest of the paper is then devoted to the proof of our main results. Namely, in section \ref{ce} we prove a comparison estimate enabling us to carry out the approximation argument and also the smoothing procedure mentioned in section \ref{app}. Section \ref{gl} is devoted to proving good-$\lambda$ inequalities, both at the diagonal and far away from the diagonal. In section \ref{cover}, we then set up the mentioned covering argument and use the good-$\lambda$ inequalities from section \ref{gl} in order to estimate the measure of the level sets of $\mathcal{M}(U^2)$. In section \ref{ape}, this level set estimate is then used in order to prove the desired regularity in the form of an a priori estimate. Finally, in section \ref{pmr} we then use smoothing techniques in order to deduce our main results from the a priori estimates obtained in section \ref{ape}.

\subsection{Some definitions and notation}
For convenience, let us fix some notation which we use throughout the paper. By $C,c$ and $C_i,c_i$, $i \in \mathbb{N}_0$, we always denote positive constants, while dependences on parameters of the constants will be shown in parentheses. As usual, by
$$ B_r(x_0):= \{x \in \mathbb{R}^n \mid |x-x_0|<r \}$$
we denote the open euclidean ball with center $x_0 \in \mathbb{R}^n$ and radius $r>0$. We also set $B_r:=B_r(0)$. In addition, by 
$$ Q_r(x_0):= \{x \in \mathbb{R}^n \mid |x-x_0|_\infty <r/2\}$$
we denote the open cube with center $x_0 \in \mathbb{R}^n$ and sidelength $r>0$.
Moreover, if $E \subset \mathbb{R}^n$ is measurable, then by $|E|$ we denote the $n$-dimensional Lebesgue-measure of $E$. If $0<|E|<\infty$, then for any $u \in L^1(E)$ we define
$$ \overline u_{E}:= \dashint_{E} u(x)dx := \frac{1}{|E|} \int_{E} u(x)dx.$$
Throughout this paper, we often consider integrals and functions on $\mathbb{R}^{2n}=\mathbb{R}^{n} \times \mathbb{R}^{n}$. Instead of dealing with the usual euclidean balls in $\mathbb{R}^{2n}$, for this purpose it is more convenient for us to use the balls generated by the norm
$$ ||(x_0,y_0)||:=\max\{|x_0|,|y_0|\}, \quad (x_0,y_0) \in \mathbb{R}^{2n}.$$
These balls with center $(x_0,y_0) \in \mathbb{R}^{2n}$ and radius $r>0$ are denoted by $\mathcal{B}_r(x_0,y_0)$ and are of the form
$$\mathcal{B}_r(x_0,y_0):=B_r(x_0) \times B_r(y_0).$$
In the case when $x_0=y_0$ we also write $\mathcal{B}_r(x_0):=\mathcal{B}_r(x_0,x_0),$
we call such balls diagonal balls. We also set $\mathcal{B}_r:=\mathcal{B}_r(0)$. Similarly, for $x_0,y_0 \in \mathbb{R}^n$ and $r>0$ we define $\mathcal{Q}_r(x_0,y_0):=Q_r(x_0) \times Q_r(y_0)$ and $\mathcal{Q}_r(x_0):=\mathcal{Q}_r(x_0,x_0)$ and also $\mathcal{Q}_r:=\mathcal{Q}_r(0)$.
\section{Fractional Sobolev spaces} \label{fracSob}
\begin{defin}
	Let $\Omega \subset \mathbb{R}^n$ be a domain. For $p \in [1,\infty)$ and $s \in (0,1)$, we define the fractional Sobolev space
	$$W^{s,p}(\Omega):=\left \{u \in L^p(\Omega) \mathrel{\Big|} \int_{\Omega} \int_{\Omega} \frac{|u(x)-u(y)|^p}{|x-y|^{n+sp}}dydx<\infty \right \}$$
	with norm
	$$ ||u||_{W^{s,p}(\Omega)} := \left (||u||_{L^p(\Omega)}^p + [u]_{W^{s,p}(\Omega)}^p \right )^{1/p} ,$$
	where
	$$ [u]_{W^{s,p}(\Omega)}:=\left (\int_{\Omega} \int_{\Omega} \frac{|u(x)-u(y)|^p}{|x-y|^{n+sp}}dydx \right )^{1/p} .$$
	Moreover, we define the corresponding local fractional Sobolev spaces by
	$$ W^{s,p}_{loc}(\Omega):= \left \{ u \in L^p_{loc}(\Omega) \mid u \in W^{s,p}(\Omega^\prime) \text{ for any domain } \Omega^\prime \Subset \Omega \right \}.$$
	Also, we define the space 
	$$W^{s,p}_0(\Omega):= \left \{u \in W^{s,2}(\mathbb{R}^n) \mid u = 0 \text{ in } \mathbb{R}^n \setminus \Omega \right \}.$$
\end{defin}

We use the following fractional Poincar\'e inequality, see \cite[section 4]{Mingione}.
\begin{lem} \label{Poincare} (fractional Poincar\'e inequality)
	Let $s \in (0,1)$, $p \in [1,\infty)$, $r>0$ and $x_0 \in \mathbb{R}^n$. For any $u \in W^{s,p}(B_r(x_0))$, we have
	$$ \int_{B_r(x_0)} \left | u(x)- \overline u_{B_r(x_0)} \right |^p dx \leq C r^{sp} \int_{B_r(x_0)} \int_{B_r(x_0)} \frac{|u(x)-u(y)|^p}{|x-y|^{n+sp}}dydx,$$
	where $C=C(s,p)>0$.
\end{lem}
We also use another Poncar\'e-type inequality, see \cite[Lemma 2.3]{MeH}.
\begin{lem} \label{Friedrichsx} (fractional Friedrichs-Poincar\'e inequality)
	Let $s \in(0,1)$ and consider a bounded domain $\Omega \subset \mathbb{R}^n$. For any $u \in W^{s,2}_0(\Omega)$, we have
	\begin{equation} \label{FPI4}
	\int_{\mathbb{R}^n} |u(x)|^2 dx \leq C |\Omega|^{\frac{2s}{n}} \int_{\mathbb{R}^n} \int_{\mathbb{R}^n} \frac{|u(x)-u(y)|^2}{|x-y|^{n+2s}}dydx,
	\end{equation}
	where $C=C(n,s)>0$.
\end{lem}
For the following embedding results we refer to \cite[Theorem 6.7, Theorem 6.10, Theorem 8.2]{Hitch}.
\begin{prop} \label{Sobemb}
Let $\Omega \subset \mathbb{R}^n$ be a Lipschitz domain, $s \in (0,1)$ and $p \in [1,\infty)$.
\begin{itemize}
	\item If $sp<n$, then we have the continuous embedding
	$$ W^{s,p}(\Omega) \hookrightarrow L^\frac{np}{n-sp}(\Omega).$$
	\item If $sp=n$, then for any $q \in [1,\infty)$ we have the continuous embedding
	$$ W^{s,p}(\Omega) \hookrightarrow L^q(\Omega).$$
	\item If $sp>n$, then we have the continuous embedding
	$$ W^{s,p}(\Omega) \hookrightarrow C^{s-\frac{n}{p}}(\Omega).$$
\end{itemize}
\end{prop}
\begin{lem} \label{SobPoincare} (fractional Sobolev-Poincar\'e inequality)
Let $s \in (0,1)$, $p \in [1,\infty)$, $r>0$ and $x_0 \in \mathbb{R}^n$. In addition, let 
$$ q \in \begin{cases} 
\left [1,\frac{np}{n-sp} \right ], & \text{if } sp<n \\
[1,\infty), & \text{if } sp \geq n.
\end{cases}$$
Then for any $u \in W^{s,p}(B_r(x_0))$, we have
	$$ \left (\dashint_{B_r(x_0)} \left | u(x)- \overline u_{B_r(x_0)} \right |^q dx \right )^\frac{1}{q} \leq C r^{s} \left ( \dashint_{B_r(x_0)} \int_{B_r(x_0)} \frac{|u(x)-u(y)|^p}{|x-y|^{n+sp}}dydx \right )^\frac{1}{p},$$
	where $C=C(n,s,p,q)>0$.
\end{lem}
\begin{proof}
First, we consider the case when $r=1$ and $x_0 =0$, so that $u \in W^{s,p}(B_1)$.  By applying Proposition \ref{Sobemb} to $u- \overline u_{B_1} \in W^{s,p}(B_1)$ and then using Lemma \ref{Poincare}, we obtain
\begin{align*}
\left (\int_{B_1} \left | u(x)- \overline u_{B_1} \right |^q dx \right )^\frac{1}{q} \leq & C_1 \left (\int_{B_1} \left | u(x)- \overline u_{B_1}\right |^p dx + \int_{B_1} \int_{B_1} \frac{|u(x)-u(y)|^p}{|x-y|^{n+sp}}dydx \right )^\frac{1}{p} \\
\leq & C \left (\int_{B_1} \int_{B_1} \frac{|u(x)-u(y)|^p}{|x-y|^{n+sp}}dydx \right )^\frac{1}{p},
\end{align*}
where $C_1$ and $C$ depend only on $n,s,p$ and $q$. \par
In order to treat the general case, let $u \in W^{s,p}(B_r(x_0))$ and let $u_r(x):=u(rx+x_0)$. Then in view of changing variables and the previous case, we have
\begin{align*}
\left (\dashint_{B_r(x_0)} \left | u(x)- \overline u_{B_r(x_0)} \right |^q dx \right )^\frac{1}{q} = & \left (\int_{B_1} \left | u_r(x)- \overline {(u_r)}_{B_1} \right |^q dx \right )^\frac{1}{q} \\
\leq & C \left (\int_{B_1} \int_{B_1} \frac{|u_r(x)-u_r(y)|^p}{|x-y|^{n+sp}}dydx \right )^\frac{1}{p} \\
= & C r^{s} \left ( \dashint_{B_r(x_0)} \int_{B_r(x_0)} \frac{|u(x)-u(y)|^p}{|x-y|^{n+sp}}dydx \right )^\frac{1}{p},
\end{align*}
so that the proof is finished.
\end{proof}
We also use the following type of fractional Sobolev spaces.
\begin{defin}
	For $p \in (1,\infty)$ and $s \in (0,2)$, consider the Bessel potential space
	$$ H^{s,p}(\mathbb{R}^n):=\left \{u \in L^p(\mathbb{R}^n) \mid (-\Delta)^\frac{s}{2} u \in L^p(\mathbb{R}^n) \right \} ,$$
	where $$(-\Delta)^\frac{s}{2} u(x) = C_{n,s} \text{ } p.v. \int_{\mathbb{R}^n} \frac{u(x)-u(y)}{|x-y|^{n+s}}dy$$
	denotes the fractional Laplacian of $u$. We equip $H^{s,p}(\mathbb{R}^n)$ with the norm
	$$ ||u||_{H^{s,p}(\mathbb{R}^n)} := ||u||_{L^p(\mathbb{R}^n)}+||(-\Delta)^\frac{s}{2} u||_{L^p(\mathbb{R}^n)}.$$
	Moreover, for any domain $\Omega \subset \mathbb{R}^n$ we define 
	$$H^{s,p}(\Omega):= \left \{ u \big|_\Omega \mid u \in H^{s,p}(\mathbb{R}^n) \right \} $$
	with norm 
	$$ ||u||_{H^{s,p}(\Omega)} := \inf \left \{ ||v||_{H^{s,p}(\mathbb{R}^n)} \mid v \big |_\Omega = u \right \}$$
	and also the corresponding local Bessel potential spaces by 
	$$ H^{s,p}_{loc}(\Omega) := \left \{ u \in L^p_{loc}(\Omega) \mid u \in H^{s,p}(\Omega^\prime) \text{ for any domain } \Omega^\prime \Subset \Omega \right \}.$$   
\end{defin}
The following embedding result follows from \cite[Theorem 2.5]{Tr}, where it is given in the more general context of Besov and Triebel-Lizorkin spaces.
\begin{prop} \label{BesselTr}
	Let $1 < p_0 < p < p_1 < \infty$, $s \in (0,2)$, $s_0,s_1 \in (0,1)$ and assume that $\Omega \subset \mathbb{R}^n$ is a smooth domain. If $ s_0 - \frac{n}{p_0} = s - \frac{n}{p} = s_1 - \frac{n}{p_1}, $ then
		$$ W^{s_0,p_0}(\Omega) \hookrightarrow H^{s,p}(\Omega) \hookrightarrow W^{s_1,p_1}(\Omega).$$
\end{prop}
Unlike the classical Sobolev spaces $W^{1,p}(\Omega)$ on a bounded domain $\Omega \subset \mathbb{R}^n$, the fractional Sobolev spaces $W^{s,p}(\Omega)$ are not contained in each other as the integrability exponent $p$ decreases. Nevertheless, we have the following result, essentially stating that the mentioned inclusions are almost true.
\begin{prop} \label{Sobcont}
Let $1< p_0 \leq p<\infty$, $s \in (0,1)$ and assume that $\Omega \subset \mathbb{R}^n$ is a smooth bounded domain. Then for any $0<\varepsilon <\min  \left \{1-s,\frac{2n}{p},2n \left (1-\frac{1}{p_0} \right ) \right \}$, we have
$$ W^{s+\varepsilon,p}(\Omega) \hookrightarrow W^{s,p_0}(\Omega).$$
\end{prop}
\begin{proof}
By Proposition \ref{BesselTr}, we have $W^{s+\varepsilon,p}(\Omega) \hookrightarrow H^{s +\varepsilon/2,\widetilde p}(\Omega)$, where $\widetilde p:= \frac{np}{n-\varepsilon p/2}$. Now since for $\widehat p_0:=\frac{np_0}{n+\varepsilon p_0/2}$ we have $1<\widehat p_0< p_0 \leq p <\widetilde p$, by \cite[Theorem 2.85(ii)]{Tr} we have $H^{s +\varepsilon/2,\widetilde p}(\Omega) \hookrightarrow H^{s +\varepsilon/2,\widehat p_0}(\Omega)$. Since in addition by Proposition \ref{BesselTr} we have $H^{s +\varepsilon/2,\widehat p_0}(\Omega) \hookrightarrow W^{s,p_0}(\Omega)$, the proof is finished by combining the above three embeddings.
\end{proof}

\section{The measure $\mu$}
\subsection{Basic properties of $\mu$} \label{dm}
For the rest of this paper, we fix some $s \in (0,1)$ along with some parameter $\theta$ in the range
\begin{equation} \label{thetarange}
	0<\theta <\min \{s,1-s \}
\end{equation}
and let the measure $\mu$ be defined by (\ref{mudef}). Moreover, for any function $u:\mathbb{R}^n \to \mathbb{R}$ let the function $U$ be given by (\ref{Udef}). We now proof the relation (\ref{muequiv}).
\begin{lem} \label{Sobolevrelate}
Let $p \geq 2$ and set $s_\theta:=s+\theta \left (1-\frac{2}{p} \right )$. Then we have
$$ u \in W^{s_\theta,p}(\Omega) \quad \text{if and only if} \quad u \in L^p(\Omega) \text{ and } U \in L^p(\Omega \times \Omega,\mu)$$
and $$||U||_{L^p(\Omega \times \Omega,d\mu)}=[u]_{W^{s_\theta,p}(\Omega)}.$$
\end{lem}
\begin{proof}
We have
$$ \int_{\Omega \times \Omega} U^p d\mu = \int_{\Omega} \int_{\Omega} \frac{|u(x)-u(y)|^p}{|x-y|^{n+sp+(p-2)\theta}}dydx = \int_{\Omega} \int_{\Omega} \frac{|u(x)-u(y)|^p}{|x-y|^{n+s_\theta p}}dydx,$$
so that the claim follows.
\end{proof}

The next proposition contains some further important properties of the measure $\mu$ which we use frequently throughout the paper, usually without explicit reference.
\begin{prop} \label{doublingmeasure}
\begin{enumerate} [label=(\roman*)]
\item For all $r>0$ and $x_0 \in \mathbb{R}^n$, we have
$$ \mu(\mathcal{B}_r(x_0))= \mu(\mathcal{B}_r) =cr^{n+2\theta},$$
where $c=c(n,\theta)>0$.
\item (volume doubling property) For any $(x_0,y_0) \in \mathbb{R}^{2n}$, any $r>0$ and any $M >0$, we have
$$ \mu(\mathcal{B}_{Mr}(x_0,y_0)) = M^{n+2 \theta} \mu(\mathcal{B}_{r}(x_0,y_0)).$$
\end{enumerate}
\end{prop}

\begin{proof}
$(i)$ The changes of variables $x^\prime:=\frac{x}{r}-x_0$, $y^\prime:=\frac{y}{r}-y_0$ yield
$$ \mu(\mathcal{B}_r(x_0)) = \int_{B_r(x_0)}\int_{B_r(x_0)} \frac{dxdy}{|x-y|^{n-2\theta}} = \frac{r^{2n}}{r^{n-2 \theta}} \int_{B_1}\int_{B_1} \frac{dx^\prime dy^\prime}{|x^\prime-y^\prime|^{n-2\theta}} = cr^{n+2\theta},$$
where in view of integration in polar coordinates
\begin{align*}
c:= \int_{B_1}\int_{B_1} \frac{dx^\prime dy^\prime}{|x^\prime-y^\prime|^{n-2\theta}} &= \int_{B_1}\left (\int_{B_1(y^\prime)} \frac{dz}{|z|^{n-2\theta}}\right) dy^\prime \\
&\leq |B_1|\int_{B_2} \frac{dz}{|z|^{n-2\theta}}
= |B_1| \omega_n \int_0^2 \frac{\rho^{n-1}}{\rho^{n-2\theta}}d\rho = \frac{2^{2\theta}}{2\theta}|B_1|\omega_n < \infty.
\end{align*}
Here $\omega_n$ denotes the surface area of the $n-1$ dimensional unit sphere $S^{n-1}$. \par
$(ii)$ The changes of variables $x^\prime:=\frac{x}{M}$, $y^\prime:=\frac{y}{M}$ yield
\begin{align*}
\mu(\mathcal{B}_{Mr}(x_0,y_0)) = \int_{B_{Mr}(x_0)}\int_{B_{Mr}(y_0)} \frac{dxdy}{|x-y|^{n-2\theta}} & = M^{n+2\theta} \int_{B_r(x_0)}\int_{B_r(y_0)} \frac{dx^\prime dy^\prime}{|x^\prime-y^\prime|^{n-2\theta}}\\ & = M^{n+2\theta} \mu(\mathcal{B}_{r}(x_0,y_0)),
\end{align*}
which finishes the proof.
\end{proof}
\subsection{The Hardy-Littlewood maximal function} \label{HL}

Another tool we use is the Hardy-Littlewood maximal function with respect to the measure $\mu$.

\begin{defin}
	Let $F \in L^1_{loc}(\mathbb{R}^{2n},\mu)$. We define the Hardy-Littlewood maximal function \newline $\mathcal{M} F: \mathbb{R}^{2n} \to [0,\infty]$ of $F$ by 
	$$ \mathcal{M} (F)(x,y) := \sup_{\rho>0} \dashint_{\mathcal{B}_\rho(x,y)} |F|d\mu ,$$ where 
	$$ \dashint_{\mathcal{B}_\rho(x,y)} |F|d\mu := \frac{1}{\mu(\mathcal{B}_\rho(x,y))} \int_{\mathcal{B}_\rho(x,y)} |F|d\mu.$$
	Moreover, for any open set $E \subset \mathbb{R}^{2n}$, we define
	$$ \mathcal{M}_{E} (F) := \mathcal{M} \left(F \chi_{E} \right ), $$
	where $\chi_{E}$ is the characteristic function of $E$.
	In addition, for any $r>0$ we define $$ \mathcal{M}_{\geq r} (F)(x,y) := \sup_{\rho \geq r} \dashint_{\mathcal{B}_\rho(x,y)} |F|d\mu $$ and $$ \mathcal{M}_{\geq r,E} (F) := \mathcal{M}_{\geq r} \left(F \chi_{E} \right ). $$
\end{defin}

The following result shows that the Hardy-Littlewood maximal function behaves nicely in the context of $L^p$ spaces. Since by Proposition \ref{doublingmeasure} $\mu$ is a doubling measure with doubling constant $2^{n+2\theta}$, the result follows directly from \cite[Chapter 1, Section 3, Theorem 1]{Stein}.

\begin{prop} \label{Maxfun} 
	Let $E$ be an open subset of $\mathbb{R}^{2n}$.
	\begin{enumerate}[label=(\roman*)]
		\item (weak p-p estimates) If $F \in L^p(E,\mu)$ for some $p \geq 1$ and $\lambda>0$, then 
		$$
		\mu \left ( \{x \in E \mid \mathcal{M}_E(F)(x) > \lambda \} \right ) \leq \frac{C}{\lambda^p} \int_{E} |F|^p d\mu, 
		$$
		where $C$ depends only on $n,\theta$ and $p$.
		\item (strong p-p estimates) If $F \in L^p(E,\mu)$ for some $p \in (1,\infty]$, then 
		$$
		||\mathcal{M}_E (F)||_{L^p(E,d\mu)} \leq C ||F||_{L^p(E,d\mu)}, 
		$$
		where $C$ depends only on $n$, $\theta$ and $p$.
	\end{enumerate}
\end{prop}

For the following result we also refer to \cite[Chapter 1, Section 3]{Stein}.
\begin{prop} \label{LDT}
	\emph{(Lebesgue differentiation theorem)} \newline
	If $F \in L_{loc}^1(\mathbb{R}^{2n},\mu)$, then for almost every $(x,y) \in \mathbb{R}^{2n}$, we have 
	$$
	\lim_{r \to 0} \dashint_{\mathcal{B}_r(x,y)} Fd\mu=F(x,y).
	$$
\end{prop}
An immediate corollary of the Lebesgue differentiation theorem is given as follows.
\begin{cor} \label{maxdominate}
	Let $F \in L_{loc}^1(\mathbb{R}^{2n},\mu)$. Then for almost every $(x,y) \in \mathbb{R}^{2n}$, we have 
	$$
	|F(x,y)| \leq \mathcal{M}(F)(x,y). 
	$$
	In addition, for any open set $E \subset \mathbb{R}^{2n}$ and any $p \in [1,\infty]$, we have 
	$$
	||F||_{L^p(E,d\mu)} \leq ||\mathcal{M}_E (F)||_{L^p(E,d\mu)}. 
	$$
\end{cor}

\section{Some preliminary estimates} \label{pe}

\subsection{Higher H\"older regularity} \label{hr}
The following result on higher H\"older regularity plays an essential role in our approach and follows from \cite[Theorem 1.1]{MeH}.
\begin{thm} \label{HiHol}
	Let $\Omega \subset \mathbb{R}^n$ be a domain and let $f \in L^\infty_{loc}(\Omega)$. Consider a coefficient $A \in \mathcal{L}_0(\lambda)$ that is continuous in $\Omega \times \Omega$ and suppose that $\Phi$ satisfies (\ref{PhiLipschitz}) and (\ref{PhiMonotone}) with respect to $\lambda$. Moreover, assume that $u \in W^{s,2}(\mathbb{R}^n)$ is a weak solution of the equation $L_{A}^\Phi u =f $ in $\Omega$. Then for any $0<\alpha<\min \big \{2s,1 \big\}$, we have $u \in C^\alpha_{loc}(\Omega)$. \newline Furthermore, for all $R>0$, $x_0 \in \mathbb{R}^n$ such that $B_R(x_0) \Subset \Omega$ and any $\sigma \in (0,1)$, we have
	\begin{align*}
	[u]_{C^\alpha(B_{\sigma R}(x_0))} \leq \frac{C}{R^\alpha} & \bigg ( R^{-\frac{n}{2}} ||u||_{L^2(B_R(x_0))} + R^{2s} \int_{\mathbb{R}^n \setminus B_{R}(x_0)} \frac{|u(y)|}{|x_0-y|^{n+2s}}dy \\ & + R^{2s} ||f||_{L^\infty(B_R(x_0))} \bigg ),
	\end{align*}
	where $C=C(n,s,\lambda,\alpha,\sigma)>0$ and 
	$$[u]_{C^{\alpha}(B_{\sigma R}(x_0))}:=\sup_{\substack{_{x,y \in B_{\sigma  R}(x_0)}\\{x \neq y}}} \frac{|u(x)-u(y)|}{|x-y|^{\alpha}}.$$
\end{thm}

\begin{rem} \label{HiHolrem} \normalfont
In \cite{MeH}, Theorem \ref{HiHol} is proved in the more general context of so-called local weak solutions, see \cite[Section 1.1]{MeH} for a precise definition. From \cite[Lemma 3.5]{MeN} it follows that any local weak solution is a weak solution in our sense, so that Theorem \ref{HiHol} indeed follows from \cite[Theorem 1.1]{MeH}. Moreover, it is immediate from the definition of local weak solutions in \cite[Section 1.1]{MeH} that for any local weak solution $u$ of the equation $L_{A}^\Phi u =0 $ in $\Omega$ and any constant $c \in \mathbb{R}$, $u-c$ is also a local weak solution of the same equation. Therefore, in the setting of Theorem \ref{HiHol} for any weak solution $u \in W^{s,2}(\mathbb{R}^n)$ of $L_{A}^\Phi u =0 $ in $\Omega$ and any $c \in \mathbb{R}$, we have the estimate 
$$
[u]_{C^\alpha(B_{\sigma R}(x_0))} \leq \frac{C}{R^\alpha} \bigg ( R^{-\frac{n}{2}} ||u-c||_{L^2(B_R(x_0))} + R^{2s} \int_{\mathbb{R}^n \setminus B_{R}(x_0)} \frac{|u(y)-c|}{|x_0-y|^{n+2s}}dy \bigg ),
$$
where $C=C(n,s,\lambda,\alpha,\sigma)>0$.
\end{rem}

\subsection{Higher integrability of $U$ for small exponents} \label{kms}
For technical reasons, we also study equations with a more general right-hand side than in (\ref{nonlocaleq}).
\begin{defin}
	Let $2_\star:=\frac{2n}{n+2s}$. Given $f \in L^{2_\star}(\Omega)$ and $g \in W^{s,2}(\mathbb{R}^n)$, we say that $u \in W^{s,2}(\mathbb{R}^n)$ is a weak solution of the equation $L_A^\Phi u = (-\Delta)^s g + f$ in $\Omega$, if 
	\begin{align*}
	& \int_{\mathbb{R}^n} \int_{\mathbb{R}^n} \frac{A(x,y)}{|x-y|^{n+2s}} \Phi(u(x)-u(y))(\varphi(x)-\varphi(y))dydx \\
	= & C_{n,s} \int_{\mathbb{R}^n} \int_{\mathbb{R}^n} \frac{g(x)-g(y)}{|x-y|^{n+2s}} (\varphi(x)-\varphi(y))dydx + \int_{\Omega} f \varphi dx \quad \forall \varphi \in W^{s,2}_0(\Omega).
	\end{align*}
\end{defin}

Throughout this work, whenever we deal with functions $u$ and $g$ as in the above definition, for $(x,y) \in \mathbb{R}^{2n}$ with $x \neq y$ we define the functions
$$
U(x,y):=\frac{|u(x)-u(y)|}{|x-y|^{s+\theta}}, \quad G(x,y):=\frac{|g(x)-g(y)|}{|x-y|^{s+\theta}}.
$$
The following higher integrability result is essentially given by \cite[Theorem 6.1]{selfimpro}, where it is stated under the stronger assumptions that the equation holds on the whole space $\mathbb{R}^n$ and that $g$ is higher differentiable and integrable in the whole $\mathbb{R}^n$. Nevertheless, in \cite{selfimpro} the equation in only used in the proof of the Caccioppoli-type inequality \cite[Theorem 3.1]{selfimpro}, where the equation is tested with test functions that are supported in the ball where the estimate is proved. Therefore, it is enough to assume that the equation holds locally. Moreover, as indicated by the below estimate, in comparison with \cite{selfimpro} it is also sufficient to prescribe the higher differentiability and integrability on $g$ locally.

\begin{thm} \label{KMS}
	Let $r>0$, $x_0 \in \mathbb{R}^n$ and $\sigma_0>0$. Moreover, consider a coefficient
	$A \in \mathcal{L}_0(\Lambda)$ and assume that the Borel function $\Phi:\mathbb{R} \to \mathbb{R}$ satisfies
	\begin{equation} \label{weakassump}
		|\Phi(t)| \leq \Lambda t, \quad \Phi(t)t \geq \Lambda^{-1} t^2 \quad \forall t \in \mathbb{R}.
	\end{equation} In addition, assume that $u \in W^{s,2}(\mathbb{R}^n)$ is a weak solution of the equation $L_{A}^\Phi u = (-\Delta)^s g+f$ in $B_{2r}(x_0)$. Then there exist small enough positive constants $\gamma=\gamma(n,s,\Lambda,\sigma_0)\in \left (0,\frac{s}{2} \right )$ and $\sigma=\sigma(n,s,\Lambda,\sigma_0) \in (0,\sigma_0)$ such that if $f \in L^{2_\star+\sigma_0}(B_{2r}(x_0))$ and $g \in W^{s,2}(\mathbb{R}^n) \cap W^{s_\gamma,2+\sigma_0}(B_{2r}(x_0))$ for $s_\gamma:=s+\gamma \left (1-\frac{2}{2+\sigma_0} \right)$, then
	\begin{align*}
	\left (\dashint_{\mathcal{B}_r(x_0)} U^{2+\sigma}_\gamma d\mu_\gamma \right )^{\frac{1}{2+\sigma}} & \leq C \sum_{k=1}^\infty 2^{-k(s-\gamma)} \left ( \dashint_{\mathcal{B}_{2^kr}(x_0)} U_\gamma^2 d\mu_\gamma \right )^\frac{1}{2} \\
	& + C \left ( \dashint_{\mathcal{B}_{2r}(x_0)} G_\gamma^{2+\sigma_0} d\mu_\gamma \right )^\frac{1}{2+\sigma_0} + C \sum_{k=1}^\infty 2^{-k(s-\gamma)} \left ( \dashint_{\mathcal{B}_{2^kr}(x_0)} G_\gamma^2 d\mu_\gamma \right )^\frac{1}{2} \\
	& + C r^{s-\gamma} \left ( \dashint_{\mathcal{B}_{2r}(x_0)} F^{2_\star+\sigma_0} d\mu_\gamma \right )^\frac{1}{2_\star+\sigma_0},
	\end{align*}
where $C=C(n,s,\Lambda,\sigma_0)>0$. Here we denote
\begin{align*}
U_\gamma(x,y):=\frac{|u(x)-u(y)|}{|x-y|^{s+\gamma}}, \quad G_\gamma(x,y):=\frac{|g(x)-g(y)|}{|x-y|^{s+\gamma}}, \quad F(x,y):=f(x),
\end{align*}
while the measure $\mu_\gamma$ is defined on measurable sets $E \subset \mathbb{R}^{2n}$ by
$$
	\mu_\gamma(E):= \int_{E} \frac{dxdy}{|x-y|^{n-2\gamma}}.
$$
\end{thm}
We note that the assumptions in (\ref{weakassump}) are clearly implied by the assumptions $\Phi(0)=0$, (\ref{PhiLipschitz}) and (\ref{PhiMonotone}) which are used in our main results. \newline
Since working with the measure $\mu_\gamma$ and the functions $U_\gamma$ and $G_\gamma$ is inconvenient for us, our next goal is to rewrite the right-hand side of the estimate from Theorem \ref{KMS} in terms of the measure $\mu$ and the functions $U$ and $G$.
In order to accomplish this, for the rest of this paper in the above Theorem \ref{KMS} we let $\sigma_0>0$ to be chosen small enough and choose $\gamma>0$ small enough such that $\gamma \leq \theta$. By the latter inequality, we see that $s_\gamma \leq s_\theta:=s+\theta \left (1-\frac{2}{2+\sigma_0} \right)$. Together with Lemma \ref{Sobolevrelate}, we see that for the second term on the right-hand side of the estimate from Theorem \ref{KMS} we have
\begin{align*}
\left ( \dashint_{\mathcal{B}_{2r}(x_0)} G_\gamma^{2+\sigma_0} d\mu_\gamma \right )^{\frac{1}{2+\sigma_0}} & = \left ( \frac{C_1}{r^{n+2\gamma}} \int_{B_{2r}(x_0)}\int_{B_{2r}(x_0)} \frac{|g(x)-g(y)|^{2+\sigma_0}}{|x-y|^{n+(2+\sigma_0)s_\gamma}}dydx \right )^{\frac{1}{2+\sigma}} \\
& \leq \left ( (4r)^{(2+\sigma_0)(s_\theta-s_\gamma)} \frac{C_1}{r^{n+2\gamma}} \int_{B_{2r}(x_0)}\int_{B_{2r}(x_0)} \frac{|g(x)-g(y)|^{2+\sigma_0}}{|x-y|^{n+(2+\sigma_0)s_\theta}}dydx \right )^{\frac{1}{2+\sigma_0}} \\
& = C_2 \left ( r^{(2+\sigma_0)(\theta-\gamma) \left (1-\frac{2}{2+\sigma_0} \right)+2(\theta-\gamma)} \dashint_{\mathcal{B}_{2r}(x_0)} G^{2+\sigma_0} d\mu \right )^{\frac{1}{2+\sigma_0}} \\
& = C_2 r^{\theta-\gamma} \left ( \dashint_{\mathcal{B}_{2r}(x_0)} G^{2+\sigma_0} d\mu \right )^{\frac{1}{2+\sigma_0}},
\end{align*}
where $C_1$ and $C_2$ depend only on $n,s,\theta$ and $\sigma_0$.
Moreover, for the first and third term on the right-hand side of the estimate from Theorem \ref{KMS} we have 
$$
\sum_{k=1}^\infty 2^{-k(s-\gamma)} \left ( \dashint_{\mathcal{B}_{2^kr}(x_0)} U_\gamma^2 d\mu_\gamma \right )^\frac{1}{2} = C_3 r^{\theta-\gamma} \sum_{k=1}^\infty 2^{-k(s-\theta)} \left ( \dashint_{\mathcal{B}_{2^kr}(x_0)} U^2 d\mu \right )^\frac{1}{2},
$$
and 
$$
\sum_{k=1}^\infty 2^{-k(s-\gamma)} \left ( \dashint_{\mathcal{B}_{2^kr}(x_0)} G_\gamma^2 d\mu_\gamma \right )^\frac{1}{2} = C_3 r^{\theta-\gamma} \sum_{k=1}^\infty 2^{-k(s-\theta)} \left ( \dashint_{\mathcal{B}_{2^kr}(x_0)} G^2 d\mu \right )^\frac{1}{2},
$$
where $C_3=C_3(n,s,\gamma,\theta)$. Furthermore, in view of integration in polar coordinates for the last term on the right-hand side of the estimate in question we have 
\begin{align*}
r^{s-\gamma} \left ( \dashint_{\mathcal{B}_{2r}(x_0)} F^{2_\star+\sigma_0} d\mu_\gamma \right )^\frac{1}{2_\star+\sigma_0} & = r^{s-\gamma} \left ( \frac{C_1}{r^{n+2\gamma}} \int_{B_{2r}(x_0)}\int_{B_{2r}(x_0)} \frac{|f(x)|^{2_\star+\sigma_0}}{|x-y|^{n-2\gamma}}dydx \right )^\frac{1}{2_\star+\sigma_0} \\
& \leq r^{s-\gamma} \left ( \frac{C_1}{r^{n+2\gamma}} \int_{B_{2r}(x_0)}|f(x)|^{2_\star+\sigma_0} dx \int_{B_{4r}} \frac{dz}{|z|^{n-2\gamma}} \right )^\frac{1}{2_\star+\sigma_0} \\
& = C_4r^{s-\gamma} \left ( \dashint_{B_{2r}(x_0)}|f(x)|^{2_\star+\sigma_0} dx \right )^\frac{1}{2_\star+\sigma_0} ,
\end{align*}
where $C_4=C_4(n,s,\gamma,\theta)$.
Thus, we arrive at the estimate
\begin{align} \label{KMSmod}
\begin{split}
\left (\dashint_{\mathcal{B}_r(x_0)} U^{2+\sigma}_\gamma d\mu_\gamma \right )^{\frac{1}{2+\sigma}} & \leq C r^{\theta-\gamma} \sum_{k=1}^\infty 2^{-k(s-\theta)} \left ( \dashint_{\mathcal{B}_{2^kr}(x_0)} U^2 d\mu \right )^\frac{1}{2} \\
& + C r^{\theta-\gamma} \left ( \dashint_{\mathcal{B}_{2r}(x_0)} G^{2+\sigma_0} d\mu \right )^{\frac{1}{2+\sigma_0}} \\ & + C r^{\theta-\gamma} \sum_{k=1}^\infty 2^{-k(s-\theta)} \left ( \dashint_{\mathcal{B}_{2^kr}(x_0)} G^2 d\mu \right )^\frac{1}{2} \\
& + C r^{s-\gamma} \left ( \dashint_{B_{2r}(x_0)}f^{2_\star+\sigma_0} dx \right )^\frac{1}{2_\star+\sigma_0}
\end{split}
\end{align}
for a different $C$ as in Theorem \ref{KMS} depending only on $n,s,\Lambda,\theta,\gamma$ and $\sigma_0$.

\subsection{$H^{2s,p}$ estimates for the fractional Laplacian} \label{fl1}

As mentioned, for any regular enough function $u:\mathbb{R}^n \to \mathbb{R}$ and $s \in (0,1)$, the fractional Laplacian of $u$ is formally defined by
$$ (-\Delta)^s u(x) = C_{n,s}\text{ } p.v. \int_{\mathbb{R}^n} \frac{u(x)-u(y)}{|x-y|^{n+2s}}dy,$$
where $C_{n,s}$ is a constant depending on $n$ and $s$ whose exact value is not important for our purposes. In other words, $(-\Delta)^s$ corresponds to the operator $L_A^\Phi$ in the special case when $\Phi(t)=t$ and $A=C_{n,s}$. 
Since we want to study the Dirichlet problem associated to the fractional Laplacian, let us formally define weak solutions to such Dirichlet problems. For later use, we also include general nonlocal operators of the type (\ref{no}) in the definition.
\begin{defin}
	Let $\Omega$ be a domain and consider functions $h \in W^{s,2}(\mathbb{R}^n)$ and $f \in L^\frac{2n}{n+2s}(\Omega)$. We say that $v \in W^{s,2}(\mathbb{R}^n)$ is a weak solution of the Dirichlet problem $$ \begin{cases} \normalfont
		L_{A}^\Phi v = f & \text{ in } \Omega \\
		v = h & \text{ a.e. in } \mathbb{R}^n \setminus \Omega,
	\end{cases} $$ if we have $v = h \text{ a.e. in } \mathbb{R}^n \setminus \Omega$ and
	$$
	\int_{\mathbb{R}^n} \int_{\mathbb{R}^n} \frac{A(x,y)}{|x-y|^{n+2s}} \Phi(u(x)-u(y))(\varphi(x)-\varphi(y))dydx = \int_{\Omega} f \varphi dx \quad \forall \varphi \in W^{s,2}_0(\Omega).
	$$
\end{defin}

For the following local regularity result we follow an approach applied for instance in \cite{Warma} or \cite{KassMengScott}. The main idea is to multiply the solution with an appropriate cutoff function in order to reduce the problem to a corresponding one which is posed on the whole space $\mathbb{R}^n$, for which the desired estimate can be inferred by classical techniques from Fourier analysis, see \cite[Lemma 3.5]{KassMengScott}.

\begin{thm} \label{H2spest}
	Let $\Omega \subset \mathbb{R}^n$ be a bounded domain, $s \in (0,1)$ and $p \in \left (\frac{2n}{n+2s},\infty \right )$. If $f \in L^p(\Omega) \cap L^2(\Omega)$, then the unique weak solution $u \in W^{s,2}(\mathbb{R}^n)$ of the Dirichlet problem
	\begin{equation} \label{fl}
		\begin{cases} \normalfont
			(-\Delta)^s u = f & \text{ in } \Omega \\
			u=0 & \text{ a.e. in } \mathbb{R}^n \setminus \Omega
		\end{cases}
	\end{equation}
	belongs to $H^{2s,p}_{loc}(\Omega)$. Moreover, for any open set $\Omega^\prime \Subset \Omega$, we have the estimate
	\begin{equation} \label{flest}
		||u||_{H^{2s,p}(\Omega^\prime)} \leq C ||f||_{L^p(\Omega)},
	\end{equation}
	where $C=C(n,s,p,\Omega^\prime,\Omega)>0$.
\end{thm}

\begin{proof}
	First of all, the existence of a unique weak solution $u$ to (\ref{fl}) can either be ensured by using the Riesz representation theorem (see e.g.\ \cite[Proposition 5.1]{Me}), by variational methods (see \cite[Theorem 2.3]{finnish}) or by using the theory of monotone operators (see \cite[Proposition 4.1]{MeN}). From \cite[Lemma 2.5]{Warma}, it follows that we have the estimate
	\begin{equation} \label{Lpreg}
		||u||_{L^p(\Omega)} \leq C_1 ||f||_{L^p(\Omega)},
	\end{equation}
	where $C_1=C_1(n,s,p)>0$. Now in view of \cite[Lemma 3.5]{KassMengScott}, for any function $\widetilde f \in L^p(\mathbb{R}^n) \cap L^2(\mathbb{R}^n)$ and any weak solution $\widetilde u \in W^{s,2}(\mathbb{R}^n)$ of $(-\Delta)^s \widetilde u = \widetilde f$ in $\mathbb{R}^n$, we have the estimate
	\begin{equation} \label{flreg}
		||(-\Delta)^s \widetilde u||_{L^p(\mathbb{R}^n)} \leq C_2 ||\widetilde f||_{L^p(\mathbb{R}^n)},
	\end{equation}
	where $C_2=C_2(n,s,p)>0$. \par
	Now fix an open set $\Omega^\prime \Subset \Omega$ and smooth domains $\Omega_0^\star, \Omega_0$ and $\widetilde \Omega_0$ such that $$ \Omega^\prime \Subset \Omega_0^\star \Subset \Omega_0 \Subset \widetilde \Omega_0 \Subset \Omega.$$ In addition, fix a smooth cutoff function $\eta \in C_0^{\infty}(\Omega_0^\star)$ with the properties 
	$$ 0 \leq \eta \leq 1 \text{ in } \Omega_0^\star, \quad \eta = 1 \text{ in } \Omega^\prime, \quad \eta=0 \text{ in } \mathbb{R}^n \setminus \Omega_0^\star.$$ Note that we have
	\begin{equation} \label{lcut}
		(-\Delta)^s (u \eta) = \eta f + u (-\Delta)^s \eta - I_s(u,\eta) \text{ weakly in } \mathbb{R}^n,
	\end{equation}
	where $f$ is extended to $\mathbb{R}^n$ by setting $f=0$ in $\mathbb{R}^n \setminus \Omega$ and $$ I_s(u,\eta)(x):= C_{n,s} \int_{\mathbb{R}^n} \frac{(u(x)-u(y))(\eta(x)-\eta(y))}{|x-y|^{n+2s}}dy, \quad x \in \mathbb{R}^n.$$
	Set $g:= u (-\Delta)^s \eta - I_s(u,\eta).$
	We start by proving that for any $q \in \left (\frac{2n}{n+2s},\infty \right )$, there exists a constant $C_q=C_q(n,s,q,\Omega^\prime,\Omega_0^\star,\Omega_0,\widetilde \Omega_0,\Omega)$ such that
	\begin{equation} \label{gest}
		||g||_{L^q(\mathbb{R}^n)} \leq C_q ( [u]_{W^{s,q}(\widetilde \Omega_0)} + ||u||_{L^{q}(\Omega)}).
	\end{equation}
	First of all, in view of changing variables, for any $x \in \mathbb{R}^n$ we have 
	\begin{align*}
		& |(-\Delta)^s \eta(x)| \\ = & C_{n,s} \left | \int_{\mathbb{R}^n} \frac{\eta(x)-\eta(y)}{|x-y|^{n+2s}} dy \right | \\
		=& \frac{C_{n,s}}{2} \left | \int_{\mathbb{R}^n} \frac{\eta(x+z)+\eta(x-z)-2\eta(x)}{|z|^{n+2s}}dz \right | \\
		\leq& C_{n,s} \left (\int_{B_1} \frac{|\eta(x+z)+\eta(x-z)-2\eta(x)|}{|z|^{n+2s}}dz + \int_{\mathbb{R}^n \setminus B_1} \frac{|\eta(x+z)+\eta(x-z)-2\eta(x)|}{|z|^{n+2s}}dz \right ) \\
		\leq& C_{n,s} \left (||\nabla^2 \eta||_{L^\infty(\mathbb{R}^n,\mathbb{R}^{n \times n})} \int_{B_1} \frac{1}{|z|^{n+2s-2}}dz + 4 ||\eta||_{L^\infty(\mathbb{R}^n)} \int_{\mathbb{R}^n \setminus B_1} \frac{dz}{|z|^{n+2s}} \right ) =:C_3.
	\end{align*}
	Using integration in polar coordinates, we obtain
	$$ \int_{B_1} \frac{1}{|z|^{n+2s-2}}dz = \omega_n \int_0^1 \frac{\rho^{n-1}}{\rho^{n+2s-2}}d\rho = \frac{\omega_n}{2-2s}<\infty$$
	and 
	$$ \int_{\mathbb{R}^n \setminus B_1} \frac{1}{|z|^{n+2s}}dz = \omega_n \int_1^\infty \frac{\rho^{n-1}}{\rho^{n+2s}}d\rho = \frac{\omega_n}{2s}<\infty$$
	where $\omega_n$ denotes the surface area of the $n-1$ dimensional unit sphere $S^{n-1}$. Therefore, $C_3$ is finite.
	Using that $u=0$ a.e. in $\mathbb{R}^n \setminus \Omega$ and H\"older's inequality, we deduce
	$$ ||u L_A \eta||_{L^q(\mathbb{R}^n)} = ||u L_A \eta||_{L^q(\Omega)} \leq ||L_A \eta||_{L^\infty(\Omega)} ||u||_{L^q(\Omega)} \leq C_1 ||u||_{L^q(\Omega)}.$$
	Note also that 
	\begin{align*}
		I_s(u,\eta)(x) = & \int_{\Omega_0}\frac{(u(x)-u(y))(\eta(x)-\eta(y))}{|x-y|^{n+2s}}  dy \\
		& + \eta(x) \int_{\mathbb{R}^n \setminus \Omega_0}\frac{u(x)-u(y)}{|x-y|^{n+2s}} dy = I_1(x)+I_2(x).
	\end{align*}
	Setting $q^\prime :=\frac{q}{q-1}$, by H\"olders inequality we have
	\begin{equation} \label{11}
		|I_1(x)| \leq \left ( \int_{\Omega_0} \frac{|u(x)-u(y)|^{q}}{|x-y|^{n+sq}}dy \right )^{\frac{1}{q}} \left ( \int_{\Omega_0} \frac{|\eta(x)-\eta(y)|^{q^\prime}}{|x-y|^{n+q^\prime s}}dy \right )^{\frac{1}{q^\prime}}.
	\end{equation}
	Since $\widetilde \Omega_0$ is bounded, we have $R:=\textnormal{diam}(\widetilde \Omega_0)<\infty$. In addition, we have $r:=\textnormal{dist}(\partial \Omega_0,\partial \widetilde \Omega_0)>0$.
	Since furthermore $\eta$ is Lipschitz continuous with some Lipschitz constant $L>0$, we have
	\begin{equation} \label{12}
		\begin{aligned}
		& \sup_{ x \in \mathbb{R}^n } \int_{\Omega_0} \frac{|\eta(x)-\eta(y)|^{q^\prime}}{|x-y|^{n+q^\prime s}}dy \\ \leq & \sup_{ x \in \widetilde \Omega_0} \int_{\Omega_0} \frac{|\eta(x)-\eta(y)|^{q^\prime}}{|x-y|^{n+q^\prime s}}dy 
		+ \sup_{ x \in \mathbb{R}^n \setminus \widetilde \Omega_0} \int_{\Omega_0} \frac{\eta(y)^{q^\prime}}{|x-y|^{n+q^\prime s}}dy \\ \leq & L^{q^\prime} \sup_{ x \in \widetilde \Omega_0} \int_{B_R(x)} \frac{1}{|x-y|^{n+q^\prime s- q^\prime}}dy + ||\eta||_{L^\infty(\mathbb{R}^n)}^{q^\prime} \sup_{ x \in \mathbb{R}^n \setminus \widetilde \Omega_0} \int_{\Omega_0} \frac{1}{|x-y|^{n+q^\prime s}}dy \\ \leq & L^{q^\prime} \int_{B_R} \frac{1}{|z|^{n+q^\prime s- q^\prime}}dz + \int_{\mathbb{R}^n \setminus B_{r}} \frac{1}{|z|^{n+q^\prime s}}dz \\
		= & L^{q^\prime} \frac{\omega_n}{q^\prime(1-s)} R^{q^\prime(1-s)} + \frac{\omega_n}{q^\prime s} r^{-q^\prime s} =:C_4< \infty.
		\end{aligned}
	\end{equation}
	By (\ref{11}), (\ref{12}), using the elementary inequality \cite[Formula (A.10)]{Warma} and then the assumption that $u=0$ a.e. in $\mathbb{R}^n \setminus \Omega$, we obtain 
	\begin{align*}
		||I_1||_{L^q(\mathbb{R}^n)}^q \leq & C_4^\frac{1}{q^\prime} \left ( \int_{\widetilde \Omega_0} \int_{\Omega_0} \frac{|u(x)-u(y)|^{q}}{|x-y|^{n+sq}}dydx +  \int_{\mathbb{R}^n \setminus \widetilde \Omega_0} \int_{\Omega_0} \frac{|u(x)-u(y)|^{q}}{|x-y|^{n+sq}}dydx \right ) \\
		\leq & C_5 \Bigg (\int_{\widetilde \Omega_0} \int_{\Omega_0} \frac{|u(x)-u(y)|^{q}}{|x-y|^{n+sq}}dydx + \int_{\mathbb{R}^n \setminus \widetilde \Omega_0} \int_{\Omega_0} \frac{|u(x)|^{q}+|u(y)|^{q}}{(1+|x|)^{n+sq}} dydx \Bigg ) \\
		\leq & C_6 \left ( [u]_{W^{s,q}(\widetilde \Omega_0)}^q + ||u||_{L^{q}(\Omega)}^q \right ),
	\end{align*}
	where $C_5$ and $C_6$ depend only on $\Omega, \Omega_0,\widetilde \Omega_0,n,s$ and $q$. Let us now estimate the term $I_2$. Clearly, we have $I_2=0$ in $\mathbb{R}^n \setminus \Omega_0^\star$. By H\"older's inequality, for any $x \in \Omega_0^\star$ we have 
	$$
	|I_2(x)|^q \leq \eta^{q}(x) \left ( \int_{\mathbb{R}^n \setminus \Omega_0} \frac{1}{|x-y|^{n+sq^\prime}} dy \right )^{q-1} \left ( \int_{\mathbb{R}^n \setminus \Omega_0} \frac{|u(x)-u(y)|^{q}}{|x-y|^{n+sq}}dy\right ).
	$$
	Since $\textnormal{dist}(\partial \Omega_0^\star,\partial \Omega_0)>0$, observe that in view of \cite[Formula (A.10)]{Warma}, there exists a constant $C_7=C_7(\Omega_0^\star,\Omega_0,n,s,q)$ such that for any $x \in \Omega_0^\star$ we have
	$$
	\eta^{q}(x) \left ( \int_{\mathbb{R}^n \setminus \Omega_0} \frac{1}{|x-y|^{n+q^\prime s}} dy \right )^{q-1} \leq \left ( C_7 \int_{\mathbb{R}^n \setminus \Omega_0} \frac{1}{(1+|y|)^{n+sq^\prime}} dy \right )^{q-1} =:C_8< \infty,
	$$
	where we also used that $\eta^{q}(x) \in [0,1]$.
	Thus, we conclude that
	\begin{align*}
		||I_2||_{L^q(\mathbb{R}^n)}^q & = ||I_2||_{L^q(\Omega_0^\star)}^q \leq C_8 \int_{\Omega_0^\star} \int_{\mathbb{R}^n \setminus \Omega_0} \frac{|u(x)-u(y)|^{q}}{|x-y|^{n+sq}}dydx \\
		& \leq C_9 \int_{\Omega_0^\star} \int_{\mathbb{R}^n \setminus \Omega_0} \frac{|u(x)|^{q}+|u(y)|^{q}}{(1+|y|)^{n+sq}}dydx \leq C_{10} ||u||_{L^{q}(\Omega)}^q,
	\end{align*}
	where $C_9$ and $C_{10}$ depend only on $\Omega_0^\star,\Omega_0,n,s$ and $q$.
	By combining the above calculations, we obtain that (\ref{gest}) holds. \par
	In addition, in view of using $u \in W_0^{s,2}(\Omega)$ itself as a test function, H\"older's inequality and the fractional Sobolev inequality (see \cite[Theorem 6.5]{Hitch}), for any $p \in \left (\frac{2n}{n+2s},\infty \right )$ we have
	\begin{align*}
		& C_{n,s} [u]_{W^{s,2}(\mathbb{R}^n)}^2 = \int_{\Omega} fudx \leq ||f||_{L^\frac{2n}{n+2s}(\Omega)} ||u||_{L^\frac{2n}{n-2s}(\mathbb{R}^n)} \leq C_{11} ||f||_{L^{p}(\Omega)} [u]_{W^{s,2}(\mathbb{R}^n)},
	\end{align*}
	where $C_{11}=C_{11}(n,s,p,\Omega)>0$. Thus, we obtain that
	\begin{equation} \label{basic}
		[u]_{W^{s,2}(\mathbb{R}^n)} \leq C_{12} ||f||_{L^{p}(\Omega)},
	\end{equation}
	where $C_{12}=C_{12}(n,s,p,\Omega)>0$.
	Now we consider the case when $p \in \big (\frac{2n}{n+2s}, 2 \big ]$.
	Since $p>\frac{2n}{n+2s}$, by combining the estimate from Theorem \ref{KMS} with Proposition \ref{Sobolevrelate} and a standard covering argument (see our proof of Theorem \ref{mainint5} below and in particular (\ref{cest})), we see that there exists some small $\varepsilon_1=\varepsilon_1(n,s,p)>0$ and some small $\varepsilon_2=\varepsilon_2(n,s,p)>0$ such that
	\begin{equation} \label{kmsflap}
	[u]_{W^{s+\varepsilon_1,2+\varepsilon_2}(\widetilde \Omega_0)} \leq C_{13} ([u]_{W^{s,2}(\mathbb{R}^n)} + ||f||_{L^{p}(\Omega)})
	\leq C_{14} ||f||_{L^{p}(\Omega)},
	\end{equation}
	where $C_{13}$ and $C_{14}$ depend only on $n,s,p,\widetilde \Omega_0$ and $\Omega$ and we used (\ref{basic}) in order to obtain the last inequality. Since we consider the case when $p \in \big (\frac{2n}{n+2s}, 2 \big ]$, by combining (\ref{kmsflap}) with Proposition \ref{Sobcont}, obtain
	\begin{equation} \label{sh}
	[u]_{W^{s,p}(\widetilde \Omega_0)} \leq C_{15} [u]_{W^{s+\varepsilon_1,2+\varepsilon_2}(\widetilde \Omega_0)}
	\leq C_{16} ||f||_{L^{p}(\Omega)},
	\end{equation}
	where again $C_{15}$ and $C_{16}$ depend only on $n,s,p,\widetilde \Omega_0$ and $\Omega$.
	Now in view of (\ref{lcut}), (\ref{flreg}), (\ref{gest}) with $q=p$, (\ref{Lpreg}) and (\ref{sh}), we arrive at
	\begin{align*}
		||(-\Delta)^s (\eta u)||_{L^p(\mathbb{R}^n)} \leq C_{17}( ||\eta f||_{L^p(\mathbb{R}^n)} + ||g||_{L^p(\mathbb{R}^n)}) \leq & C_{18} (||f||_{L^p(\Omega)} + [u]_{W^{s,p}(\widetilde \Omega_0)} + ||u||_{L^{p}(\Omega)}) \\
		\leq & C_{19} ||f||_{L^p(\Omega)},
	\end{align*}
	where all constants depend only on $n,s,q,\Omega^\prime,\Omega_0^\star,\Omega_0,\widetilde \Omega_0,\Omega$. Note that the same is true for all other constants in the further proof. Together with (\ref{Lpreg}) and the fact that $\eta=1$ in $\Omega^\prime$, in this case we conclude that
	\begin{align*}
		||u||_{H^{2s,p}(\Omega^\prime)} \leq ||\eta u||_{H^{2s,p}(\mathbb{R}^n)} \leq ||u||_{L^{p}(\Omega)} + ||(-\Delta)^s (\eta u)||_{L^p(\mathbb{R}^n)} \leq C_{20} ||f||_{L^p(\Omega)}.
	\end{align*}
	Since $\Omega^\prime$ is an arbitrary relatively compact open subset of $\Omega$, this proves the estimate (\ref{H2spest}) in the case when $p \in \big (\frac{2n}{n+2s}, 2 \big ]$. \par
	Next, we consider the case when $p \in (2,\infty)$. For any $q \in (1,p]$, we define
	$$ q^\star := \begin{cases} 
		\min \{\frac{nq}{n-sq},p\}, & \text{if } sq < n \\
		p, & \text{if } sq \geq n .
	\end{cases}$$
	In view of (\ref{lcut}), (\ref{flreg}) and (\ref{gest}) with $q=2^\star$, Proposition \ref{BesselTr}, (\ref{Lpreg}) and using the case $p=2$ which was treated above, we obtain
	\begin{align*}
		||(-\Delta)^s (\eta u)||_{L^{2\star}(\mathbb{R}^n)} \leq C_{21}( ||\eta f||_{L^{2^\star}(\mathbb{R}^n)} + ||g||_{L^{2^\star}(\mathbb{R}^n)}) \leq & C_{22} (||f||_{L^{2^\star}(\Omega)} + [u]_{W^{s,2^\star}(\widetilde \Omega_0)} + ||u||_{L^{2^\star}(\Omega)}) \\
		\leq & C_{23} (||f||_{L^{2^\star}(\Omega)} + ||u||_{H^{2s,2}(\widetilde \Omega_0)}) \\
		\leq & C_{24} ||f||_{L^{2^\star}(\Omega)},
	\end{align*}
	so that together with (\ref{Lpreg}), we obtain
	\begin{align*}
		||u||_{H^{2s,2^\star}(\Omega^\prime)} \leq ||\eta u||_{H^{2s,{2^\star}}(\mathbb{R}^n)} \leq ||u||_{L^{2^\star}(\Omega)} + ||(-\Delta)^s (\eta u)||_{L^{2^\star}(\mathbb{R}^n)} \leq C_{25} ||f||_{L^{2^\star}(\Omega)}.
	\end{align*}
	If $2^\star=p$, the proof is finished. Otherwise, by a similar reasoning we obtain the estimate
	\begin{align*}
		||u||_{H^{2s,{2^\star}^\star}(\Omega^\prime)} \leq C_{26} ||f||_{L^{{2^\star}^\star}(\Omega)}.
	\end{align*}
	If ${2^\star}^\star=p$, the proof is finished. Otherwise, by iterating the above procedure we also arrive at the estimate (\ref{H2spest}) after finitely many iterations, so that the proof is finished.
\end{proof}
For some more local and global regularity results for the fractional Laplacian, we refer to \cite{FracLap} and \cite{Grubb}.

\section{Comparison estimates} \label{ce}
The following Lemma relates the nonlocal tail of a function $u$ to the corresponding function $U$.
\begin{lem} \label{anullusdecomp}
	Let $R>0$ and $x_0 \in \mathbb{R}^n$. For any function $u \in W^{s,2}(\mathbb{R}^n)$, we have
	\begin{equation} \label{andc}
	\int_{\mathbb{R}^n \setminus B_{R}(x_0)} \frac{|u(y)-\overline{u}_{B_R(x_0)}|}{|x_0-y|^{n+2s}}dy \leq C R^{-s+\theta} \sum_{k=1}^\infty 2^{-k(s-\theta)} \left ( \dashint_{\mathcal{B}_{2^kR}(x_0)} U^2 d\mu \right )^\frac{1}{2},
	\end{equation}
	where $C=C(n,s,\theta)>0$.
\end{lem}

\begin{proof}
	First of all, splitting the integral on the left-hand side into annuli yields
	\begin{align*}
	\int_{\mathbb{R}^n \setminus B_{R}(x_0)} \frac{|u(y)-\overline{u}_{B_R(x_0)}|}{|x_0-y|^{n+2s}}dy & = \sum_{j=0}^\infty \int_{B_{2^{j+1}R}(x_0) \setminus B_{2^jR}(x_0)}\frac{|u(y)-\overline{u}_{B_R(x_0)}|}{|x_0-y|^{n+2s}}dy \\
	& \leq \sum_{j=0}^\infty (2^j R)^{-n-2s} \int_{B_{2^{j+1}R}(x_0)} |u(y)-\overline{u}_{B_R(x_0)}|dy \\
	& = C_1 \sum_{j=0}^\infty (2^j R)^{-2s} \dashint_{B_{2^{j+1}R}(x_0)} |u(y)-\overline{u}_{B_R(x_0)}|dy ,
	\end{align*}
	where $C_1=C_1(n)$.
	Using the Cauchy-Schwarz inequality, we deduce
	\begin{align*}
	& \dashint_{B_{2^{j+1}R}(x_0)} |u(y)-\overline{u}_{B_R(x_0)}|dy \\
	\leq & \dashint_{B_{2^{j+1}R}(x_0)} |u(y)-\overline{u}_{B_{2^{j+1} R}(x_0)}|dy + \sum_{k=0}^j |\overline{u}_{B_{2^{k+1} R}(x_0)}-\overline{u}_{B_{2^{k} R}(x_0)}| \\
	\leq & \dashint_{B_{2^{j+1}R}(x_0)} |u(y)-\overline{u}_{B_{2^{j+1} R}(x_0)}|dy + \sum_{k=0}^j \dashint_{B_{2^{k}R}(x_0)} |u(y)-\overline{u}_{B_{2^{k+1} R}(x_0)}|dy \\
	\leq & \dashint_{B_{2^{j+1}R}(x_0)} |u(y)-\overline{u}_{B_{2^{j+1} R}(x_0)}|dy + 2^n \sum_{k=0}^j \dashint_{B_{2^{k+1}R}(x_0)} |u(y)-\overline{u}_{B_{2^{k+1} R}(x_0)}|dy \\
	\leq & 2^{n+1} \sum_{k=1}^{j+1} \dashint_{B_{2^{k}R}(x_0)} |u(y)-\overline{u}_{B_{2^{k} R}(x_0)}|dy \\
	\leq & 2^{n+1} \sum_{k=1}^{j+1} \bigg( \dashint_{B_{2^{k}R}(x_0)} |u(y)-\overline{u}_{B_{2^{k} R}(x_0)}|^{2} dy \bigg)^\frac{1}{2}.
	\end{align*} 
	In order to further estimate the right-hand side of the previous computation, we use the fractional Poincar\'e inequality (Lemma \ref{Poincare}) in order to obtain
	\begin{align*}
	\dashint_{B_{2^{k}R}(x_0)} |u(y)-\overline{u}_{B_{2^{k} R}(x_0)}|^{2}dx & \leq C_2 (2^k R)^{2s} \int_{B_{2^{k}R}(x_0)}\dashint_{B_{2^{k}R}(x_0)} \frac{|u(x)-u(y)|^2}{|x-y|^{n+2s}}dydx \\
	& = C_3 (2^k R)^{2(s+\theta)} \dashint_{\mathcal{B}_{2^{k}R}(x_0)} U^2 d\mu,
	\end{align*}
	where $C_2$ and $C_3$ depend only on $n,s$ and $\theta$.
	Combining the previous two calculations leads to
	$$ \dashint_{B_{2^{j+1}R}(x_0)} |u(y)-\overline{u}_{B_R(x_0)}|dy \leq 
	C_4 \sum_{k=1}^{j+1} (2^k R)^{(s+\theta)} \bigg (\dashint_{\mathcal{B}_{2^{k}R}(x_0)} U^2 d\mu \bigg )^\frac{1}{2}, $$
	where $C_4=C_4(n,s,\theta)>0$. Next, combining the previous estimate with the first display in this proof yields
	\begin{equation} \label{ds}
	\int_{\mathbb{R}^n \setminus B_{R}(x_0)} \frac{|u(y)-\overline{u}_{B_R(x_0)}|}{|x_0-y|^{n+2s}}dy
	\leq C_1 C_4 R^{-s+\theta} \sum_{j=0}^\infty \sum_{k=1}^{j+1} 2^{-2sj} 2^{k(s+\theta)} \bigg (\dashint_{\mathcal{B}_{2^{k}R}(x_0)} U^2 d\mu \bigg )^\frac{1}{2}.
	\end{equation}
	By reverting the order of summation, we obtain
	\begin{align*}
	& \sum_{j=0}^\infty \sum_{k=1}^{j+1} 2^{-2sj} 2^{k(s+\theta)} \bigg (\dashint_{\mathcal{B}_{2^{k}R}(x_0)} U^2 d\mu \bigg )^\frac{1}{2} \\
	= & \sum_{k=1}^\infty 2^{k(s+\theta)} \bigg (\dashint_{\mathcal{B}_{2^{k}R}(x_0)} U^2 d\mu \bigg )^\frac{1}{2} \sum_{j=k-1}^\infty 2^{-2sj} \\
	= & \sum_{k=1}^\infty 2^{k(s+\theta)} \bigg (\dashint_{\mathcal{B}_{2^{k}R}(x_0)} U^2 d\mu \bigg )^\frac{1}{2} \sum_{j=1}^\infty 2^{-2s(j+k-2)} \\
	\leq & 4^{2s} \sum_{k=1}^\infty 2^{-k(s-\theta)} \bigg (\dashint_{\mathcal{B}_{2^{k}R}(x_0)} U^2 d\mu \bigg )^\frac{1}{2} \sum_{j=1}^\infty 2^{-2sj}.
	\end{align*}
	Since the sum $\sum_{j=1}^\infty 2^{-2sj}$ is finite, we conclude that
	\begin{align*}
	& \sum_{j=0}^\infty \sum_{k=1}^{j+1} 2^{-2sj} 2^{k(s+\theta)} \bigg (\dashint_{\mathcal{B}_{2^{k}R}(x_0)} U^2 d\mu \bigg )^\frac{1}{2} \\
	\leq & C_5 \sum_{k=1}^\infty 2^{-k(s-\theta)} \bigg (\dashint_{\mathcal{B}_{2^{k}R}(x_0)} U^2 d\mu \bigg )^\frac{1}{2},
	\end{align*}
	where $C_5=C_5(s)>0$. Finally, by combining the last display with (\ref{ds}), we arrive at (\ref{andc}).
\end{proof}

A crucial tool for the proof of the higher integrability of $U$ is given by the following comparison estimate. Essentially, it will allow us to transfer some regularity from a solution of a more well-behaved equation to the solution of the original equation.

\begin{prop} \label{appplxy}
Let $x_0 \in \mathbb{R}^n$, $r>0$, $g \in W^{s,2}(\mathbb{R}^n) \cap W^{s_\theta,2+\sigma_0}(B_{2r}(x_0))$, $f \in L^{2_\star+\sigma_0}(B_{2r}(x_0))$,  $\widetilde f \in L^{2_\star}(B_{2r}(x_0))$ and $A \in \mathcal{L}_0(\Lambda)$. In addition, assume that $\Phi$ satisfies the conditions (\ref{PhiLipschitz}) and (\ref{PhiMonotone}). Moreover, let $u \in W^{s,2}(\mathbb{R}^n)$ be a weak solution of the equation
\begin{equation} \label{neq5}
L_A^\Phi u = (-\Delta)^s g + f \text{ in } B_{2r}(x_0),
\end{equation}
and let $v \in W^{s,2}(\mathbb{R}^n)$ be the unique weak solution of the Dirichlet problem
\begin{equation} \label{constcof3}
\begin{cases} \normalfont
L_{\widetilde A}^\Phi v = \widetilde f & \text{ in } B_{2r}(x_0) \\
v = u & \text{ a.e. in } \mathbb{R}^n \setminus B_{2r}(x_0),
\end{cases}
\end{equation}
where $\widetilde A$ is another coefficient of class $\mathcal{L}_0(\Lambda)$ such that $\widetilde A=A$ in $\mathbb{R}^{2n} \setminus \mathcal{B}_{r}(x_0)$.
Then the function $w:=u-v \in W^{s,2}_0(B_{2r}(x_0))$
satisfies 
\begin{align*}
& \int_{\mathbb{R}^n} \int_{\mathbb{R}^n} \frac{(w(x)-w(y))^2}{|x-y|^{n+2s}}dydx \\
\leq & C \omega(A-\widetilde A,r,x_0)^\frac{\gamma}{n-\gamma} \mu(\mathcal{B}_r(x_0)) \left (\sum_{k=1}^\infty 2^{-k(s-\theta)} \left ( \dashint_{\mathcal{B}_{2^kr}(x_0)} U^2 d\mu \right )^\frac{1}{2} \right )^2 \\	& + C \omega(A-\widetilde A,r,x_0)^\frac{\gamma}{n-\gamma} \mu(\mathcal{B}_r(x_0)) \left ( \dashint_{\mathcal{B}_{2r}(x_0)} G^{2+\sigma_0} d\mu \right )^{\frac{2}{2+\sigma_0}} \\ & + C (\omega(A-\widetilde A,r,x_0)^\frac{\gamma}{n-\gamma}+1) \mu(\mathcal{B}_r(x_0)) \left (\sum_{k=1}^\infty 2^{-k(s-\theta)} \left ( \dashint_{\mathcal{B}_{2^kr}(x_0)} G^2 d\mu \right )^\frac{1}{2}\right )^2 \\ & + C\omega(A-\widetilde A,r,x_0)^\frac{\gamma}{n-\gamma} r^{2(s-\theta)}\mu(\mathcal{B}_r(x_0)) \left ( \dashint_{B_{2r}(x_0)}|f|^{2_\star+\sigma_0} dx \right )^\frac{2}{2_\star+\sigma_0} \\
& + C r^{2(s-\theta)}\mu(\mathcal{B}_r(x_0)) \left (\dashint_{B_{2r}(x_0)}|f-\widetilde f|^{2_\star} dx \right )^\frac{2}{2_\star} ,
\end{align*}
where $C=C(n,s,\theta,\Lambda,\sigma_0)>0$ and $$ \omega(A-\widetilde A,r,x_0) := \dashint_{B_r(x_0)} \dashint_{B_r(x_0)} |A(x,y)- \widetilde A(x,y)|dydx .$$
\end{prop}

\begin{proof} 
First of all, note that the function $v$ that uniquely solves (\ref{constcof3}) exists by \cite[Proposition 4.1]{MeN}. Using $w$ as a test function in (\ref{constcof3}) and also in (\ref{neq5}), using (\ref{PhiMonotone}) and taking into account that $A(x,y)=\widetilde A(x,y)$ whenever $(x,y) \notin \mathcal{B}_r(x_0)$, we obtain
\begin{align*}
& \int_{\mathbb{R}^n} \int_{\mathbb{R}^n} \frac{(w(x)-w(y))^2}{|x-y|^{n+2s}}dydx \\
\leq & \Lambda \int_{\mathbb{R}^n} \int_{\mathbb{R}^n} \widetilde A(x,y) \frac{((u(x)-u(y))-(v(x)-v(y)))^2}{|x-y|^{n+2s}}dydx \\
& \leq \Lambda^2 \bigg ( \int_{\mathbb{R}^n} \int_{\mathbb{R}^n} \widetilde A(x,y) \frac{\Phi(u(x)-u(y))(w(x)-w(y))}{|x-y|^{n+2s}}dydx \\
& - \int_{\mathbb{R}^n} \int_{\mathbb{R}^n} \widetilde A(x,y) \frac{ \Phi(v(x)-v(y))(w(x)-w(y))}{|x-y|^{n+2s}}dydx \bigg ) \\
= & \Lambda^2 \bigg (\int_{\mathbb{R}^n} \int_{\mathbb{R}^n} (\widetilde A(x,y)-A(x,y)) \frac{\Phi(u(x)-u(y))(w(x)-w(y))}{|x-y|^{n+2s}}dydx \\
& + \int_{\mathbb{R}^n} \int_{\mathbb{R}^n} A(x,y) \frac{\Phi(u(x)-u(y))(w(x)-w(y))}{|x-y|^{n+2s}}dydx - \int_{B_{2r}(x_0)} \widetilde f(x)w(x)dx \bigg )  \\
= & \Lambda^2 \bigg ( \underbrace{\int_{B_{r}(x_0)} \int_{B_{r}(x_0)} (\widetilde A(x,y)-A(x,y)) \frac{\Phi(u(x)-u(y))(w(x)-w(y))}{|x-y|^{n+2s}}dydx}_{=: I_1} \\
& + C_{n,s} \underbrace{\int_{\mathbb{R}^n} \int_{\mathbb{R}^n} \frac{(g(x)-g(y))(w(x)-w(y))}{|x-y|^{n+2s}}dydx}_{:= I_2} \\
& + \underbrace{\int_{B_{2r}(x_0)} (f(x) - \widetilde f(x))w(x)dx}_{:=I_3} \bigg ).
\end{align*}
Let $\sigma=\sigma(n,s,\Lambda,\sigma_0)>0$ and  $\gamma=\gamma(n,s,\Lambda,\sigma_0)>0$ be given by Theorem \ref{KMS}.
By using (\ref{PhiLipschitz}), the Cauchy-Schwarz inequality and then H\"older's inequality with conjugated exponents $\frac{2+\sigma}{\sigma}$ and $\frac{2+\sigma}{2}$, we estimate $I_1$ as follows
\begin{align*}
I_1 \leq & \Lambda \left (\int_{B_{r}(x_0)} \int_{B_{r}(x_0)} (\widetilde A(x,y)-A(x,y))^2 \frac{(u(x)-u(y))^2}{|x-y|^{n+2s}}dydx \right )^\frac{1}{2} \\
& \times \left ( \int_{B_{r}(x_0)} \int_{B_{r}(x_0)} \frac{(w(x)-w(y))^2}{|x-y|^{n+2s}}dydx \right )^\frac{1}{2} \\
= & \Lambda \left (\int_{\mathcal{B}_{r}(x_0)} (\widetilde A(x,y)-A(x,y))^2 U_\gamma^2(x,y) d\mu_\gamma \right )^\frac{1}{2} \\
& \times \left ( \int_{B_{r}(x_0)} \int_{B_{r}(x_0)} \frac{(w(x)-w(y))^2}{|x-y|^{n+2s}}dydx \right )^\frac{1}{2} \\
\leq & C_1 \left ( \left (\int_{\mathcal{B}_{r}(x_0)} |\widetilde A(x,y)-A(x,y)|^{\frac{4}{\sigma}+2} d\mu_\gamma \right )^\frac{\sigma}{2+\sigma} \left (\dashint_{\mathcal{B}_{r}(x_0)} U_\gamma^{2+\sigma} d\mu_\gamma \right )^{\frac{2}{2+\sigma}}r^\frac{2n+4\gamma}{2+\sigma} \right )^\frac{1}{2} \\
& \times \left ( \int_{\mathbb{R}^n} \int_{\mathbb{R}^n} \frac{(w(x)-w(y))^2}{|x-y|^{n+2s}}dydx \right )^\frac{1}{2} ,
\end{align*}
where $C_1=C_1(n,\Lambda,\sigma,\gamma)>0$.
By using H\"older's inequality with conjugated exponents $\frac{n-\gamma}{n-2\gamma}$ and $\frac{n-\gamma}{\gamma}$, we obtain
\begin{align*}
& \int_{\mathcal{B}_{r}(x_0)} |\widetilde A(x,y)-A(x,y)|^{\frac{4}{\sigma}+2} d\mu_\gamma \\
= & \int_{B_{r}(x_0)} \int_{B_{r}(x_0)} \frac{|\widetilde A(x,y)-A(x,y)|^{\frac{4}{\sigma}+2}}{|x-y|^{n-2\gamma}}dydx \\
\leq & \left ( \int_{B_{r}(x_0)} \int_{B_{r}(x_0)} |\widetilde A(x,y)-A(x,y)|^{\left (\frac{4}{\sigma}+2 \right ) \left (\frac{n-\gamma}{\gamma} \right)} dydx \right )^\frac{\gamma}{n-\gamma} \left ( \int_{B_{r}(x_0)} \int_{B_{r}(x_0)} \frac{dydx}{|x-y|^{n-\gamma}} \right )^\frac{n-2\gamma}{n-\gamma} \\
\leq & (2\Lambda)^{\frac{4}{\sigma}+2} \left ( \dashint_{B_{r}(x_0)} \dashint_{B_{r}(x_0)} |\widetilde A(x,y)-A(x,y)| dydx \right )^\frac{\gamma}{n-\gamma} C_2 r^{2n \frac{\gamma}{n-\gamma}} \mu_{\gamma/2}(\mathcal{B}_{r}(x_0))^\frac{n-2\gamma}{n-\gamma} \\
\leq & C_3 \omega(A-\widetilde A,r,x_0)^\frac{\gamma}{n-\gamma} r^{\frac{2n \gamma}{n-\gamma}+(n+\gamma) \frac{n-2\gamma}{n-\gamma}},
\end{align*}
where $C_2=C_2(n,s,\gamma)>0$ and $C_3=C_3(n,s,\gamma,\sigma,\Lambda)>0$.

By combining the last two displays with the estimate (\ref{KMSmod}) and the fact that
$$ \frac{\sigma}{2+\sigma} \left (\frac{2n \gamma}{n-\gamma}+(n+\gamma) \frac{n-2\gamma}{n-\gamma} \right ) + \frac{2n+4\gamma}{2+\sigma}+2(\theta-\gamma) = n+2\theta,$$
we arrive at
\begin{align*}
I_1 \leq & C_4 \Bigg ( \omega(A-\widetilde A,r,x_0)^\frac{\gamma}{n-\gamma} \mu(\mathcal{B}_r(x_0)) \left (\sum_{k=1}^\infty 2^{-k(s-\theta)} \left ( \dashint_{\mathcal{B}_{2^kr}(x_0)} U^2 d\mu \right )^\frac{1}{2} \right )^2 \\
& + \omega(A-\widetilde A,r,x_0)^\frac{\gamma}{n-\gamma} \mu(\mathcal{B}_r(x_0)) \left ( \dashint_{\mathcal{B}_{2r}(x_0)} G^{2+\sigma_0} d\mu \right )^{\frac{2}{2+\sigma_0}} \\ & + \omega(A-\widetilde A,r,x_0)^\frac{\gamma}{n-\gamma} \mu(\mathcal{B}_r(x_0)) \left (\sum_{k=1}^\infty 2^{-k(s-\theta)} \left ( \dashint_{\mathcal{B}_{2^kr}(x_0)} G^2 d\mu \right )^\frac{1}{2}\right )^2 \\ & + \omega(A-\widetilde A,r,x_0)^\frac{\gamma}{n-\gamma} r^{2(s-\theta)}\mu(\mathcal{B}_r(x_0)) \left ( \dashint_{B_{2r}(x_0)}|f|^{2_\star+\sigma_0} dx \right )^\frac{2}{2+\sigma_0}\Bigg ) ,
\end{align*}
where $C_4=C_4(n,s,\gamma,\theta,\sigma,\sigma_0,\Lambda)>0$.

In order to estimate $I_2$, we set $g_1:=g-\overline {g}_{B_{2r}(x_0)}$ and split the integral as follows
\begin{align*}
I_2 \leq & \int_{B_{4r}(x_0)} \int_{B_{4r}(x_0)} \frac{|g(x)-g(y)| |w(x)-w(y)|}{|x-y|^{n+2s}}dydx \\
& + 2 \int_{B_{2r}(x_0)} \int_{\mathbb{R}^n \setminus B_{4r}(x_0)} \frac{|g_1(x)-g_1(y)| |w(x)|}{|x-y|^{n+2s}}dydx \\
\leq & \underbrace{\int_{B_{4r}(x_0)} \int_{B_{4r}(x_0)} \frac{|g(x)-g(y)| |w(x)-w(y)|}{|x-y|^{n+2s}}dydx}_{=:I_{2,1}} \\
& + 2 \underbrace{ \int_{B_{2r}(x_0)} \int_{\mathbb{R}^n \setminus B_{4r}(x_0)} \frac{|g_1(x)| |w(x)|}{|x-y|^{n+2s}}dydx}_{=:I_{2,2}} + 2 \underbrace{ \int_{B_{2r}(x_0)} \int_{\mathbb{R}^n \setminus B_{4r}(x_0)} \frac{|g_1(y)| |w(x)|}{|x-y|^{n+2s}}dydx}_{=:I_{2,3}}.
\end{align*}
By using the Cauchy-Schwarz inequality, we estimate $I_{2,1}$ as follows
\begin{align*}
I_{2,1} \leq & C_5 \mu(\mathcal{B}_r(x_0))^\frac{1}{2} \left ( \dashint_{\mathcal{B}_{4r}(x_0)} G^2 d\mu \right )^\frac{1}{2} \left ( \int_{\mathbb{R}^n} \int_{\mathbb{R}^n} \frac{(w(x)-w(y))^2}{|x-y|^{n+2s}}dydx \right )^\frac{1}{2},
\end{align*}
where $C_5=C_5(n,\theta)>0$.
For any $x \in B_{2r}(x_0)$ and any $y \in \mathbb{R}^n \setminus B_{4r}(x_0)$ we have
\begin{equation} \label{elemest21}
|x_0-y| \leq |x_0-x|+|x-y| <\left(\frac{2r}{|x-y|}+1 \right )|x-y| \leq \left(\frac{2r}{2r}+1 \right) |x-y|=2|x-y|. 
\end{equation}
Moreover, in view of integration in polar coordinates we have
\begin{equation} \label{intpolarnon}
\int_{\mathbb{R}^n \setminus B_{4r}(x_0)} \frac{dy}{|x_0-y|^{n+2s}} =  \int_{\mathbb{R}^n \setminus B_{4r}} \frac{dy}{|y|^{n+2s}} = \frac{C_6}{r^{2s}},
\end{equation}
where $C_6=C_6(n,s) >0$.
Therefore, by using (\ref{elemest21}), (\ref{intpolarnon}), the Cauchy-Schwarz inequality, the fractional Poincar\'e inequality (Lemma \ref{Poincare}) and the fractional Friedrichs-Poincar\'e inequality (Proposition \ref{Friedrichsx}), we obtain
\begin{align*}
I_{2,2} \leq & 2^{n+2s} \int_{B_{2r}(x_0)} \int_{\mathbb{R}^n \setminus B_{4r}(x_0)} \frac{|g_1(x)| |w(x)|}{|x_0-y|^{n+2s}}dydx \\
= & C_7 r^{-2s} \int_{B_{2r}(x_0)} |g_1(x)| |w(x)| dx \\
\leq & C_7 r^{-2s} \left (\int_{B_{2r}(x_0)} |g_1(x)|^2 dx \right )^\frac{1}{2} \left ( \int_{B_{2r}(x_0)} |w(x)|^2 dx \right )^\frac{1}{2} \\
\leq & C_8 \mu(\mathcal{B}_r(x_0))^\frac{1}{2} \left (\dashint_{\mathcal{B}_{2r}(x_0)} G^2d\mu \right )^\frac{1}{2} \left ( \int_{\mathbb{R}^n} \int_{\mathbb{R}^n} \frac{|w(x)-w(y)|^2}{|x-y|^{n+2s}} dydx \right )^\frac{1}{2} ,
\end{align*}
where and all constants depend only on $n,s,\theta$ and $\Lambda$. 
Next, by (\ref{elemest21}), the Cauchy-Schwarz inequality, the fractional Friedrichs-Poincar\'e inequality (Proposition \ref{Friedrichsx}) and Lemma \ref{anullusdecomp}, we obtain
\begin{align*}
I_{2,3} \leq & 2^{n+2s} \bigg (\int_{\mathbb{R}^n \setminus B_{2r}(x_0)} \frac{|g_1(y)|}{|x_0-y|^{n+2s}}dy \bigg ) \bigg (\int_{B_{2r}(x_0)}|w(x)|dx \bigg ) \\
\leq & C_9 \left (r^{-s+\theta} \sum_{k=1}^\infty 2^{-k(s-\theta)} \left ( \dashint_{\mathcal{B}_{2^k r}(x_0)} G^2 d\mu \right )^\frac{1}{2} \right) |B_{2r}|^\frac{1}{2} \left (\int_{B_{2r}(x_0)}|w(x)|^2dx \right )^\frac{1}{2} \\
\leq & C_{10} \mu(\mathcal{B}_r(x_0))^\frac{1}{2} \left (\sum_{k=1}^\infty 2^{-k(s-\theta)} \left ( \dashint_{\mathcal{B}_{2^k r}(x_0)} G^2 d\mu \right )^\frac{1}{2} \right) \left (\int_{\mathbb{R}^n} \int_{\mathbb{R}^n} \frac{|w(x)-w(y)|^2}{|x-y|^{n+2s}} dydx \right )^\frac{1}{2} ,
\end{align*}
$C_{9}$ and $C_{10}$ depend only on $n,s$ and $\theta$. Next, by H\"older's inequality and the fractional Sobolev inequality (see \cite[Theorem 6.5]{Hitch}), for $I_3$ we get
\begin{align*}
I_3 & \leq \left (\int_{B_{2r}(x_0)} |f(x)-\widetilde f(x)|^{2_\star} dx \right )^\frac{1}{2_\star} \left ( \int_{B_{2r}(x_0)} |w(x)|^\frac{2n}{n-2s} dx \right )^\frac{n-2s}{2n} \\
& \leq C_{11} \left (\int_{B_{2r}(x_0)} |f(x)-\widetilde f(x)|^{2_\star} dx \right )^\frac{1}{2_\star} \left ( \int_{\mathbb{R}^n} \int_{\mathbb{R}^n} \frac{|w(x)-w(y)|^2}{|x-y|^{n+2s}} dydx  \right )^\frac{1}{2} ,
\end{align*}
where $C_{11}=C_{11}(n,s)>0$.
Combining all the above estimates along with squaring both sides of the resulting inequality and then dividing by $\int_{\mathbb{R}^n} \int_{\mathbb{R}^n} \frac{|w(x)-w(y)|^2}{|x-y|^{n+2s}} dydx$ on both sides now finishes the proof.
\end{proof}

From now on, we fix some $p \in (2,\infty)$, some $\Lambda \geq 1$ and some coefficient $A \in \mathcal{L}_0(\Lambda)$ that is $\delta$-vanishing in $B_{4n}$, where $\delta >0$ remains to be chosen small enough later. Moreover, we fix another number $q \in [2,p)$ and define
\begin{equation} \label{qstar}
q^\star:=\begin{cases} 
\frac{nq}{n-sq}, & \text{if } n>sq \\
2p, & \text{if } n \leq sq.
\end{cases}
\end{equation}
In addition, we choose the number $\sigma_0>0$ small enough such that $2+\sigma_0 <\min\{(q+q^\star)/2,p\}$ and set 
\begin{equation} \label{q0}
q_0 := \max\{2+\sigma_0,q\}<\min\{(q+q^\star)/2,p\}.\end{equation}
Furthermore, we fix some $g \in W^{s,2}(\mathbb{R}^n)$ and a weak solution $u \in W^{s,2}(\mathbb{R}^n)$ 
of the equation 
\begin{equation} \label{helpeqx}
L_A^\Phi u = (-\Delta)^s g \text { in } B_{4n}
\end{equation}
and set
\begin{equation} \label{tailcontrol}  
\lambda_0:= M_0 \Bigg (\sum_{k=1}^\infty 2^{-k(s-\theta)} \left (\dashint_{\mathcal{B}_{2^k 4n}} U^2 d\mu \right )^\frac{1}{2} + \delta^{-1} \sum_{k=1}^\infty 2^{-k(s-\theta)} \left (\dashint_{\mathcal{B}_{2^k 4n}} G^2 d\mu \right )^\frac{1}{2} +\left (\dashint_{\mathcal{B}_{4n}} G^{q_0} d\mu \right )^\frac{1}{q_0} \Bigg ),
\end{equation}
where $M_0 \geq 1 $ remains to be chosen large enough.

\begin{lem} \label{apppl}
	Let $M>0$, $x_0 \in B_{\frac{\sqrt{n}}{2}}$, $r \in \left (0,\frac{\sqrt{n}}{2} \right )$ and $\lambda \geq \lambda_0$. Moreover, consider the coefficient
	$$ \widetilde A(x,y) := \begin{cases} \normalfont
	\overline A_{3r,x_0,x_0} & \text{if } (x,y) \in \mathcal{B}_{3r}(x_0) \\
	A(x,y) & \text{if } (x,y) \notin \mathcal{B}_{3r}(x_0).
	\end{cases} $$
	Then for any $\varepsilon_0 >0$, there exists some small enough $\delta = \delta (\varepsilon_0,n,s,\theta,\Lambda,M) \in (0,1)$, such that
	under the assumptions made above along with
\begin{equation} \label{condU}
 {\mathcal{M}}_{\mathcal{B}_{4n}} (U^2)(x_0) \leq  M\lambda^2, \quad  {\mathcal{M}}_{\mathcal{B}_{4n}} (G^{q_0})(x_0) \leq M\lambda^{q_0} \delta^{q_0} ,
\end{equation}
	for the unique weak solution $v \in W^{s,2}(\mathbb{R}^n)$ of the Dirichlet problem
\begin{equation} \label{constcof3x}
\begin{cases} \normalfont
L_{\widetilde A}^\Phi v = 0 & \text{ in } B_{6r}(x_0) \\
v = u & \text{ a.e. in } \mathbb{R}^n \setminus B_{6r}(x_0)
\end{cases}
\end{equation}
	and the function $$W(x,y):=\frac{|u(x)-v(x)-u(y)+v(y)|}{|x-y|^{s+\theta}}, \quad (x,y) \in \mathbb{R}^{2n},$$ we have
		\begin{equation} \label{west}
		\int_{\mathbb{R}^{2n}} W^2 d \mu \leq \varepsilon^2 \lambda^2 \mu(\mathcal{B}_r(x_0)) .
		\end{equation}
		\newline Moreover, the function
		$$V(x,y):=\frac{|v(x)-v(y)|}{|x-y|^{s+\theta}}, \quad (x,y) \in \mathbb{R}^{2n}$$
		satisfies the estimate
		\begin{equation} \label{Vest}
		||V||_{L^\infty(\mathcal{B}_{2r}(x_0),d\mu)} \leq N_0 \lambda
		\end{equation}
		for some constant $N_0= N_0(n,s,\theta,\Lambda,M)>0$.
\end{lem}

\begin{proof}
Fix $x_0 \in B_{\frac{\sqrt{n}}{2}}$ and $r \in \left (0,\frac{\sqrt{n}}{2} \right )$ and note that $\widetilde A=A $ in $\mathbb{R}^{2n} \setminus \mathcal{B}_{3r}(x_0).$ Moreover, since $A$ is $\delta$-vanishing in $B_{4n}$, we have
\begin{equation} \label{van}
\omega(A-\widetilde A,3r,x_0) = \dashint_{B_{3r}(x_0)} \dashint_{B_{3r}(x_0)} |A(x,y)-\overline A_{3r,x_0,x_0}|dydx \leq \delta.
\end{equation}
First, we prove (\ref{west}).
Let $m \in \mathbb{N}$ be determined by $2^{m-1} r < \sqrt{n} \leq 2^{m} r$, note that $m \geq 2$. Then for any $k < m$, by (\ref{condU}) we have
\begin{equation} \label{UG}
\dashint_{\mathcal{B}_{2^k3r}(x_0)} U^2 d\mu \leq M\lambda^2, \quad \dashint_{\mathcal{B}_{2^k3 r}(x_0)} G^{q_0} d\mu \leq M\lambda^{q_0} \delta^{q_0} .
\end{equation}
On the other hand, in view of (\ref{tailcontrol}) and the inclusions $$\mathcal{B}_{2^k \sqrt{n}}(x_0) \subset \mathcal{B}_{2^{k+m-1}3r}(x_0) \subset \mathcal{B}_{2^k 3\sqrt{n}}(x_0) \subset \mathcal{B}_{2^k 4n},$$ we have
\begin{align*}
\sum_{k=m}^\infty 2^{-k(s-\theta)} \left ( \dashint_{\mathcal{B}_{2^k 3r}(x_0)} U^2 d\mu \right )^\frac{1}{2} = & 2^{-(m-1)(s-\theta)} \sum_{k=1}^\infty 2^{-k(s-\theta)} \left ( \dashint_{\mathcal{B}_{2^{k+m-1}3r}(x_0)} U^2 d\mu \right )^\frac{1}{2} \\
\leq & \sum_{k=1}^\infty 2^{-k(s-\theta)} \left (\frac{\mu(\mathcal{B}_{2^k 4n})}{\mu \left (\mathcal{B}_{2^k \sqrt{n}} \right)} \dashint_{\mathcal{B}_{2^k 4n}} U^2 d\mu \right )^\frac{1}{2} \\
= & C_1 \sum_{k=1}^\infty 2^{-k(s-\theta)} \left ( \dashint_{\mathcal{B}_{2^k 4n}} U^2 d\mu \right )^\frac{1}{2} \leq C_1 \lambda_0,
\end{align*}
where $C_1=C_1(n,\theta)>0$.
Together with (\ref{UG}) and the facts that $\theta < s$ and $\lambda \geq \lambda_0$, we arrive at
\begin{equation} \label{Ulambda}
\begin{aligned}
& \sum_{k=1}^\infty 2^{-k(s-\theta)} \left ( \dashint_{\mathcal{B}_{2^k3r}(x_0)} U^2 d\mu \right )^\frac{1}{2} \\ \leq & \sum_{k=1}^{m-1} 2^{-k(s-\theta)} \left ( \dashint_{\mathcal{B}_{2^k3r}(x_0)} U^2 d\mu \right )^\frac{1}{2} + \sum_{k=m}^\infty 2^{-k(s-\theta)} \left ( \dashint_{\mathcal{B}_{2^k3r}(x_0)} U^2 d\mu \right )^\frac{1}{2} \\
\leq & M^\frac{1}{2} \lambda \sum_{k=1}^\infty 2^{-k(s-\theta)}+ C_1\lambda_0 \leq C_2 \lambda, 
\end{aligned}
\end{equation}
where $C_2=C_2(n,s,\theta,M)>0$.
By a similar reasoning as above, we have
\begin{align*}
\sum_{k=m}^\infty 2^{-k(s-\theta)} \left ( \dashint_{\mathcal{B}_{2^k3r}(x_0)} G^2 d\mu \right )^\frac{1}{2} 
\leq C_1 \sum_{k=1}^\infty 2^{-k(s-\theta)} \left ( \dashint_{\mathcal{B}_{2^k 4n}} G^2 d\mu \right )^\frac{1}{2} \leq C_1 \lambda_0 \delta
\end{align*}
and therefore along with H\"older's inequality
\begin{equation} \label{Glambda}
\begin{aligned}
& \sum_{k=1}^\infty 2^{-k(s-\theta)} \left ( \dashint_{\mathcal{B}_{2^k3r}(x_0)} G^2 d\mu \right )^\frac{1}{2} \\ \leq & \sum_{k=1}^{m-1} 2^{-k(s-\theta)} \left ( \dashint_{\mathcal{B}_{2^k3r}(x_0)} G^{q_0} d\mu \right )^\frac{1}{q_0} + \sum_{k=m}^\infty 2^{-k(s-\theta)} \left ( \dashint_{\mathcal{B}_{2^k3r}(x_0)} G^2 d\mu \right )^\frac{1}{2} \\
\leq & M^\frac{1}{q_0} \lambda \delta \sum_{k=1}^\infty 2^{-k(s-\theta)}+ C_1\lambda_0\delta \leq C_2 \lambda\delta.
\end{aligned}
\end{equation}
By combining (\ref{van}), (\ref{UG}), (\ref{Ulambda}) and (\ref{Glambda}) with Proposition \ref{appplxy}, the fact that $2+\sigma_0 \leq q_0$ and H\"older's inequality, we obtain
\begin{align*}
& \int_{\mathbb{R}^{2n}} W^2 d \mu
= \int_{\mathbb{R}^n} \int_{\mathbb{R}^n} \frac{(w(x)-w(y))^2}{|x-y|^{n+2s}}dydx \\
\leq & C_3 \omega(A-\widetilde A,3r,x_0)^\frac{\gamma}{n-\gamma} \mu(\mathcal{B}_r(x_0)) \left (\sum_{k=1}^\infty 2^{-k(s-\theta)} \left ( \dashint_{\mathcal{B}_{2^k3r}(x_0)} U^2 d\mu \right )^\frac{1}{2} \right )^2 \\	& + C_3 \omega(A-\widetilde A,3r,x_0)^\frac{\gamma}{n-\gamma} \mu(\mathcal{B}_r(x_0)) \left ( \dashint_{\mathcal{B}_{6r}(x_0)} G^{q_0} d\mu \right )^{\frac{2}{q_0}} \\ & + C_3 (\omega(A-\widetilde A,3r,x_0)^\frac{\gamma}{n-\gamma}+1) \mu(\mathcal{B}_r(x_0)) \left (\sum_{k=1}^\infty 2^{-k(s-\theta)} \left ( \dashint_{\mathcal{B}_{2^k3r}(x_0)} G^2 d\mu \right )^\frac{1}{2}\right )^2 \\
\leq & C_4 \mu(\mathcal{B}_r(x_0)) \lambda^2 \delta^\frac{2\gamma}{n-\gamma} < \varepsilon^2 \lambda^2 \mu(\mathcal{B}_r(x_0)) ,
\end{align*}
where $C_4=C_4(n,s,\theta,\Lambda,M)>0$ and the last inequality was obtained by choosing $\delta$ sufficiently small. This proves (\ref{west}). \par
Let us now proof the estimate (\ref{Vest}).
Since $\widetilde A$ is constant and therefore continuous in $\mathcal{B}_{3r}(x_0)$, by Theorem \ref{HiHol} and Remark \ref{HiHolrem} with $\alpha=s+\theta \in \left (0,\min \{2s,1\} \right )$ and $c=\overline v_{B_{3r}(x_0)}$, we have
\begin{align*}
||V||_{L^\infty(\mathcal{B}_{2r}(x_0),d\mu)}^2 \leq & [v]_{C^{\alpha}(B_{2r}(x_0))}^2 \\
\leq & \frac{C_{5}}{r^{2(s+\theta)}} \left (r^{-n} ||v_1||_{L^2(B_{3r}(x_0))}^2 + \left (r^{2s} \int_{\mathbb{R}^n \setminus B_{3r}(x_0)} \frac{|v_1(y)|}{|x_0-y|^{n+2s}}dy \right )^2 \right ) \\
\leq & \frac{C_{6}}{r^{2(s+\theta)}} \left (r^{2s-n} [v]_{W^{s,2}(B_{3r}(x_0))}^2 + \left (r^{2s} \int_{\mathbb{R}^n \setminus B_{3r}(x_0)} \frac{|v_1(y)|}{|x_0-y|^{n+2s}}dy \right )^2 \right ),
\end{align*}
where $v_1:=v-\overline v_{B_{3r}(x_0)}$ and $C_5$ and $C_6$ depend only on $n,s$ and $\Lambda$. Here we also used Lemma \ref{Poincare} in order to obtain the last inequality.
By using (\ref{UG}) and (\ref{west}), We further estimate the first term on the right-hand side as follows
\begin{align*}
& r^{2s-n} [v]_{W^{s,2}(B_{3r}(x_0))}^2 \\ \leq & 2 r^{2s-n} \left (\int_{B_{3r}(x_0)} \int_{B_{3r}(x_0)} \frac{ |u(x)-u(y)|^2 }{|x-y|^{n+2s}}dydx + \int_{B_{3r}(x_0)} \int_{B_{3r}(x_0)} \frac{ |w(x)-w(y)|^2 }{|x-y|^{n+2s}}dydx \right )
\\ \leq & 2 r^{2s-n} \left (\lambda^2\mu(\mathcal{B}_{3r}(x_0)) + \varepsilon^2 \lambda^2 \mu(\mathcal{B}_r(x_0)) \right ) \leq C_7\lambda^2 r^{2(s+\theta)},
\end{align*}
where $C_7=C_7(n,\theta)>0$.
Moreover, by taking into account that $v=u$ in $\mathbb{R}^n \setminus B_{3r}(x_0)$, we split the tail term as follows
\begin{align*}
& \left (r^{2s} \int_{\mathbb{R}^n \setminus B_{3r}(x_0)} \frac{|v_1(y)|}{|x_0-y|^{n+2s}}dy \right )^2 \\ 
\leq & 2 \underbrace{ \left (r^{2s} \int_{\mathbb{R}^n \setminus B_{3r}(x_0)} \frac{|u(y)- \overline u_{B_{3r}(x_0)}|}{|x_0-y|^{n+2s}}dy \right )^2}_{=:J_1} + 2 \underbrace{\left ( r^{2s} \int_{\mathbb{R}^n \setminus B_{3r}(x_0)} \frac{|\overline u_{B_{3r}(x_0)}- \overline v_{B_{3r}(x_0)}|}{|x_0-y|^{n+2s}}dy \right )^2 }_{=:J_2}
\end{align*}
In view of Lemma \ref{anullusdecomp} and (\ref{Ulambda}), for $J_1$ we have
$$ J_1 \leq C_8 \left (r^{s+\theta} \sum_{k=1}^\infty 2^{-k(s-\theta)} \left ( \dashint_{\mathcal{B}_{2^k3r}(x_0)} U^2 d\mu \right )^\frac{1}{2} \right )^2 \leq C_8 \lambda^2 r^{2(s+\theta)} .$$
Moreover, by using (\ref{intpolarnon}), applying Jensen's inequality twice, using the fractional Friedrichs-Poincar\'e inequality with respect to $w=u-v \in W^{s,2}_0(B_{2r}(x_0))$ and again taking into account (\ref{UG}) and (\ref{west}), for $J_2$ we obtain
\begin{align*}
J_2 = & C_9 \left | \dashint_{B_{3r}(x_0)} (u(x) - \overline v_{B_{3r}(x_0)})dx \right |^2 \\ \leq & C_9 \dashint_{B_{3r}(x_0)} \left | u(x) - \overline v_{B_{3r}(x_0)} \right |^2 dx \\
= & C_9 \dashint_{B_{3r}(x_0)} \left | \dashint_{B_{3r}(x_0)} (u(x) - v(y))dy \right |^2 dx \\
\leq & C_9 \dashint_{B_{3r}(x_0)} \dashint_{B_{3r}(x_0)}  \left |u(x) - v(y) \right |^2 dydx \\
\leq & 2 C_{10} \left (r^{-2n} \int_{B_{3r}(x_0)} \int_{B_{3r}(x_0)}  \left |u(x) - u(y) \right |^2 dydx + r^{-n} \int_{B_{3r}(x_0)}  \left |w(y) \right |^2 dy \right ) \\
\leq & C_{11} r^{2s-n} \left ( \int_{B_{3r}(x_0)} \int_{B_{3r}(x_0)} \frac{|u(x) - u(y) |^2}{|x-y|^{n+2s}} dydx + \int_{\mathbb{R}^n} \int_{\mathbb{R}^n} \frac{|w(x) - w(y) |^2}{|x-y|^{n+2s}} dydx \right ) \\
\leq & C_{12} \lambda^2 r^{2(s+\theta)},
\end{align*}
where all the constants depend only on $n,s,\theta,\Lambda$ and $M$.
Combining the last five displays now shows that the estimate (\ref{Vest})
holds for some $N_0=N_0(n,s,\theta,\Lambda,M)>0$, which finishes the proof.
\end{proof}

\section{Good-$\lambda$ inequalities} \label{gl}
\subsection{Diagonal good-$\lambda$ inequalities}
The following result is a consequence of the above approximation lemma and roughly speaking shows that if the set where the maximal function of $U$ is very large has a large enough density in a ball, then in this ball the maximal functions of $U$ and $G$ cannot be too small.
\begin{lem} \label{mfuse}
	There is a constant $N_d=N_d(n,s,\theta,\Lambda) \geq 1$, such that the following holds. For any $\varepsilon > 0$ and any $\kappa>0$ there exists some small enough $\delta = \delta(\varepsilon,\kappa,n,s,\theta,\Lambda) \in (0,1)$, 
	such that for any $\lambda \geq \lambda_0$, any $r \in \left (0,\frac{\sqrt{n}}{2} \right )$ and any point $x_0 \in Q_1$
	with
	\begin{equation} \label{ccll}
	\mu \left ( \left \{(x,y) \in \mathcal{B}_{r}(x_0) \mid  {\mathcal{M}}_{\mathcal{B}_{4n}} (U^2)(x,y) > N_d^2 \lambda^2 \right \} \right ) \geq \kappa \varepsilon \mu(\mathcal{B}_r(x_0)),
	\end{equation}
	we have
	\begin{equation} \label{incly}
	\begin{aligned}
	\mathcal{B}_r(x_0) \subset & \left \{ (x,y) \in \mathcal{B}_r(x_0) \mid  {\mathcal{M}}_{\mathcal{B}_{4n}} (U^2)(x,y) > \lambda^2 \right \} \\ 
	& \cap \left \{ (x,y) \in \mathcal{B}_r(x_0) \mid  {\mathcal{M}}_{\mathcal{B}_{4n}} \left (G^{q_0} \right )(x,y) > \lambda^{q_0} \delta^{q_0} \right \},
	\end{aligned}
	\end{equation}
\end{lem}

\begin{proof}
	Let $\varepsilon_0 >0$ and $M>0$ to be chosen and consider the corresponding $\delta = \delta(\varepsilon_0,n,s,\theta,\Lambda,M) \in (0,1)$ given by Lemma $\ref{apppl}$.
	Fix $\varepsilon, \kappa > 0$, $r \in \left (0,\frac{\sqrt{n}}{2} \right )$, $x_0 \in Q_1$ and assume that (\ref{ccll}) holds, but that (\ref{incly}) is false, so that there exists
	a point $(x^\prime,y^\prime) \in \mathcal{B}_r(x_0)$ such that 
	$$ {\mathcal{M}}_{\mathcal{B}_{4n}} (U^2)(x^\prime,y^\prime) \leq \lambda^2, \quad  {\mathcal{M}}_{\mathcal{B}_{4n}} \left (G^{q_0} \right )(x^\prime,y^\prime) \leq \lambda^{q_0} \delta^{q_0}.$$
	Therefore, for any $\rho >0$ we have 
	\begin{equation} \label{lumh}
	\dashint_{\mathcal{B}_\rho(x^\prime,y^\prime)}\chi_{\mathcal{B}_{4n}} U^2 d\mu \leq \lambda^2, \quad \dashint_{\mathcal{B}_\rho(x^\prime,y^\prime)} \chi_{\mathcal{B}_{4n}} G^{q_0} d\mu \leq \lambda^{q_0} \delta^{q_0}.
	\end{equation}
	Note that for any $\rho \geq r$ we have $\mathcal{B}_\rho(x_0) \subset \mathcal{B}_{2\rho}(x^\prime,y^\prime) \subset \mathcal{B}_{3\rho}(x_0)$.
	Together with (\ref{lumh}), this observation yields
	\begin{align*}
	\dashint_{\mathcal{B}_\rho(x_0)} \chi_{\mathcal{B}_{4n}} U^2 d\mu \leq \frac{\mu(\mathcal{B}_{2\rho}(x^\prime,y^\prime))}{\mu(\mathcal{B}_{\rho}(x_0))} \text{ } \dashint_{\mathcal{B}_{2\rho}(x^\prime,y^\prime)} \chi_{\mathcal{B}_{4n}} U^2 d\mu \leq & \frac{\mu(\mathcal{B}_{3\rho}(x_0))}{\mu(\mathcal{B}_\rho(x_0))} \text{ } \dashint_{\mathcal{B}_{2\rho}(x^\prime,y^\prime)} \chi_{\mathcal{B}_{4n}} U^2 d\mu \\ \leq & 3^{n+2\theta} \lambda^2
	\end{align*}
	and similarly
	\begin{align*}
	\dashint_{\mathcal{B}_\rho(x_0)} \chi_{\mathcal{B}_{4n}} G^{q_0} d\mu \leq & \frac{\mu(\mathcal{B}_{2\rho}(x^\prime,y^\prime))}{\mu(\mathcal{B}_{\rho}(x_0))} \text{ } \dashint_{\mathcal{B}_{2\rho}(x^\prime,y^\prime)} \chi_{\mathcal{B}_{4n}} G^{q_0} d\mu \\ \leq& \frac{\mu(\mathcal{B}_{3\rho}(x_0))}{\mu(\mathcal{B}_\rho(x_0))} \text{ } \dashint_{\mathcal{B}_{2\rho}(x^\prime,y^\prime)} \chi_{\mathcal{B}_{4n}} G^{q_0} d\mu \leq 3^{n+2\theta} \lambda^{q_0} \delta^{q_0}
	\end{align*}
	so that $U$ and $G$ satisfy 
	the condition $(\ref{condU})$ with $M=3^{n+2\theta}$.
	Therefore, by Lemma \ref{apppl} the unique weak solution $v \in W^{s,2}(\mathbb{R}^n)$ of the Dirichlet problem
	$$ \begin{cases} \normalfont
	L_{\widetilde A}^\Phi v = 0 & \text{ weakly in } B_{6r}(x_0) \\
	v = u & \text{ a.e. in } \mathbb{R}^n \setminus B_{6r}(x_0),
	\end{cases} $$
	satisfies
	\begin{equation} \label{aprox5}
	\int_{\mathbb{R}^{2n}} W^2 d\mu \leq  \varepsilon_0^2 \lambda^2 \mu(\mathcal{B}_{r}(x_0)) ,
	\end{equation}
	where $W$ is given as in Lemma \ref{apppl}.
	Moreover, by Lemma $\ref{apppl}$ there exists a constant $N_0=N_0(n,s,\theta, \Lambda) >0$ such that 
	\begin{equation} \label{loclinf}
	||V||_{L^\infty(\mathcal{B}_{2r}(x_0))}^2 \leq N_0^2 \lambda^2.
	\end{equation}
	Next, we define $N_d := (\max \{ 4 N_0^2, 5^{n+2\theta} \})^{1/2} > 1$ and claim that 
	\begin{equation} \label{inclusion}
	\begin{aligned}
	& \left \{ (x,y) \in \mathcal{B}_r(x_0) \mid  {\mathcal{M}}_{\mathcal{B}_{4n}} ( U^2 )(x,y) > N_d^2\lambda^2 \right \} \\ \subset & \left \{ (x,y) \in \mathcal{B}_r(x_0) \mid  {\mathcal{M}}_{\mathcal{B}_{2r}(x_0)} ( W^2 )(x,y) > N_0^2\lambda^2 \right \}. 
	\end{aligned}
	\end{equation}
	To see this, assume that 
	\begin{equation} \label{menge}
	(x_1,y_1) \in \left \{ x \in \mathcal{B}_r(x_0) \mid  {\mathcal{M}}_{\mathcal{B}_{2r}(x_0)} ( W^2 ) (x,y) \leq N_0^2\lambda^2 \right \}. 
	\end{equation}
	For $ \rho < r$, we have $\mathcal{B}_\rho (x_1,y_1) \subset \mathcal{B}_r(x_1,y_1) \subset \mathcal{B}_{2r}(x_0)$, so that together with $(\ref{menge})$ and $(\ref{loclinf})$ we deduce 
	\begin{align*}
	\dashint_{\mathcal{B}_\rho (x_1,y_1)} U^2 d\mu & \leq 2 \text{ } \dashint_{\mathcal{B}_\rho(x_1,y_1)} \left ( W^2+ V^2 \right )d\mu \\
	& \leq 2 \text{ } \dashint_{\mathcal{B}_\rho(x_1,y_1)} W^2 d\mu + 2 \text{ } ||V||_{L^\infty(\mathcal{B}_\rho(x_1,y_1))}^2 \\
	& \leq 2 \text{ }  {\mathcal{M}}_{\mathcal{B}_{2r}(x_0)} ( W^2 )(x_1,y_1) + 2 \text{ } ||V||_{L^\infty(\mathcal{B}_{2r}(x_0))}^2 \leq 4 N_0^2\lambda^2. 
	\end{align*}
	On the other hand, for $\rho \geq r$ we have $ \mathcal{B}_\rho (x_1,y_1) \subset \mathcal{B}_{3 \rho}(x^\prime,y^\prime) \subset \mathcal{B}_{5 \rho}(x_1,y_1)$, so that $(\ref{lumh})$ implies 
	\begin{align*}
	\dashint_{\mathcal{B}_\rho(x_1,y_1)} \chi_{\mathcal{B}_{4n}} U^2 d\mu \leq & \frac{\mu(\mathcal{B}_{3 \rho} (x^\prime,y^\prime))}{\mu(\mathcal{B}_\rho(x_1,y_1))} \dashint_{\mathcal{B}_{3 \rho} (x^\prime,y^\prime)} \chi_{\mathcal{B}_{4n}} U^2 d\mu \\ \leq & \frac{\mu(\mathcal{B}_{5 \rho} (x_1,y_1))}{\mu(\mathcal{B}_\rho(x_1,y_1))} \dashint_{\mathcal{B}_{3 \rho} (x^\prime,y^\prime)} \chi_{\mathcal{B}_{4n}} U^2 d\mu \leq 5^{n+2\theta}\lambda^2.
	\end{align*}
	Thus, we have $$ (x_1,y_1) \in \left \{ (x,y) \in \mathcal{B}_r(x_0,y_0) \mid  {\mathcal{M}}_{\mathcal{B}_{4n}} ( U^2 )(x,y) \leq N_d^2 \lambda^2 \right \} ,$$ 
	which implies $(\ref{inclusion})$. 
	Now using $(\ref{inclusion})$, the weak $1$-$1$ estimate from Proposition \ref{Maxfun} and $(\ref{aprox5})$, we conclude that there exists some constant $C=C(n,\theta)>0$ such that 
	\begin{align*}
	& \mu \left ( \left \{(x,y) \in \mathcal{B}_r(x_0) \mid  {\mathcal{M}}_{\mathcal{B}_{4n}} (U^2)(x,y) > N_d^2\lambda^2 \right \} \right ) \\
	\leq & \mu \left ( \left \{ (x,y) \in \mathcal{B}_r(x_0) \mid  {\mathcal{M}}_{\mathcal{B}_{2r}(x_0)} ( W^2 )(x,y) > N_0^2\lambda^2 \right \} \right ) \\
	\leq & \frac{C}{N_0^2\lambda^2} \int_{\mathbb{R}^{2n}} W^2 d\mu \\
	\leq & \frac{C}{N_0^2} \mu(\mathcal{B}_{r}(x_0)) \varepsilon_0^2 < \varepsilon \kappa \mu(\mathcal{B}_{r}(x_0)),
	\end{align*}
	where the last inequality is obtained by choosing $\varepsilon_0$ and thus also $\delta$ sufficiently small. This contradicts (\ref{ccll}) and thus finishes our proof.
\end{proof}

\subsection{Off-diagonal reverse H\"older inequalities}
Our next goal is to prove an analogue of the above diagonal good-$\lambda$ inequality on balls that are far away from the diagonal. In order to prove Lemma \ref{mfuse}, the main tool was given by the comparison estimates from section \ref{ce}. Unfortunately, far away from the diagonal the equation cannot be used very efficiently, since the further away we are from the diagonal, the less the estimates available reflect the scale we are working at. In particular, far away from the diagonal no useful comparison estimates are available. \par In order to bypass this loss of information, we replace the comparison estimates used in the diagonal setting by certain off-diagonal reverse H\"older inequalities with diagonal correction terms.
Although such reverse H\"older inequalities lead to in some sense weaker good-$\lambda$ inequalities than comparison estimates, by using combinatorial techniques and an iteration argument in the end we will nevertheless be able to deduce the desired regularity. \par
For this reason, from now on we assume that for any $r>0$, $x_0 \in \mathbb{R}^n$ with $B_{r}(x_0) \subset B_{4n}$, $U$ satisfies an estimate of the form
\begin{equation} \label{J}
\begin{aligned}
\left (\dashint_{\mathcal{B}_{r/2}(x_0)} U^{q} d\mu \right )^{\frac{1}{q}} \leq C_q& \Bigg ( \sum_{k=1}^\infty 2^{-k(s-\theta)} \left ( \dashint_{\mathcal{B}_{2^kr}(x_0)} U^2 d\mu \right )^\frac{1}{2} \\ & + \left (\dashint_{\mathcal{B}_{r}(x_0)} G^{q_0} d\mu \right )^\frac{1}{q_0} + \sum_{k=1}^\infty 2^{-k(s-\theta)} \left ( \dashint_{\mathcal{B}_{2^kr}(x_0)} G^2 d\mu \right )^\frac{1}{2} \Bigg ),
\end{aligned}
\end{equation}
where $C_q$ depends only on $q,n,s,\theta$ and $\Lambda$.
\begin{prop} \label{offdiagreverse}
	Let $r>0$, $x_0,y_0 \in \mathbb{R}^n$ and suppose that for some $m \in (0,1]$ we have $\textnormal{dist}(B_r(x_0),B_r(y_0)) \geq mr$.
	Then we have
	\begin{align*}
	& \left (\dashint_{\mathcal{B}_r(x_0,y_0)} U^{q^\star} d\mu \right )^{\frac{1}{q^\star}} \\ \leq & C_{nd} \left (\dashint_{\mathcal{B}_r(x_0,y_0)} U^{2} d\mu \right )^\frac{1}{2} \\
	& + C_{nd} \left ( \frac{r}{\textnormal{dist}(B_r(x_0),B_r(y_0))} \right )^{s+\theta} \Bigg ( \sum_{k=1}^\infty 2^{-k(s-\theta)} \left ( \dashint_{\mathcal{B}_{2^kr}(x_0)} U^2 d\mu \right )^\frac{1}{2} \\ & \text{ }+ \left (\dashint_{\mathcal{B}_{2r}(x_0)} G^{q_0} d\mu \right )^\frac{1}{q_0} + \sum_{k=1}^\infty 2^{-k(s-\theta)} \left ( \dashint_{\mathcal{B}_{2^kr}(x_0)} G^2 d\mu \right )^\frac{1}{2} \Bigg ) \\
	& + C_{nd} \left ( \frac{r}{\textnormal{dist}(B_r(x_0),B_r(y_0))} \right )^{s+\theta} \Bigg ( \sum_{k=1}^\infty 2^{-k(s-\theta)} \left ( \dashint_{\mathcal{B}_{2^kr}(y_0)} U^2 d\mu \right )^\frac{1}{2} \\ & \text{ }+ \left (\dashint_{\mathcal{B}_{2r}(y_0)} G^{q_0} d\mu \right )^\frac{1}{q_0} + \sum_{k=1}^\infty 2^{-k(s-\theta)} \left ( \dashint_{\mathcal{B}_{2^kr}(y_0)} G^2 d\mu \right )^\frac{1}{2} \Bigg ),
	\end{align*}
	where $C_{nd}=C_{nd}(n,s,\theta,\Lambda,m,q,p) \geq 1$ and $q^\star$ is given by (\ref{qstar}).
\end{prop}

\begin{proof}
	Choose points $x_1 \in \overline B_r(x_0)$ and $y_1 \in \overline B_r(y_0)$ such that $\textnormal{dist}(B_r(x_0),B_r(y_0))=|x_1-y_1|$. For any $(x,y) \in \mathcal{B}_r(x_0,y_0)$, we obtain
	\begin{align*}
	|x-y| \leq & |x_1-y_1| + |x_1-x| + |y_1-y| \\
	\leq & \textnormal{dist}(B_r(x_0),B_r(y_0)) + 2 r
	\leq 3 \textnormal{dist}(B_r(x_0),B_r(y_0))/m.
	\end{align*}
	Together with the definition of $\textnormal{dist}(B_r(x_0),B_r(y_0))$, it follows that for any $(x,y) \in \mathcal{B}_r(x_0,y_0)$, we have
	\begin{equation} \label{ineq99}
	1 \leq \frac{|x-y|}{\textnormal{dist}(B_r(x_0),B_r(y_0))} \leq 3 /m.
	\end{equation}
	Therefore, by taking into account the definition of the measure $\mu$, we conclude that
	\begin{equation} \label{comparable}
	\frac{c_1 r^{2n}}{\textnormal{dist}(B_r(x_0),B_r(y_0))^{n-2\theta}} \leq \mu(\mathcal{B}_r(x_0,y_0)) \leq \frac{C_1 r^{2n}}{\textnormal{dist}(B_r(x_0),B_r(y_0))^{n-2\theta}},
	\end{equation}
	where $c_1=c_1(n,m,\theta) \in (0,1)$ and $C_1=C_1(n) \geq 1$.
	By (\ref{comparable}) and (\ref{ineq99}), we have
	\begin{align*}
	\left (\dashint_{\mathcal{B}_r(x_0,y_0)} U^{q^\star} d\mu \right )^{\frac{1}{q^\star}} \leq & \left (\frac{\textnormal{dist}(B_r(x_0),B_r(y_0))^{n-2\theta}}{c_1 r^{2n}} \int_{B_r(x_0)}\int_{B_r(y_0)} \frac{|u(x)-u(y)|^{q^\star}}{|x-y|^{n-2\theta+{q^\star}(s+\theta)}}dydx \right )^{\frac{1}{q^\star}} \\
	\leq & C_2 \textnormal{dist}(B_r(x_0),B_r(y_0))^{-(s+\theta)} \left ( \dashint_{B_r(x_0)} \dashint_{B_r(y_0)} |u(x)-u(y)|^{q^\star}dydx \right )^{\frac{1}{q^\star}},
	\end{align*}
	where $C_2=C_2(n,m,\theta)\geq 1$. By using Minkowski's inequality, we further estimate the integral on the right-hand side as follows
	\begin{align*}
	\left ( \dashint_{B_r(x_0)} \dashint_{B_r(y_0)} |u(x)-u(y)|^{q^\star}dydx \right )^{\frac{1}{q^\star}} \leq & \underbrace{\left ( \dashint_{B_r(x_0)} |u(x)- \overline u_{B_r(x_0)}|^{q^\star}dydx \right )^{\frac{1}{q^\star}}}_{=:I_1} \\
	& + \underbrace{\left ( \dashint_{B_r(y_0)} |u(x)- \overline u_{B_r(y_0)}|^{q^\star}dydx \right )^{\frac{1}{q^\star}}}_{=:I_2} \\
	& + \underbrace{|\overline u_{B_r(x_0)}-\overline u_{B_r(y_0)}|}_{=:I_3}.
	\end{align*}
	By using the fractional Sobolev-Poincar\'e inequality (Lemma \ref{SobPoincare}) and then the estimate (\ref{J}), for $I_1$ we obtain
	\begin{align*}
	I_1 \leq & C_3 r^s \left ( \frac{1}{r^n} \int_{B_r(x_0)} \int_{B_r(x_0)} \frac{|u(x)-u(y)|^{q}}{|x-y|^{n+sq}}dydx \right )^{\frac{1}{q}} \\
	= & C_4 r^s \left (\frac{r^{2 \theta} }{\mu(\mathcal{B}_r(x_0))} \int_{B_r(x_0)} \int_{B_r(x_0)} \frac{|u(x)-u(y)|^{q} |x-y|^{(q-2)\theta}}{|x-y|^{n-2\theta+q(s+\theta)}}dydx \right )^{\frac{1}{q}} \\
	\leq & C_5 r^s \left ( r^{q\theta} \dashint_{\mathcal{B}_r(x_0)} \frac{|u(x)-u(y)|^{q}}{|x-y|^{q(s+\theta)}}d\mu \right )^{\frac{1}{q}} \\
	= & C_5 r^{s+\theta} \left ( \dashint_{\mathcal{B}_r(x_0)} U^q d\mu \right )^\frac{1}{q} \\
	\leq & C_q C_5 r^{s+\theta} \Bigg ( \sum_{k=1}^\infty 2^{-k(s-\theta)} \left ( \dashint_{\mathcal{B}_{2^kr}(x_0)} U^2 d\mu \right )^\frac{1}{2} \\ & \text{ } + \left (\dashint_{\mathcal{B}_{2r}(x_0)} G^{q_0} d\mu \right )^\frac{1}{q_0} + \sum_{k=1}^\infty 2^{-k(s-\theta)} \left ( \dashint_{\mathcal{B}_{2^kr}(x_0)} G^2 d\mu \right )^\frac{1}{2} \Bigg )
	\end{align*}
	where $C_3,C_4$ and $C_5$ depend only on $n,s$ and $\theta$. In the same way, for $I_2$ we have
	\begin{align*}
	I_2 \leq & C_q C_5 r^{s+\theta} \Bigg ( \sum_{k=1}^\infty 2^{-k(s-\theta)} \left ( \dashint_{\mathcal{B}_{2^kr}(y_0)} U^2 d\mu \right )^\frac{1}{2} \\ & \text{ } + \left (\dashint_{\mathcal{B}_{2r}(y_0)} G^{q_0} d\mu \right )^\frac{1}{q_0} + \sum_{k=1}^\infty 2^{-k(s-\theta)} \left ( \dashint_{\mathcal{B}_{2^kr}(y_0)} G^2 d\mu \right )^\frac{1}{2} \Bigg ).
	\end{align*}
	Finally, by the Cauchy-Schwarz inequality, (\ref{comparable}) and (\ref{ineq99}), for $I_3$ we have
	\begin{align*}
	I_3 \leq & \dashint_{B_r(x_0)} \dashint_{B_r(y_0)} |u(x)-u(y)|dydx \\
	\leq & \left (\dashint_{B_r(x_0)} \dashint_{B_r(y_0)} |u(x)-u(y)|^2 dydx \right )^\frac{1}{2} \\
	\leq & \left (\frac{C_1}{\textnormal{dist}(B_r(x_0),B_r(y_0))^{n-2\theta} \mu(\mathcal{B}_r(x_0,y_0))} \int_{B_r(x_0)} \int_{B_r(y_0)} |u(x)-u(y)|^2 dydx \right )^\frac{1}{2} \\
	\leq & C_6 \left (\dashint_{\mathcal{B}_r(x_0,y_0)} |u(x)-u(y)|^2 d\mu \right )^\frac{1}{2} \\
	\leq & C_7 \textnormal{dist}(B_r(x_0),B_r(y_0))^{s+\theta} \left (\dashint_{\mathcal{B}_r(x_0,y_0)} U^2 d\mu \right )^\frac{1}{2},
	\end{align*}
	where $C_6=C_6(n,m,\theta) \geq 1$ and $C_7=C_7(n,m,\theta) \geq 1$. The claim now follows by combining the last five displays, so that the proof is finished.
\end{proof}
\subsection{Off-diagonal good-$\lambda$ inequalities}
In what follows, fix some $\varepsilon \in (0,1)$ to be chosen small enough and set 
\begin{equation} \label{Neps}
N_{\varepsilon,q}:=\frac{C_{nd} C_{s,\theta} N_d 10^{10n}}{ \varepsilon^{1/q^\star}},
\end{equation}
where $N_d=N_d(n,s,\theta,q,\Lambda) \geq 1$ is given by Lemma \ref{mfuse},  $C_{nd}=C_{nd}(n,s,\theta,\Lambda,m,q) \geq 1$ is given by Proposition \ref{offdiagreverse} with $m$ to be chosen and 
\begin{equation} \label{Cst}
1 \leq C_{s,\theta}:= \sum_{k=1}^\infty 2^{-k(s-\theta)} < \infty.
\end{equation}
Moreover, for all $r \in \left (0,\frac{\sqrt{n}}{2} \right )$ and all $(x_0,y_0) \in \mathcal{Q}_1$ we define
\begin{equation} \label{distfct}
\widetilde \phi(r,x_0,y_0):= \frac{r}{\textnormal{dist}(B_\frac{r}{2}(x_0),B_\frac{r}{2}(y_0))}.
\end{equation}

\begin{lem} \label{mfusenondiag}
	For any $\lambda \geq \lambda_0$, $r \in \left (0,\frac{\sqrt{n}}{2} \right )$ and any point $(x_0,y_0) \in \mathcal{Q}_1$ satisfying $|x_0-y_0| \geq (3\sqrt{n}+1)r$ and
	\begin{equation} \label{ccllz}
	\mu \left ( \left \{(x,y) \in \mathcal{B}_{\frac{\sqrt{n}}{2}r}(x_0,y_0) \mid  {\mathcal{M}}_{\mathcal{B}_{4n}} (U^2)(x,y) > N_{\varepsilon,q}^2 \lambda^2 \right \} \right ) \geq \varepsilon \mu(\mathcal{B}_{\frac{r}{2}}(x_0,y_0)),
	\end{equation}
	we have
	\begin{align*}
	\mathcal{B}_{\frac{r}{2}}(x_0,y_0)
	\subset & \left \{ (x,y) \in \mathcal{B}_{\frac{r}{2}}(x_0,y_0) \mid  {\mathcal{M}}_{\mathcal{B}_{4n}} (U^2)(x,y) > \lambda^2 \right \} \\ 
	\cup & \left \{ (x,y) \in \mathcal{B}_{\frac{r}{2}}(x_0,y_0) \mid  {\mathcal{M}}_{\geq r,\mathcal{B}_{4n}} (U^2)(x,x) > 3^{n+2\theta}N_d^2 \widetilde \phi(r,x_0,y_0)^{-2(s+\theta)} \lambda^2 \right \} \\
	\cup & \left \{ (x,y) \in \mathcal{B}_{\frac{r}{2}}(x_0,y_0) \mid  {\mathcal{M}}_{\geq r,\mathcal{B}_{4n}} (U^2)(y,y) > 3^{n+2\theta}N_d^2 \widetilde \phi(r,x_0,y_0)^{-2(s+\theta)} \lambda^2 \right \} \\
	\cup & \left \{ (x,y) \in \mathcal{B}_{\frac{r}{2}}(x_0,y_0) \mid  {\mathcal{M}}_{\geq r,\mathcal{B}_{4n}} (G^{q_0})(x,x) > 3^{n+2\theta} \widetilde \phi(r,x_0,y_0)^{-q_0(s+\theta)} \lambda^{q_0} \right \} \\
	\cup & \left \{ (x,y) \in \mathcal{B}_{\frac{r}{2}}(x_0,y_0) \mid  {\mathcal{M}}_{\geq r,\mathcal{B}_{4n}} (G^{q_0})(y,y) > 3^{n+2\theta} \widetilde \phi(r,x_0,y_0)^{-q_0(s+\theta)} \lambda^{q_0} \right \}.
	\end{align*}
\end{lem}

\begin{proof}
	Assume that (\ref{ccllz}) holds, but that the conclusion is false, so that there exists a point $(x^\prime,y^\prime) \in \mathcal{B}_{\frac{r}{2}}(x_0,y_0)$ such that 
	\begin{align*}
	& {\mathcal{M}}_{\mathcal{B}_{4n}}(U^2)(x^\prime,y^\prime) \leq \lambda^2, \\ &  {\mathcal{M}}_{\geq r,\mathcal{B}_{4n}} (U^2)(x^\prime,x^\prime) \leq 3^{n+2\theta}N_d^2 \widetilde \phi(r,x_0,y_0)^{-2(s+\theta)} \lambda^2, \quad
	 {\mathcal{M}}_{\geq r,\mathcal{B}_{4n}} (U^2)(y^\prime,y^\prime) \leq 3^{n+2\theta}N_d^2 \widetilde \phi(r,x_0,y_0)^{-2(s+\theta)} \lambda^2 \\ &  {\mathcal{M}}_{\geq r,\mathcal{B}_{4n}} (G^{q_0})(x^\prime,x^\prime) \leq 3^{n+2\theta} \widetilde \phi(r,x_0,y_0)^{-{q_0}(s+\theta)} \lambda^{q_0}, \quad
	 {\mathcal{M}}_{\geq r,\mathcal{B}_{4n}} (G^{q_0})(y^\prime,y^\prime) \leq 3^{n+2\theta} \widetilde \phi(r,x_0,y_0)^{-{q_0}(s+\theta)} \lambda^{q_0}.
	\end{align*}
	Therefore, for any $\rho \geq r$ we have
	\begin{equation} \label{lumhn}
	\dashint_{\mathcal{B}_\rho(x^\prime,y^\prime)}\chi_{\mathcal{B}_{4n}} U^2 d\mu \leq \lambda^2 ,
	\end{equation}
	\begin{equation} \label{lumhn1}
	\begin{aligned}
	&\dashint_{\mathcal{B}_\rho(x^\prime)} \chi_{\mathcal{B}_{4n}} U^2 d\mu \leq 3^{n+2\theta}N_d^2 \widetilde \phi(r,x_0,y_0)^{-{2}(s+\theta)} \lambda^2, \\ & \dashint_{\mathcal{B}_\rho(y^\prime)} \chi_{\mathcal{B}_{4n}} U^2 d\mu \leq 3^{n+2\theta}N_d^2 \widetilde \phi(r,x_0,y_0)^{-{2}(s+\theta)}\lambda^2
	\end{aligned}
	\end{equation} 
	and
	\begin{equation} \label{lumhn1x}
	\begin{aligned}
	& \dashint_{\mathcal{B}_\rho(x^\prime)} \chi_{\mathcal{B}_{4n}} G^{q_0} d\mu \leq 3^{n+2\theta}\widetilde \phi(r,x_0,y_0)^{-{q_0}(s+\theta)} \lambda^{q_0}, \\ & \dashint_{\mathcal{B}_\rho(y^\prime)} \chi_{\mathcal{B}_{4n}} G^{q_0} d\mu \leq 3^{n+2\theta} \widetilde \phi(r,x_0,y_0)^{-{q_0}(s+\theta)}\lambda^{q_0}
	\end{aligned}
	\end{equation} 
	Since for any $\rho \geq r$ we have $\mathcal{B}_\rho(x_0,y_0) \subset \mathcal{B}_{2\rho}(x^\prime,y^\prime) \subset \mathcal{B}_{3\rho}(x_0,y_0)$, from (\ref{lumhn}) we deduce
	\begin{equation} \label{Uoffdiag}
	\dashint_{\mathcal{B}_\rho(x_0,y_0)} \chi_{\mathcal{B}_{4n}} U^2 d\mu \leq \frac{\mu(\mathcal{B}_{2\rho}(x^\prime,y^\prime))}{\mu(\mathcal{B}_\rho(x_0,y_0))} \dashint_{\mathcal{B}_{2\rho}(x^\prime)} \chi_{\mathcal{B}_{4n}} U^2 d\mu \leq \frac{\mu(\mathcal{B}_{3\rho}(x_0,y_0))}{\mu(\mathcal{B}_\rho(x_0,y_0))} \leq 3^{n+2\theta} \lambda^2.
	\end{equation}
	Note also that for any $\rho \geq r$ we have $\mathcal{B}_\rho(x_0) \subset \mathcal{B}_{2\rho}(x^\prime)$.
	Together with (\ref{lumhn1}) this observation yields
	\begin{equation} \label{Ux0}
	\dashint_{\mathcal{B}_\rho(x_0)}\chi_{\mathcal{B}_{4n}} U^2 d\mu \leq \frac{\mu(\mathcal{B}_{2\rho}(x^\prime))}{\mu(\mathcal{B}_\rho(x_0))} \dashint_{\mathcal{B}_{2\rho}(x^\prime)}\chi_{\mathcal{B}_{4n}} U^2 d\mu \leq 6^{n+2\theta}N_d^2 \widetilde \phi(r,x_0,y_0)^{-2(s+\theta)} \lambda^2 
	\end{equation}
	and similarly
	\begin{equation} \label{Gx0}
	\dashint_{\mathcal{B}_\rho(x_0)}\chi_{\mathcal{B}_{4n}} G^{q_0} d\mu \leq \frac{\mu(\mathcal{B}_{2\rho}(x^\prime))}{\mu(\mathcal{B}_\rho(x_0))} \dashint_{\mathcal{B}_{2\rho}(x^\prime)}\chi_{\mathcal{B}_{4n}} G^{q_0} d\mu \leq 6^{n+2\theta} \widetilde \phi(r,x_0,y_0)^{-{q_0}(s+\theta)} \lambda^{q_0}.
	\end{equation}
	By the same reasoning, (\ref{Ux0}) and (\ref{Gx0}) clearly hold also with $x_0$ replaced by $y_0$.
	Next, we claim that 
	\begin{equation} \label{inclusion1}
	\begin{aligned}
	& \left \{ (x,y) \in \mathcal{B}_{\frac{\sqrt{n}}{2}r}(x_0,y_0) \mid  {\mathcal{M}}_{\mathcal{B}_{4n}}( U^2 )(x,y) > N_{\varepsilon,q}^2 \lambda^2 \right \} \\ \subset & \left \{ (x,y) \in \mathcal{B}_{\frac{ \sqrt{n}}{2}r}(x_0,y_0) \mid  {\mathcal{M}}_{\mathcal{B}_{\frac{3\sqrt{n}}{2}r}(x_0,y_0)} ( U^2 )(x,y) > N_{\varepsilon,q}^2 \lambda^2 \right \}. 
	\end{aligned}
	\end{equation}
	To see this, assume that 
	\begin{equation} \label{menge1}
	(x_1,y_1) \in \left \{ x \in \mathcal{B}_{\frac{\sqrt{n}}{2}r}(x_0,y_0) \mid  {\mathcal{M}}_{\mathcal{B}_{\frac{3\sqrt{n}}{2}r}(x_0,y_0)} ( U^2 ) (x,y) \leq N_{\varepsilon,q}^2 \lambda^2 \right \}. 
	\end{equation}
	For $ \rho < \sqrt{n}r$, we have $\mathcal{B}_\rho (x_1,y_1) \subset \mathcal{B}_{\sqrt{n}r}(x_1,y_1) \subset \mathcal{B}_{\frac{3\sqrt{n}}{2}r}(x_0,y_0)$, so that together with $(\ref{menge1})$ we deduce 
	\begin{align*}
	\dashint_{\mathcal{B}_\rho (x_1,y_1)} U^2 d\mu & \leq  {\mathcal{M}}_{\mathcal{B}_{\frac{3\sqrt{n}}{2}r}(x_0,y_0)} ( U^2 )(x_1,y_1) \leq N_{\varepsilon,q}^2 \lambda^2. 
	\end{align*}
	On the other hand, for $\rho \geq \sqrt{n}r$ we have $ \mathcal{B}_\rho (x_1,y_1) \subset \mathcal{B}_{3 \rho}(x^\prime,y^\prime) \subset \mathcal{B}_{5 \rho}(x_1,y_1)$, so that $(\ref{lumhn})$ implies 
	\begin{align*}
	\dashint_{\mathcal{B}_\rho(x_1,y_1)} \chi_{\mathcal{B}_{4n}} U^2 d\mu \leq \frac{\mu(\mathcal{B}_{3 \rho} (x^\prime,y^\prime))}{\mu(\mathcal{B}_\rho(x_1,y_1))} \dashint_{\mathcal{B}_{3 \rho} (x^\prime,y^\prime)} \chi_{\mathcal{B}_{4n}} U^2 d\mu \leq & \frac{\mu(\mathcal{B}_{5 \rho} (x_1,y_1))}{\mu(\mathcal{B}_\rho(x_1,y_1))} \dashint_{\mathcal{B}_{3 \rho} (x^\prime,y^\prime)} \chi_{\mathcal{B}_{4n}} U^2 d\mu \\ \leq & 5^{n+2\theta} \lambda^2 \leq N_{\varepsilon,q}^2 \lambda^2.
	\end{align*}
	Thus, we have $$ (x_1,y_1) \in \left \{ (x,y) \in \mathcal{B}_r(x_0,y_0) \mid  {\mathcal{M}}_{\mathcal{B}_{4n}}( U^2 )(x,y) \leq N_{\varepsilon,q}^2 \lambda^2 \right \} ,$$ 
	which implies $(\ref{inclusion1})$.
As in the proof of Lemma \ref{apppl}, let $m \in \mathbb{N}$ be determined by $2^{m-1} r < \sqrt{n} \leq 2^{m} r$, note that $m \geq 2$. Then for any $k < m$, by (\ref{Ux0}) and (\ref{Gx0}) we have
\begin{equation} \label{UGx}
\begin{aligned}
&\dashint_{\mathcal{B}_{2^k\frac{3\sqrt{n}}{2} r}(x_0)} U^2 d\mu \leq 6^{n+2\theta}N_d^2 \lambda^2 \widetilde \phi(r,x_0,y_0)^{-2(s+\theta)}, \\ &\dashint_{\mathcal{B}_{2^k\frac{3\sqrt{n}}{2} r}(x_0)} G^{q_0} d\mu \leq 6^{n+2\theta}\lambda^{q_0} \delta^{q_0} \widetilde \phi(r,x_0,y_0)^{-q_0(s+\theta)}.
\end{aligned}
\end{equation}
Moreover, in view of (\ref{tailcontrol}), the inclusions $$\mathcal{B}_{2^k \frac{n}{2}}(x_0) \subset \mathcal{B}_{2^{k+m-1} \frac{3\sqrt{n}}{2} r}(x_0) \subset \mathcal{B}_{2^k \frac{3n}{2}}(x_0) \subset \mathcal{B}_{2^k 4n}$$ and the fact that $\widetilde \phi(r,x_0,y_0) \leq 1$, we have
\begin{align*}
& \sum_{k=m}^\infty 2^{-k(s-\theta)} \left ( \dashint_{\mathcal{B}_{2^k \frac{3\sqrt{n}}{2}r}(x_0)} U^2 d\mu \right )^\frac{1}{2} \\ = & 2^{-(m-1)(s-\theta)} \sum_{k=1}^\infty 2^{-k(s-\theta)} \left ( \dashint_{\mathcal{B}_{2^{k+m-1}\frac{3\sqrt{n}}{2}r}(x_0)} U^2 d\mu \right )^\frac{1}{2} \\
\leq & \sum_{k=1}^\infty 2^{-k(s-\theta)} \left (\frac{\mu(\mathcal{B}_{2^k 4n})}{\mu \left (\mathcal{B}_{2^k \frac{n}{2}} \right)} \dashint_{\mathcal{B}_{2^k 4n}} U^2 d\mu \right )^\frac{1}{2} \\
\leq & 8^{\frac{n}{2}+\theta} \lambda_0 \leq 8^{\frac{n}{2}+\theta} \lambda_0 \widetilde \phi(r,x_0,y_0)^{-(s+\theta)}. 
\end{align*}
Together with (\ref{UGx}) and the assumption that $\lambda \geq \lambda_0$, we arrive at
\begin{equation} \label{Ulambdax}
\begin{aligned}
& \sum_{k=1}^\infty 2^{-k(s-\theta)} \left ( \dashint_{\mathcal{B}_{2^k \frac{3\sqrt{n}}{2}r}(x_0)} U^2 d\mu \right )^\frac{1}{2} \\ \leq & \sum_{k=1}^{m-1} 2^{-k(s-\theta)} \left ( \dashint_{\mathcal{B}_{2^k \frac{3\sqrt{n}}{2} r}(x_0)} U^2 d\mu \right )^\frac{1}{2} + \sum_{k=m}^\infty 2^{-k(s-\theta)} \left ( \dashint_{\mathcal{B}_{2^k \frac{3\sqrt{n}}{2}r}(x_0)} U^2 d\mu \right )^\frac{1}{2} \\
\leq & 6^{\frac{n}{2}+\theta} C_{s,\theta} N_d^2 \widetilde \phi(r,x_0,y_0)^{-(s+\theta)} \lambda + 8^{\frac{n}{2}+\theta} \widetilde \phi(r,x_0,y_0)^{-(s+\theta)}\lambda_0 \\ \leq & 8^{\frac{n}{2}+\theta} C_{s,\theta}N_d^2 \widetilde \phi(r,x_0,y_0)^{-(s+\theta)} \lambda.
\end{aligned}
\end{equation}
By a similar reasoning as above, we have
\begin{align*}
\sum_{k=m}^\infty 2^{-k(s-\theta)} \left ( \dashint_{\mathcal{B}_{2^k \frac{3\sqrt{n}}{2}r}(x_0)} G^2 d\mu \right )^\frac{1}{2} 
\leq & 8^{\frac{n}{2}+\theta} \sum_{k=1}^\infty 2^{-k(s-\theta)} \left ( \dashint_{\mathcal{B}_{2^k 4n}} G^2 d\mu \right )^\frac{1}{2} \\ \leq & 8^{\frac{n}{2}+\theta} \widetilde \phi(r,x_0,y_0)^{-2(s+\theta)} \lambda_0
\end{align*}
and therefore along with H\"older's inequality
\begin{equation} \label{Glambdax}
\begin{aligned}
& \sum_{k=1}^\infty 2^{-k(s-\theta)} \left ( \dashint_{\mathcal{B}_{2^k \frac{3\sqrt{n}}{2}r}(x_0)} G^2 d\mu \right )^\frac{1}{2} \\ \leq & \sum_{k=1}^{m-1} 2^{-k(s-\theta)} \left ( \dashint_{\mathcal{B}_{2^k \frac{3\sqrt{n}}{2} r}(x_0)} G^{q_0} d\mu \right )^\frac{1}{q_0} + \sum_{k=m}^\infty 2^{-k(s-\theta)} \left ( \dashint_{\mathcal{B}_{2^k \frac{3\sqrt{n}}{2} r}(x_0)} G^2 d\mu \right )^\frac{1}{2} \\ 
\leq & 8^{\frac{n}{2}+\theta} C_{s,\theta} \widetilde \phi(r,x_0,y_0)^{-(s+\theta)} \lambda.
\end{aligned}
\end{equation}
Again, by the same arguments as above (\ref{Ulambdax}) and (\ref{Glambdax}) clearly also hold for $x_0$ replaced by $y_0$.
	Therefore, together with the weak $\frac{q^\star}{2}-\frac{q^\star}{2}$ estimate for the Hardy-Littlewood maximal function, Proposition \ref{offdiagreverse} with $m=\frac{1}{3\sqrt{n}}$, (\ref{Uoffdiag}), (\ref{Ulambdax}), (\ref{Gx0}), (\ref{Glambdax}) and (\ref{Neps}), we arrive at
	\begin{align*}
	& \mu \left (\left \{ (x,y) \in \mathcal{B}_{\frac{\sqrt{n}}{2}r}(x_0,y_0) \mid  {\mathcal{M}}_{\mathcal{B}_{4n}}( U^2 )(x,y) > N_{\varepsilon,q}^2 \lambda^2 \right \} \right ) \\ \leq & \mu \left (\left \{ (x,y) \in \mathcal{B}_{\frac{ \sqrt{n}}{2}r}(x_0,y_0) \mid  {\mathcal{M}}_{\mathcal{B}_{\frac{3\sqrt{n}}{2}r}(x_0,y_0)} ( U^2 )(x,y) > N_{\varepsilon,q}^2 \lambda^2 \right \} \right ) \\
	\leq & N_{\varepsilon,q}^{-q^\star} \lambda^{-q^\star} \int_{\mathcal{B}_{\frac{3\sqrt{n}}{2}r}(x_0,y_0)} U^{q^\star} d\mu \\
	\leq & N_{\varepsilon,q}^{-q^\star} \lambda^{-q^\star}3^{q^\star} C_{nd}^{q^\star} \mu \left(\mathcal{B}_{\frac{3\sqrt{n}}{2}r}(x_0,y_0) \right ) \Bigg [ \left (\dashint_{\mathcal{B}_{\frac{3\sqrt{n}}{2}r}(x_0,y_0)} U^{2} d\mu \right )^\frac{q^\star}{2} \\
	& + \left ( \frac{3\sqrt{n}r/2}{\textnormal{dist} \left (B_{\frac{3\sqrt{n}}{2}r}(x_0),B_{\frac{3\sqrt{n}}{2}r}(y_0) \right )} \right )^{q^\star (s+\theta)} \Bigg ( \sum_{k=1}^\infty 2^{-k(s-\theta)} \left ( \dashint_{\mathcal{B}_{2^k\frac{3\sqrt{n}}{2}r}(x_0)} U^2 d\mu \right )^\frac{1}{2} \\ & \text{ }+ \left (\dashint_{\mathcal{B}_{3 \sqrt{n} r}(x_0)} G^{q_0} d\mu \right )^\frac{1}{q_0} + \sum_{k=1}^\infty 2^{-k(s-\theta)} \left ( \dashint_{\mathcal{B}_{2^k\frac{3\sqrt{n}}{2}r}(x_0)} G^2 d\mu \right )^\frac{1}{2} \Bigg )^{q^\star} \\
	& + \left ( \frac{3\sqrt{n}r/2}{\textnormal{dist} \left (B_{\frac{3\sqrt{n}}{2}r}(x_0),B_{\frac{3\sqrt{n}}{2}r}(y_0) \right )} \right )^{q^\star (s+\theta)} \Bigg ( \sum_{k=1}^\infty 2^{-k(s-\theta)} \left ( \dashint_{\mathcal{B}_{2^k \frac{3\sqrt{n}}{2}r}(y_0)} U^2 d\mu \right )^\frac{1}{2} \\ & \text{ }+ \left (\dashint_{\mathcal{B}_{3\sqrt{n}r}(y_0)} G^{q_0} d\mu \right )^\frac{1}{q_0} + \sum_{k=1}^\infty 2^{-k(s-\theta)} \left ( \dashint_{\mathcal{B}_{2^k \frac{3\sqrt{n}}{2}r}(y_0)} G^2 d\mu \right )^\frac{1}{2} \Bigg )^{q^\star} \Bigg ] \\
	\leq & N_{\varepsilon,q}^{-q\star} \lambda^{-q\star} 3^{q^\star} C_{nd}^{q^\star} \left (3\sqrt{n} \right )^{n+2\theta} \mu \left(\mathcal{B}_{\frac{r}{2}}(x_0,y_0) \right ) \bigg (3^{(\frac{n}{2}+\theta)q^\star} \lambda^{q^\star} \\
	& +6^{q^\star} (9n)^{q^\star(s+\theta)} \widetilde \phi(r,x_0,y_0)^{q^\star (s+\theta)} 8^{(\frac{n}{2}+\theta)q^\star} C_{s,\theta}^{q^\star} N_d^{q^\star} \widetilde \phi(r,x_0,y_0)^{-q^\star(s+\theta)} \lambda^{q^\star} \bigg ) \\
	< & \varepsilon \mu \left(\mathcal{B}_{\frac{r}{2}}(x_0,y_0) \right ),
	\end{align*}
	which contradicts (\ref{ccllz}) and thus finishes the proof.
\end{proof}

Since we are going to use a Calderón-Zygmund cube decomposition, we next prove a version of the previous Lemma with balls replaced by cubes. For notational convenience, in analogy to the quantity $\widetilde \phi(r,x_0,y_0)$ defined in (\ref{distfct}), for any $r \in \left (0,\frac{\sqrt{n}}{2} \right )$ and all $x_0,y_0 \in \mathbb{R}^n$ with $|x_0-y_0|>\sqrt{n}r$, we define the quantity
\begin{equation} \label{phidef}
\phi(r,x_0,y_0):=\frac{r}{\textnormal{dist}(Q_r(x_0),Q_r(y_0))}.
\end{equation}
Note that since $B_{r/2}(x_0) \subset Q_r(x_0)$ and $B_{r/2}(y_0) \subset Q_r(y_0)$, the two quantities are related as follows
\begin{equation} \label{phirel}
\widetilde \phi(r,x_0,y_0) \leq \phi(r,x_0,y_0).
\end{equation}

\begin{cor} \label{mfusenondiagcubes}
	For any $\lambda \geq \lambda_0$, $r \in \left (0,\frac{\sqrt{n}}{2} \right )$ and any point $(x_0,y_0) \in \mathcal{Q}_1$ satisfying $|x_0-y_0| \geq (3\sqrt{n}+1)r$ and
	\begin{equation} \label{ccllx64}
	\mu \left ( \left \{(x,y) \in \mathcal{Q}_{r}(x_0,y_0) \mid  {\mathcal{M}}_{\mathcal{B}_{4n}}(U^2)(x,y) > N_{\varepsilon,q}^2 \lambda^2 \right \} \right ) > \varepsilon \mu(\mathcal{Q}_r(x_0,y_0)),
	\end{equation}
	we have
	\begin{align*}
	& \mu(\mathcal{Q}_{r}(x_0,y_0)) \\
	\leq & (\sqrt{n})^{n+2\theta} \bigg( \mu \left (\left \{ (x,y) \in \mathcal{Q}_r(x_0,y_0) \mid  {\mathcal{M}}_{\mathcal{B}_{4n}} (U^2)(x,y) > \lambda^2 \right \} \right ) \\ 
	& + \phi(r,x_0,y_0)^{n-2 \theta} \mu \left (\left \{ (x,y) \in \mathcal{Q}_r(x_0) \mid  {\mathcal{M}}_{\mathcal{B}_{4n}} (U^2)(x,y) > N_d^2 \phi(r,x_0,y_0)^{-2(\theta+s)} \lambda^2 \right \} \right ) \\
	& + \phi(r,x_0,y_0)^{n-2 \theta} \mu \left (\left \{ (x,y) \in \mathcal{Q}_r(y_0) \mid  {\mathcal{M}}_{\mathcal{B}_{4n}} (U^2)(x,y) > N_d^2 \phi(r,x_0,y_0)^{-2(\theta+s)} \lambda^2 \right \} \right ) \\
	& + \phi(r,x_0,y_0)^{n-2 \theta} \mu \left (\left \{ (x,y) \in \mathcal{Q}_r(x_0) \mid  {\mathcal{M}}_{\mathcal{B}_{4n}} (G^{q_0})(x,y) > \phi(r,x_0,y_0)^{-{q_0}(\theta+s)} \lambda^{q_0} \right \} \right ) \\
	& + \phi(r,x_0,y_0)^{n-2 \theta} \mu \left (\left \{ (x,y) \in \mathcal{Q}_r(y_0) \mid  {\mathcal{M}}_{\mathcal{B}_{4n}} (G^{q_0})(x,y) > \phi(r,x_0,y_0)^{-{q_0}(\theta+s)}\lambda^{q_0} \right \} \right ) \bigg ).
	\end{align*}
\end{cor}

\begin{proof}
	First of all, note that $(x_0,y_0) \in \mathcal{Q}_{1} \subset \mathcal{B}_{\frac{\sqrt{n}}{2}}$.
	By the assumption $(\ref{ccllx64})$ and the fact that $\mathcal{Q}_{r}(x_0,y_0) \subset \mathcal{B}_{\frac{\sqrt{n}}{2}r}(x_0,y_0)$, we have 
	\begin{align*}
	& \mu \left ( \left \{(x,y) \in \mathcal{B}_{\frac{\sqrt{n}}{2}r}(x_0,y_0) \mid  {\mathcal{M}}_{\mathcal{B}_{4n}}(U^2)(x,y) > N_{\varepsilon,q}^2 \right \} \right ) \\
	\geq &\mu \left ( \left \{(x,y) \in \mathcal{Q}_{r}(x_0,y_0) \mid  {\mathcal{M}}_{\mathcal{B}_{4n}}(U^2)(x,y) > N_{\varepsilon,q}^2 \right \} \right ) \\
	> & \varepsilon \mu(\mathcal{Q}_r(x_0,y_0)) \geq \varepsilon \mu(\mathcal{B}_{r/2}(x_0,y_0)).
	\end{align*}
	Therefore, the assumption $(\ref{ccllz})$ from Lemma \ref{mfusenondiag} is satisfied, so that by the volume doubling property of $\mu$, Lemma \ref{mfusenondiag} and the inclusion $\mathcal{B}_{r/2}(x_0,y_0) \subset \mathcal{Q}_r(x_0,y_0)$, we obtain
	\begin{align*}
	& \mu(\mathcal{Q}_{r}(x_0,y_0)) \leq \mu(\mathcal{B}_{\frac{ \sqrt{n}}{2}r}(x_0,y_0)) \\
	\leq & (\sqrt{n})^{n+2\theta} \mu(\mathcal{B}_{r/2}(x_0,y_0)) \\
	\leq & (\sqrt{n})^{n+2\theta} \bigg( \mu \left (\left \{ (x,y) \in \mathcal{Q}_{r}(x_0,y_0) \mid  {\mathcal{M}}_{\mathcal{B}_{4n}} (U^2)(x,y) > \lambda^2 \right \} \right ) \\ 
	& + \mu \left (\left \{ (x,y) \in \mathcal{B}_{r/2}(x_0,y_0) \mid  {\mathcal{M}}_{\geq r,\mathcal{B}_{4n}} (U^2)(x,x) > 3^{n+2\theta} N_d^2 \widetilde \phi(r,x_0,y_0)^{-2(\theta+s)} \lambda^2 \right \} \right ) \\
	& + \mu \left (\left \{ (x,y) \in \mathcal{B}_{r/2}(x_0,y_0) \mid  {\mathcal{M}}_{\geq r,\mathcal{B}_{4n}} (U^2)(y,y) > 3^{n+2\theta} N_d^2 \widetilde \phi(r,x_0,y_0)^{-2(\theta+s)} \lambda^2 \right \} \right ) \\
	& + \mu \left (\left \{ (x,y) \in \mathcal{B}_{r/2}(x_0,y_0) \mid  {\mathcal{M}}_{\geq r,\mathcal{B}_{4n}} (G^{q_0})(x,x) > 3^{n+2\theta} \widetilde \phi(r,x_0,y_0)^{-q_0(\theta+s)} \lambda^{q_0} \right \} \right ) \\
	& + \mu \left (\left \{ (x,y) \in \mathcal{B}_{r/2}(x_0,y_0) \mid  {\mathcal{M}}_{\geq r,\mathcal{B}_{4n}} (G^{q_0})(y,y) > 3^{n+2\theta} \widetilde \phi(r,x_0,y_0)^{-q_0(\theta+s)} \lambda^{q_0} \right \} \right ) \bigg ).
	\end{align*}
	We proceed by further estimating the second term on the right-hand side of the last display and note the last three terms can be estimated similarly.
	Set $$F_1:=\left \{ x \in B_{r/2}(x_0) \mid  {\mathcal{M}}_{\geq r,\mathcal{B}_{4n}} (U^2)(x,x) > 3^{n+2\theta} N_d^2 \widetilde \phi(r,x_0,y_0)^{-2(\theta+s)} \lambda^2 \right \}.$$ 
	We have
	\begin{align*}
	& \mu \left (\left \{ (x,y) \in \mathcal{B}_{r/2}(x_0,y_0) \mid  {\mathcal{M}}_{\geq r,\mathcal{B}_{4n}} (U^2)(x,x) > 3^{n+2\theta} N_d^2 \widetilde \phi(r,x_0,y_0)^{-2(\theta+s)} \lambda^2 \right \} \right ) \\
	= & \mu \left (F_1 \times B_{r/2}(y_0) \right ) \\
	= & \int_{F_1} \int_{B_{r/2}(y_0)} \frac{dydx}{|x-y|^{n-2\theta}} \\
	\leq & \textnormal{dist}(B_{r/2}(x_0),B_{r/2}(y_0))^{-(n-2\theta)} |F_1| |B_{r/2}(y_0)| \\
	= & \textnormal{dist}(B_{r/2}(x_0),B_{r/2}(y_0))^{-(n-2\theta)} |F_1| |B_{r/2}(x_0)| \\
	\leq & \textnormal{dist}(B_{r/2}(x_0),B_{r/2}(y_0))^{-(n-2\theta)} r^{n-2\theta} \int_{F_1} \int_{B_{r/2}(x_0)} \frac{dydx}{|x-y|^{n-2\theta}} \\
	= & \widetilde \phi(r,x_0,y_0)^{n-2\theta} \mu(F_1 \times B_{r/2}(x_0)) \\
	= & \widetilde \phi(r,x_0,y_0)^{n-2\theta} \mu \left ( \left \{ (x,y) \in \mathcal{B}_{r/2}(x_0) \mid  {\mathcal{M}}_{\geq r,\mathcal{B}_{4n}} (U^2)(x,x) > 3^{n+2\theta} N_d^q \widetilde \phi(r,x_0,y_0)^{-2(\theta+s)} \lambda^2 \right \} \right ) .
	\end{align*}
	In order to further estimate the right-hand side, we claim that for all $x,y \in B_{r/2}(x_0)$ we have
	\begin{equation} \label{maxfrel}
	 {\mathcal{M}}_{\geq r,\mathcal{B}_{4n}} (U^2)(x,x) \leq 3^{n+2\theta}  {\mathcal{M}}_{\mathcal{B}_{4n}} (U^2)(x,y).
	\end{equation}
	Indeed, for any $\rho \geq r$ we have $\mathcal{B}_{\rho}(x) \subset \mathcal{B}_{2\rho}(x,y) \subset \mathcal{B}_{3\rho}(x)$ and therefore
	$$\dashint_{\mathcal{B_\rho}(x)} \chi_{\mathcal{B}_{4n}} U^2d\mu \leq \frac{\mu(\mathcal{B}_{3\rho}(x))}{\mu(B_\rho(x))} \dashint_{\mathcal{B}_{2\rho}(x,y)} \chi_{\mathcal{B}_{4n}} U^2 d\mu \leq 3^{n+2\theta} \widetilde{\mathcal{M}}_{\mathcal{B}_{4n}} (U^2)(x,y), $$
	which proves (\ref{maxfrel}).
	By (\ref{maxfrel}) and the display before that, we arrive at
	\begin{align*}
	& \mu \left (\left \{ (x,y) \in \mathcal{B}_{r/2}(x_0,y_0) \mid  {\mathcal{M}}_{\geq r,\mathcal{B}_{4n}} (U^2)(x,x) > 3^{n+2\theta} N_d^2 \widetilde \phi(r,x_0,y_0)^{-2(\theta+s)} \right \} \right ) \\
	\leq & \widetilde \phi(r,x_0,y_0)^{n-2\theta} \mu \left (\left \{ (x,y) \in \mathcal{B}_{r/2}(x_0) \mid  {\mathcal{M}}_{\geq r,\mathcal{B}_{4n}} (U^2)(x,x) > 3^{n+2\theta} N_d^2 \widetilde \phi(r,x_0,y_0)^{-2(\theta+s)} \right \} \right )\\
	\leq & \widetilde \phi(r,x_0,y_0)^{n-2\theta} \mu \left (\left \{ (x,y) \in \mathcal{B}_{r/2}(x_0) \mid  {\mathcal{M}}_{\mathcal{B}_{4n}} (U^2)(x,y) > N_d^2 \widetilde \phi(r,x_0,y_0)^{-2(\theta+s)} \right \} \right )\\
	\leq & \phi(r,x_0,y_0)^{n-2\theta} \mu \left (\left \{ (x,y) \in \mathcal{Q}_{r}(x_0) \mid  {\mathcal{M}}_{\mathcal{B}_{4n}} (U^2)(x,y) > N_d^2 \phi(r,x_0,y_0)^{-2(\theta+s)} \right \} \right ),
	\end{align*}
	where we used the inequality (\ref{phirel}) and the inclusion $\mathcal{B}_{r/2}(x_0) \subset \mathcal{Q}_r(x_0)$ in order to obtain the last inequality. As mentioned, by observing that the last three terms on the right-hand side in the second display of the proof can be estimated similarly, the proof is finished.
\end{proof}

\section{A covering argument} \label{cover}
In order to proceed, we split the unit cube $Q_1$ in $\mathbb{R}^n$ into dyadic cubes as follows. First, we split $Q_1$ into $2^n$ cubes of sidelength $\frac{1}{2}$. Next, we split each of the resulting cubes into $2^n$ cubes of sidelength $\frac{1}{2^2}=\frac{1}{4}$ and iterate this process. The resulting family of cubes are called dyadic cubes. By using the same procedure with $n$ replaced by $2n$, we can also split the unit cube $\mathcal{Q}_1=Q_1 \times Q_1$ in $\mathbb{R}^{2n}$ into a family of dyadic cubes. We observe that any dyadic cube $\mathcal{K} \subset \mathcal{Q}_1$ of the resulting dyadic cubes can be written as $\mathcal{K} = K_1 \times K_2$, where $K_1$ and $K_2$ are $n$-dimensional dyadic cubes contained in $Q_1$. By construction, any such dyadic cube $\mathcal{K}$ has sidelength $2^{-k(\mathcal{K})}$ for some non-negative integer $k(\mathcal{K}) \geq 0$. Moreover, for $k \geq 1$ any such dyadic cube $\mathcal{K}$ with sidelength $2^{-k(\mathcal{K})}$ is contained in exactly one dyadic cube $\widetilde{\mathcal{K}}$ with sidelength $2^{-k(\mathcal{K})+1}$, we call $\widetilde{\mathcal{K}}$ the predecessor of $\mathcal{K}$. \newline
We need the following version of the Calderón-Zygmund decomposition, which roughly speaking shows that a subset of $\mathcal{Q}_1$ with small enough density can be covered by a sequence of dyadic cubes with density properties that are desirable for our purposes. For a proof we refer to \cite[Lemma 1.1]{CaffarelliPeral}, where the result is proved with respect to the Lebesgue measure instead of $\mu$. However, the proof also works for the doubling measure $\mu$ by taking into account that the Lebesgue differentiation theorem (see Proposition \ref{LDT}) holds also with respect to $\mu$. 

\begin{lem} \label{CaldZyg}
	Let $E \subset \mathcal{Q}_1$ be a measurable set satisfying
	\begin{equation} \label{smalldensity}
	0<\mu(E)<\varepsilon \mu(\mathcal{Q}_1) \quad \text{for some } \varepsilon \in (0,1).
	\end{equation}
	Then there exists a countable family $\mathcal{U}_{\lambda}$ of dyadic cubes obtained from $\mathcal{Q}_1$, such that
	\begin{equation} \label{dyadiccover}
	\mu \left (E \setminus \bigcup_{\mathcal{K} \in \mathcal{U}_{ \lambda}}\mathcal{K} \right ) =0
	\end{equation}
	and such that for any $\mathcal{K} \in \mathcal{U}_{\lambda}$ we have
	\begin{equation} \label{lowerdensity}
	\mu(E \cap \mathcal{K}) \geq \varepsilon \mu(\mathcal{K})
	\end{equation}
	and
	\begin{equation} \label{upperdensity}
	\mu(E \cap \widetilde{\mathcal{K}}) < \varepsilon \mu(\widetilde{\mathcal{K}}),
	\end{equation}
	where $\widetilde{\mathcal{K}}$ denotes the predecessor of $\mathcal{K}$.
\end{lem}

Next, fix some $\lambda \geq \lambda_0$ and consider the level set
\begin{equation} \label{E}
E:=\left \{(x,y) \in \mathcal{Q}_{1} \mid  {\mathcal{M}}_{\mathcal{B}_{4n}} (U^2)(x,y) > N_{\varepsilon,q}^2 \lambda^2 \right \}.
\end{equation}
Provided that (\ref{smalldensity}) is satisfied with respect to this $E$ and some $\varepsilon \in (0,1)$ which we will choose later, by Lemma \ref{CaldZyg} there exists a countable family $\mathcal{U}_{\lambda}$ of dyadic cubes obtained from $\mathcal{Q}_1$, such that (\ref{dyadiccover}), (\ref{lowerdensity}) and (\ref{upperdensity}) are satisfied with respect to $E$.
\par In order to treat the cubes of the family $\mathcal{U}_{\lambda}$ which are in some sense close enough to the diagonal, we also construct an auxiliary diagonal cover consisting of balls as follows.
For $x_0 \in Q_1$ and $r >0$, consider the quantity
\begin{align*}
\Psi_{\lambda}(x_0,r):= & \mu \left (\left \{(x,y) \in \mathcal{B}_r(x_0) \cap \mathcal{Q}_1 \mid  {\mathcal{M}}_{\mathcal{B}_{4n}}(U^2)(x,y) > N_{d}^2 \lambda^2 \right \} \right ) \\
& + \mu \left (\left \{(x,y) \in \mathcal{B}_{r}(x_0) \cap \mathcal{Q}_1 \mid  {\mathcal{M}}_{\mathcal{B}_{4n}} \left (G^{q_0} \right )(x,y) > \lambda^{q_0} \right \} \right ) ,
\end{align*}
where $q_0$ is given by (\ref{q0}).
Observe that since $N_d \leq N_{\varepsilon,q}$, for any $x_0 \in Q_1$ and any $r>0$ we have
\begin{equation} \label{NDD}
\mu(E \cap B_r(x_0)) \leq \Psi_{\lambda}(x_0,r).
\end{equation}
Now fix $\kappa \in (0,1)$ to be chosen later and consider the following subset of the diagonal in $\mathcal{Q}_1$
$$D_{\kappa \varepsilon}:=\left \{ (x,x) \in \mathcal{Q}_1 \mid \sup_{0<r < \sqrt{n}/2} \Psi_{\lambda}(x,r) \geq \kappa \varepsilon \mu(\mathcal{B}_r(x)) \right \}.$$
By the weak $1$-$1$ estimate for the Hardy-Littlewood maximal function (see Proposition \ref{Maxfun}) along with using that $\lambda \geq \lambda_0$, we obtain
\begin{equation} \label{smalldensity1}
\begin{aligned}
& \mu \left (\left \{(x,y) \in \mathcal{Q}_1 \mid  {\mathcal{M}}_{\mathcal{B}_{4n}}(U^2)(x,y) > N_d^2 \lambda^2 \right \} \right ) \\
& + \mu \left (\left \{(x,y) \in \mathcal{Q}_1 \mid  {\mathcal{M}}_{\mathcal{B}_{4n}} \left (G^{q_0} \right )(x,y) > \lambda^{q_0} \right \} \right ) \\
\leq & \frac{C_1}{N_{d}^2 \lambda^2} \int_{\mathcal{B}_{4n}} U^2 d\mu + \frac{C_1}{\lambda^{q_0}} \int_{\mathcal{B}_{4n}} G^{q_0} d\mu \\
\leq & \frac{C_2}{N_{d}^2 \lambda_0^2} \left (\sum_{k=1}^\infty 2^{-k(s-\theta)} \left (\dashint_{\mathcal{B}_{2^k 4n}} U^2 d\mu \right )^\frac{1}{2} \right )^2 + \frac{C_3}{\lambda_0^{q_0}} \dashint_{\mathcal{B}_{4n}} G^{q_0} d\mu < \kappa \varepsilon \mu(\mathcal{Q}_1),
\end{aligned}
\end{equation}
where all constants depend only on $n,s$ and $\theta$ and the last inequality is obtained by choosing $M_0$ large enough in (\ref{tailcontrol}).
Since $N_d \leq N_{\varepsilon,q}$, (\ref{smalldensity1}) in particular implies the condition (\ref{smalldensity}) with respect to the set $E$ defined in (\ref{E}), so that the family $\mathcal{U}_{ \lambda}$ of dyadic cubes as stated above indeed exists. \par
Since by (\ref{smalldensity1}) for any $x \in Q_1$ and any $r \geq \sqrt{n}/2$ we have
\begin{equation} \label{smallden}
\begin{aligned}
\Psi_{\lambda}(x,r) \leq & \mu \left (\left \{(x,y) \in \mathcal{Q}_1 \mid  {\mathcal{M}}_{\mathcal{B}_{4n}}(U^2)(x,y) > N_{d}^2 \lambda^2 \right \} \right ) \\
& + \mu \left (\left \{(x,y) \in \mathcal{Q}_1 \mid  {\mathcal{M}}_{\mathcal{B}_{4n}} \left (G^{q_0} \right )(x,y) > \lambda^{q_0} \right \} \right ) \\
< & \kappa \varepsilon \mu(\mathcal{Q}_1) \leq \kappa \varepsilon \mu(\mathcal{B}_{r}) = \kappa \varepsilon \mu(\mathcal{B}_{r}(x)),
\end{aligned}
\end{equation}
and by definition of $D_{\kappa \varepsilon}$ for any $(x,x) \in D_{\kappa \varepsilon}$ there exists some $0<r<\sqrt{n}/{2}$ such that 
$$ \Psi_{\lambda}(x,r) \geq \kappa \varepsilon \mu(\mathcal{B}_{r}(x)),$$
we see that for any $(x,x) \in D_{\kappa \varepsilon}$ there exists some exit radius $r_x \in (0,\sqrt{n}/2)$ such that
\begin{equation} \label{lowerexit}
\Psi_{\lambda}(x,r_x) \geq \kappa \varepsilon \mu(\mathcal{B}_{r_x}(x))
\end{equation}
and 
\begin{equation} \label{upperexit}
\Psi_{\lambda}(x,r) \leq \kappa \varepsilon \mu(\mathcal{B}_{r}(x)) \quad \text{for all } r > r_x.
\end{equation}
Now we consider the diagonal covering $\left \{ \mathcal{B}_{r_x}(x) \mid (x,x) \in D_{\kappa \varepsilon} \right \}$. Since $(\mathbb{R}^{2n},||\cdot||)$ is a separable metric space, by the Vitali covering lemma there exists a countable subset $J_D$ of $D_{\kappa \varepsilon}$, such that the family of balls $\{B_{r_x}(x)\}_{(x,x) \in J_D}$ is disjoint and we have
\begin{equation} \label{Vitaliap}
\bigcup_{(x,x) \in D_{\kappa \varepsilon}} \mathcal{B}_{r_x}(x) \subset \bigcup_{(x,x) \in J_D} \mathcal{B}_{5r_x}(x).
\end{equation}
Next, we classify the cubes from the family $\mathcal{U}_{\lambda}$ as follows. Let $\mathcal{U}_{\lambda}^d$ be the collection of all cubes from $\mathcal{U}_{ \lambda}$ that can be sucked up by the diagonal cover, that is,
\begin{equation} \label{Ud}
\mathcal{U}_{\lambda}^d:= \left \{ \mathcal{K} \in \mathcal{U}_{ \lambda} \mid \mathcal{K} \subset \bigcup_{(x,x) \in D_{\kappa \varepsilon}} \mathcal{B}_{r_x}(x) \right \} .
\end{equation}
Moreover, we define the family
$$ \mathcal{U}_{\lambda}^{nd} := \mathcal{U}_{\lambda} \setminus \mathcal{U}_{ \lambda}^d,$$
so that $\mathcal{U}_{ \lambda}$ is the disjoint union of $\mathcal{U}_{ \lambda}^d$ and $\mathcal{U}_{\lambda}^{nd}$.

The following Lemma reduces the problem of estimating the measure of $E$ with respect to $\mu$ to estimating the measures of the diagonal balls in the family $J_D$ and the measures of the off-diagonal cubes in the family $\mathcal{U}_{\lambda}^{nd}$.
\begin{lem} \label{decompcover}
	Let $E$ be given by (\ref{E}) and let $\varepsilon \in (0,1)$. Then we have
	\begin{equation} \label{coverest}
	\mu(E) \leq C \varepsilon \left(\kappa \sum_{(x,x) \in J_D} \mu(\mathcal{B}_{r_x}(x) \cap \mathcal{Q}_1) + \sum_{\mathcal{K} \in \mathcal{U}_{\lambda}^{nd}} \mu(\mathcal{K}) \right ),
	\end{equation}
	where $C=C(n,\theta)>0$.
\end{lem}
\begin{proof}
	Any $\mathcal{K} \in \mathcal{U}_{\lambda}^{nd}$ can be written as $\mathcal{K}=\mathcal{Q}_r(x_0,y_0)$ for some $r>0$ and some $x_0,y_0 \in \mathbb{R}^n$. Since $\widetilde{\mathcal{K}} \subset \mathcal{Q}_{3r}(x_0,y_0)$, we have
	\begin{equation} \label{cubedoubling}
	\mu(\widetilde{\mathcal{K}}) \leq \mu(\mathcal{Q}_{3r}(x_0,y_0)) \leq \mu \left (\mathcal{B}_{\frac{3\sqrt{n}}{2}r}(x_0,y_0) \right ) \leq \left (3\sqrt{n})^{n+2\theta} \mu(\mathcal{B}_{\frac{r}{2}}(x_0,y_0) \right ) \leq (3\sqrt{n})^{n+2\theta} \mu(\mathcal{K}).
	\end{equation}
	By (\ref{dyadiccover}), (\ref{Vitaliap}), (\ref{NDD}), (\ref{upperexit}),  (\ref{upperdensity}), (\ref{cubedoubling}) and Lemma \ref{doublingmeasure}, we have
	\begin{align*}
	\mu(E) & \leq \mu \left (\bigcup_{\mathcal{K} \in \mathcal{U}_{ \lambda}} (\mathcal{K} \cap E) \right ) \\ & = \mu \left (\bigcup_{\mathcal{K} \in \mathcal{U}_{\lambda}^d} (\mathcal{K} \cap E) \right ) +  \mu \left (\bigcup_{\mathcal{K} \in \mathcal{U}_{\lambda}^{nd}} (\mathcal{K} \cap E) \right ) \\
	& \leq \mu \left (\bigcup_{(x,x) \in D_{\kappa \varepsilon}} (\mathcal{B}_{r_x}(x) \cap E) \right ) +  \mu \left (\bigcup_{\mathcal{K} \in \mathcal{U}_{\lambda}^{nd}} (\mathcal{K} \cap E) \right ) \\
	& \leq \mu \left (\bigcup_{(x,x) \in J_D} (\mathcal{B}_{5r_x}(x) \cap E) \right ) +  \mu \left (\bigcup_{\mathcal{K} \in \mathcal{U}_{\lambda}^{nd}} (\widetilde{\mathcal{K}} \cap E) \right ) \\
	& \leq \sum_{(x,x) \in J_D} \mu(\mathcal{B}_{5r_x}(x) \cap E) + \sum_{\mathcal{K} \in \mathcal{U}_{\lambda}^{nd}} \mu (\widetilde{\mathcal{K}} \cap E) \\
	& \leq \sum_{(x,x) \in J_D} \Psi_{\lambda}(x,5r_x) + \sum_{\mathcal{K} \in \mathcal{U}_{\lambda}^{nd}} \mu (\widetilde{\mathcal{K}} \cap E) \\
	& \leq \sum_{(x,x) \in J_D} \kappa \varepsilon \mu(\mathcal{B}_{5r_x}(x)) + \sum_{\mathcal{K} \in \mathcal{U}_{\lambda}^{nd}} \varepsilon \mu (\widetilde{\mathcal{K}}) \\
	& = \varepsilon \left( \kappa 5^{n+2\theta} \sum_{(x,x) \in J_D} \mu(\mathcal{B}_{r_x}(x)) + \sum_{\mathcal{K} \in \mathcal{U}_{ \lambda}^{nd}} \mu(\widetilde{\mathcal{K}}) \right ) \\
	& \leq \varepsilon \left( \kappa (5\sqrt{n})^{n+2\theta} C_1 \sum_{(x,x) \in J_D} \mu(\mathcal{B}_{r_x}(x) \cap \mathcal{Q}_1) + (3\sqrt{n})^{n+2\theta} \sum_{\mathcal{K} \in \mathcal{U}_{\lambda}^{nd}} \mu(\mathcal{K}) \right ),
	\end{align*}
	where $C_1=C_1(n,\theta) \geq 1$. Thus, (\ref{coverest}) holds with $C=(5\sqrt{n})^{n+2\theta}C_1$.
\end{proof}
\subsection{Diagonal estimates}
Our next goal is to use the results from section \ref{gl}
in order to further estimate the right-hand side of (\ref{coverest}).
We start by estimating the first sum on the right-hand side of (\ref{coverest}), i.e. we are first dealing with the diagonal case, so that Lemma \ref{mfuse} is the crucial tool.

\begin{lem} \label{ondiagsum}
	Let $\delta = \delta(\varepsilon,\kappa,n,s,\theta,\Lambda) \in (0,1)$ be given by Lemma \ref{mfuse}. Then have 
	\begin{equation} \label{ondiagsumeq}
	\begin{aligned}
	\sum_{(x,x) \in J_D} \mu(\mathcal{B}_{r_x}(x) \cap \mathcal{Q}_1) \leq & \mu \left (\left \{(x,y) \in \mathcal{Q}_1 \mid  {\mathcal{M}}_{\mathcal{B}_{4n}}(U^2)(x,y) > \lambda^2 \right \} \right ) \\
	& + 2 \kappa^{-1} \varepsilon^{-1} \mu \left (\left \{(x,y) \in \mathcal{Q}_1 \mid  {\mathcal{M}}_{\mathcal{B}_{4n}} \left (G^{q_0} \right )(x,y) > \delta^{q_0} \lambda^{q_0} \right \} \right ).
	\end{aligned}
	\end{equation}
\end{lem}

\begin{proof}
	Fix some $(x,x) \in J_D$, so that by (\ref{lowerexit}) for the corresponding exit radius $r_x$ at least one of the following two inequalities must hold:
	\begin{equation} \label{ineq1}
	\mu \left (\left \{(x,y) \in \mathcal{B}_r(x) \cap \mathcal{Q}_1 \mid  {\mathcal{M}}_{\mathcal{B}_{4n}}(U^2)(x,y) > N_{d}^2 \lambda^2 \right \} \right ) \geq \frac{\kappa \varepsilon}{2} \mu(\mathcal{B}_r(x)),
	\end{equation}
	\begin{equation} \label{ineq2}
	\mu \left (\left \{(x,y) \in \mathcal{B}_{r}(x) \cap \mathcal{Q}_1 \mid  {\mathcal{M}}_{\mathcal{B}_{4n}} \left (G^{q_0} \right )(x,y) > \lambda^{q_0} \right \} \right ) \geq \frac{\kappa \varepsilon}{2} \mu(\mathcal{B}_r(x)).
	\end{equation}
	If (\ref{ineq1}) is satisfied, then Lemma \ref{mfuse} with $\kappa$ replaced by $\kappa/2$ implies 
	\begin{equation} \label{inclhi}
	\begin{aligned}
	B_{r_x}(x) \cap \mathcal{Q}_1 \subset & \left \{ (x,y) \in B_{r_x}(x) \cap \mathcal{Q}_1 \mid  {\mathcal{M}}_{\mathcal{B}_{4n}}(U^2)(x,y) > \lambda^2 \right \} \\ 
	& \cup \left \{ (x,y) \in B_{r_x}(x) \cap \mathcal{Q}_1 \mid  {\mathcal{M}}_{\mathcal{B}_{4n}}\left (G^{q_0} \right )(x,y) > \delta^{q_0} \lambda^{q_0} \right \}.
	\end{aligned}
	\end{equation}
	If on the other hand (\ref{ineq2}) is satisfied, then we directly obtain that
	\begin{equation} \label{inclhi1}
	\begin{aligned}
	& \mu(\mathcal{B}_{r_x}(x) \cap \mathcal{Q}_1) \leq \mu(\mathcal{B}_{r_x}(x)) \\ \leq & 2 \kappa^{-1} \varepsilon^{-1} \mu \left (\left \{(x,y) \in \mathcal{B}_{r_x}(x) \cap \mathcal{Q}_1 \mid  {\mathcal{M}}_{\mathcal{B}_{4n}} \left (G^{q_0} \right )(x,y) > \lambda^{q_0} \right \} \right ) \\
	\leq & 2 \kappa^{-1} \varepsilon^{-1} \mu \left (\left \{(x,y) \in \mathcal{B}_{r_x}(x) \cap \mathcal{Q}_1 \mid  {\mathcal{M}}_{\mathcal{B}_{4n}} \left (G^{q_0} \right )(x,y) > \delta^{q_0} \lambda^{q_0} \right \} \right ).
	\end{aligned}
	\end{equation}
	Therefore, in view of (\ref{inclhi}) and (\ref{inclhi1}) for any $(x,x) \in J_D$ we have
	\begin{align*}
	\mu(\mathcal{B}_{r_x}(x) \cap \mathcal{Q}_1) \leq & \mu \left (\left \{(x,y) \in \mathcal{B}_{r_x}(x) \cap \mathcal{Q}_1 \mid  {\mathcal{M}}_{\mathcal{B}_{4n}}(U^2)(x,y) > \lambda^2 \right \} \right ) \\
	& + 2 \kappa^{-1} \varepsilon^{-1} \mu \left (\left \{(x,y) \in \mathcal{B}_{r_x}(x) \cap \mathcal{Q}_1 \mid  {\mathcal{M}}_{\mathcal{B}_{4n}} \left (G^{q_0} \right )(x,y) > \delta^{q_0} \lambda^{q_0} \right \} \right ).
	\end{align*}
	Since the family of balls $\{B_{r_x}(x)\}_{(x,x) \in J_D}$ is disjoint, the assertion (\ref{ondiagsumeq}) immediately follows from the last display.
\end{proof}

\subsection{Off-diagonal estimates}

In order to estimate the measures of the off-diagonal cubes of class $\mathcal{U}_{\lambda}^{nd}$, we have ensure that as our terminology suggests, such cubes are indeed sufficiently far away from the diagonal in terms of their own sidelength.

\begin{lem} \label{faraway}
	There exists $\kappa=\kappa(n,\theta)>0$ small enough, such that for any cube $\mathcal{K}=K_1 \times K_2 \in \mathcal{U}_{\lambda}^{nd}$ of sidelength $2^{-k(\mathcal{K})}$, we have 
	$$ \textnormal{dist}(K_1,K_2) \geq (3 \sqrt{n}+1) 2^{-k(\mathcal{K})}.$$
\end{lem}

\begin{proof}
	Let $\mathcal{K}=K_1 \times K_2 \in \mathcal{U}_{\lambda}^{nd}$ and assume that we have $ \textnormal{dist}(K_1,K_2) < (3 \sqrt{n}+1) 2^{-k(\mathcal{K})}.$ Let us show that in this case for $\kappa$ small enough we have $\mathcal{K} \in \mathcal{U}_{\lambda}^{d}$, leading to a contradicton. Let $x$ be the center of $K_1$, $y$ be the center of $K_2$ and set $z:= (x+y)/2$. Then for $r:=\left ( 5 \sqrt{n}+1 \right ) 2^{-k(\mathcal{K})} $, we have
	$\mathcal{K} \subset \mathcal{B}_r(z).$
	Moreover, we have
	\begin{align*}
	\mu(\mathcal{B}_r(z)) & = c r^{n+2\theta} \\ & = c \left (5 \sqrt{n}+1 \right )^{n+2\theta} \left (2^{-k(\mathcal{K})} \right )^{n+2\theta} \\
	& = c \left ( 5 \sqrt{n}+1 \right )^{n+2\theta} \left (2^{-k(\mathcal{K})} \right )^{-(n-2 \theta)} \int_{K_1} \int_{K_2} dxdy \\
	& \leq c \left ( 5 \sqrt{n}+1 \right )^{n+2\theta} \left ( 5 \sqrt{n}+1 \right )^{n-2\theta} \int_{K_1} \int_{K_2} \frac{dxdy}{|x-y|^{n-2\theta}} \\
	& = c \left ( 5 \sqrt{n}+1 \right )^{2n} \mu(\mathcal{K})
	\end{align*}
	where $c=c(n,\theta)>0$. Now we assume that
	\begin{equation} \label{kappa}
	\kappa \leq c^{-1} \left ( 5 \sqrt{n}+1 \right )^{-2n}.
	\end{equation}
	Together with (\ref{lowerdensity}) applied to the set $E$ defined in (\ref{E}), we obtain
	\begin{equation} \label{cubeball}
	\Psi_{\lambda}(z,r) \geq \mu(E \cap \mathcal{K}) \geq \varepsilon \mu(\mathcal{K}) \geq \kappa \varepsilon \mu(\mathcal{B}_r(z)).
	\end{equation}
	In particular, (\ref{cubeball}) implies that $r<\sqrt{n}/2$, since otherwise we get a contradiction to (\ref{smallden}). Therefore, we have
	$$ \sup_{0<r < \sqrt{n}/2} \Psi_{\lambda}(z,r) \geq \kappa \varepsilon \mu(\mathcal{B}_r(z)),$$
	so that by definition of $\mathcal{D}_{\kappa \varepsilon}$ we obtain that $(z,z) \in \mathcal{D}_{\kappa \varepsilon}$. Moreover, in view of (\ref{upperexit}) we deduce that $r \leq r_z$, where $r_z$ is the exit radius at the point $z$ determined in (\ref{lowerexit}) and (\ref{upperexit}). We therefore have
	$$ \mathcal{K} \subset \mathcal{B}_r(z) \subset \mathcal{B}_{r_z}(z) \subset \bigcup_{(x,x) \in D_{\kappa \varepsilon}} \mathcal{B}_{r_x}(x),$$ 
	so that $\mathcal{K} \in \mathcal{U}^d_{\lambda}$, which contradicts the assumption that $\mathcal{K} \in \mathcal{U}^{nd}_{\lambda}$.
\end{proof}
In what follows, we set
\begin{equation} \label{kappadef}
\kappa:= \min \left \{c^{-1} \left ( 5 \sqrt{n}+1 \right )^{-2n}, (6\sqrt{n})^{-(n+2\theta)} \right \},
\end{equation}
where $c$ is given as in $(\ref{kappa})$, so that in particular Lemma \ref{faraway} is at our disposal. \par For any cube $\mathcal{K}=K_1 \times K_2 \in \mathcal{U}^{nd}_{\lambda}$, we write $P_1(\mathcal{K}):=K_1 \times K_1$ and $P_2(\mathcal{K}):=K_2 \times K_2$. Furthermore, we write
$$ \phi(\mathcal{K}):=\frac{2^{-k(\mathcal{K})}}{\textnormal{dist}(K_1, K_2)}, $$
which matches the function $\phi(r,x_0,y_0)$ introduced in (\ref{phidef}). \par
In view of Corollary \ref{mfusenondiagcubes} and Lemma \ref{faraway}, for any cube $\mathcal{K} \in \mathcal{U}^{nd}_{\lambda}$ we have
\begin{equation} \label{nondiag}
\begin{aligned}
& \mu(\mathcal{K}) = \mu(\mathcal{Q}_{r}(x_0,y_0)) \\
\leq & (\sqrt{n})^{n+2\theta} \bigg( \mu \left (\left \{ (x,y) \in \mathcal{K} \mid  {\mathcal{M}}_{\mathcal{B}_{4n}} (U^2)(x,y) > \lambda^2 \right \} \right ) \\ 
& + \phi(\mathcal{K})^{n-2 \theta} \mu \left (\left \{ (x,y) \in P_1\mathcal{K} \mid  {\mathcal{M}}_{\mathcal{B}_{4n}} (U^2)(x,y) > N_d^2 \phi(\mathcal{K})^{-2(\theta+s)} \lambda^2 \right \} \right ) \\
& + \phi(\mathcal{K})^{n-2 \theta} \mu \left (\left \{ (x,y) \in P_2\mathcal{K} \mid  {\mathcal{M}}_{\mathcal{B}_{4n}} (U^2)(x,y) > N_d^2 \phi(\mathcal{K})^{-2(\theta+s)} \lambda^2 \right \} \right ) \\
& + \phi(\mathcal{K})^{n-2 \theta} \mu \left (\left \{ (x,y) \in P_1 \mathcal{K} \mid  {\mathcal{M}}_{\mathcal{B}_{4n}} (G^{q_0})(x,y) > \phi(\mathcal{K})^{-{q_0}(\theta+s)} \lambda^{q_0} \right \} \right ) \\
& + \phi(\mathcal{K})^{n-2 \theta} \mu \left (\left \{ (x,y) \in P_2\mathcal{K} \mid  {\mathcal{M}}_{\mathcal{B}_{4n}} (G^{q_0})(x,y) > \phi(\mathcal{K})^{-{q_0}(\theta+s)}\lambda^{q_0} \right \} \right ) \bigg ).
\end{aligned}
\end{equation}
For $h=1,2$, for simplicity of notation we define 
$$ Z_\lambda^{U,h}(\mathcal{K}) := \left \{ (x,y) \in P_h \mathcal{K} \mid  {\mathcal{M}}_{\mathcal{B}_{4n}} (U^2)(x,y) > N_d^2 \lambda^2 \right \} $$
and
$$ Z_\lambda^{G,h}(\mathcal{K}) := \left \{ (x,y) \in P_h \mathcal{K} \mid  {\mathcal{M}}_{\mathcal{B}_{4n}} (G^{q_0})(x,y) > \lambda^{q_0} \right \}. $$ 
In order to handle the diagonal level sets on the right-hand side of (\ref{nondiag}), for $h=1,2$ we also define the subfamilies
$$ \mathcal{G}^{h}_{\lambda} := \left \{ \mathcal{K} \in \mathcal{U}^{nd}_{\lambda} \mid \mu(Z_\lambda^{U,h}(\mathcal{K}))+\mu(Z_\lambda^{G,h}(\mathcal{K})) \leq \varepsilon \mu(P_h \mathcal{K}) \right \} $$ and
$$ \mathcal{N}^{h}_{\lambda} := \left \{ \mathcal{K} \in \mathcal{U}^{nd}_{\lambda} \mid \mu(Z_\lambda^{U,h}(\mathcal{K}))+\mu(Z_\lambda^{G,h}(\mathcal{K})) > \varepsilon \mu(P_h \mathcal{K}) \right \}. $$
Moreover, set $$\mathcal{G}_{\lambda}:=\mathcal{G}^{1}_{\lambda} \cap \mathcal{G}^{2}_{\lambda} \text{ and } \mathcal{N}_{\lambda}:=\mathcal{N}^{1}_{\lambda} \cup \mathcal{N}^{2}_{\lambda},$$
so that we clearly have $\mathcal{U}^{nd}_{\lambda} = \mathcal{G}_{\lambda} \cup \mathcal{N}_{\lambda}$, where the union is disjoint. \par
The following Lemma shows that if $\mathcal{K} \in \mathcal{G}_{\lambda}$, then the diagonal terms on the right-hand side of (\ref{nondiag}) can be treated in a fairly straightforward manner.
\begin{lem} \label{softest}
For $\varepsilon$ small enough, we have
\begin{equation} \label{soft}
\sum_{\mathcal{K} \in \mathcal{G}_{\lambda}} \mu(\mathcal{K}) \leq C \mu \left (\left \{ (x,y) \in \mathcal{Q}_1 \mid  {\mathcal{M}}_{\mathcal{B}_{4n}} (U^2)(x,y) > \lambda^2 \right \} \right ),
\end{equation}
where $C=C(n,\theta)>0$.
\end{lem}

\begin{proof}
Since in view of Lemma \ref{faraway} we have $\phi(\mathcal{K}) \leq 1$, together with (\ref{nondiag}), using that $\mu(P_1 \mathcal{K})=\mu(P_2 \mathcal{K})$, for any $\mathcal{K} \in \mathcal{G}_\lambda$ we obtain
\begin{align*}
\mu(\mathcal{K})
\leq & (\sqrt{n})^{n+2\theta} \bigg( \mu \left (\left \{ (x,y) \in \mathcal{K} \mid  {\mathcal{M}}_{\mathcal{B}_{4n}} (U^2)(x,y) > \lambda^2 \right \} \right ) \\ 
& + \phi(\mathcal{K})^{n-2 \theta} \left (\mu \left (Z^{U,1}_\lambda (\mathcal{K})\right ) + \mu \left (Z^{U,2}_\lambda(\mathcal{K}) \right ) + \mu \left (Z^{G,1}_\lambda(\mathcal{K}) \right ) +\mu \left (Z^{G,2}_\lambda(\mathcal{K}) \right ) \right ) \bigg )\\
\leq & (\sqrt{n})^{n+2\theta} \bigg( \mu \left (\left \{ (x,y) \in \mathcal{K} \mid  {\mathcal{M}}_{\mathcal{B}_{4n}} (U^2)(x,y) > \lambda^2 \right \} \right )
+ 2\phi(\mathcal{K})^{n-2 \theta} \varepsilon \mu(P_1 \mathcal{K}) \bigg ).
\end{align*}
Note that we have $\mathcal{K}=\mathcal{Q}_r(x_0,y_0)$ for $r=2^{-k(\mathcal{K})}$ and $x_0, y_0 \in Q_1$. In view of Lemma \ref{faraway}, we can argue in the same way as in (\ref{ineq99}) and (\ref{comparable}) to obtain
$$ \frac{2^{-2nk(\mathcal{K})}}{\textnormal{dist}(K_1,K_2)^{n-2\theta}} \leq 2^{2n+1} \frac{(r/2)^{2n}}{\textnormal{dist}(B_\frac{r}{2}(x_0),B_\frac{r}{2}(y_0))^{n-2\theta}} \leq 2^{2n+1}\mu(\mathcal{B}_\frac{r}{2}(x_0,y_0)) \leq 2^{2n+1}\mu(\mathcal{K}).$$
Since also $$\mu(P_1 \mathcal{K}) \leq \mu \left (\mathcal{B}_{\frac{\sqrt{n}}{2}r}(x_0) \right )=C_1r^{n+2\theta}=C_1 2^{-(n+2\theta)k(\mathcal{K})}$$ for some $C_1=C_1(n,\theta) \geq 1$, we deduce that
\begin{equation} \label{phiesta}
\phi(\mathcal{K})^{n-2\theta} \leq C_2 \frac{\mu(\mathcal{K})}{\mu(P_1 \mathcal{K})},
\end{equation}
where $C_2=C_2(n,\theta)\geq 1$.
By connecting the previous display with the first one in the proof and assuming that
\begin{equation} \label{epsrestr}
\varepsilon \leq \frac{1}{4 (\sqrt{n})^{n+2\theta} C_2},
\end{equation}
we arrive at
\begin{align*}
\mu(\mathcal{K})
\leq & (\sqrt{n})^{n+2\theta} \left( \mu \left (\left \{ (x,y) \in \mathcal{K} \mid  {\mathcal{M}}_{\mathcal{B}_{4n}} (U^2)(x,y) > \lambda^2 \right \} \right )
+ 2C_2 \varepsilon \mu(\mathcal{K}) \right ) \\
\leq & (\sqrt{n})^{n+2\theta} \mu \left (\left \{ (x,y) \in \mathcal{K} \mid  {\mathcal{M}}_{\mathcal{B}_{4n}} (U^2)(x,y) > \lambda^2 \right \} \right )
+ \frac{\mu(\mathcal{K})}{2},
\end{align*}
which by reabsorbing the last term on the the right-hand side into the left-hand side implies 
$$ \mu(\mathcal{K}) \leq 2(\sqrt{n})^{n+2\theta} \mu \left (\left \{ (x,y) \in \mathcal{K} \mid  {\mathcal{M}}_{\mathcal{B}_{4n}} (U^2)(x,y) > \lambda^2 \right \} \right ).$$
Summing over $\mathcal{K} \in \mathcal{G}_{\lambda}$ and using that all cubes in the family $\mathcal{U}_\lambda$ and therefore also all cubes in $\mathcal{G}_{\lambda}$ are disjoint and contained in $\mathcal{Q}_1$, we see that (\ref{soft}) holds with $C=2(\sqrt{n})^{n+2\theta}$. This finishes the proof.
\end{proof}
From now on, we will always assume that the restriction (\ref{epsrestr}) on $\varepsilon$ holds, so that the estimate (\ref{soft}) from Lemma \ref{softest} holds. \par
It remains to control the last four terms on the right-hand side of (\ref{nondiag}) also in the more involved case when $\mathcal{K} \in \mathcal{N}_\lambda$, which requires a delicate combinatorial argument inspired by \cite{selfimpro}. In order to accomplish this, for $h=1,2$ we define the families
$$ P_h \mathcal{N}_\lambda := \left \{ P_h \mathcal{K} \mid \mathcal{K} \in \mathcal{N}_\lambda^{h} \right \}. $$
Since $P_1 \mathcal{N}_\lambda \cup P_2 \mathcal{N}_\lambda$ is a family of dyadic cubes, clearly there is a disjoint subfamily $P \mathcal{N}_\lambda$ of $P_1 \mathcal{N}_\lambda \cup P_2 \mathcal{N}_\lambda$ such that
\begin{equation} \label{disjointcover}
\bigcup_{\mathcal{H} \in P \mathcal{N}_\lambda} \mathcal{H} = \bigcup_{\mathcal{K} \in P_1 \mathcal{N}_\lambda \cup P_2 \mathcal{N}_\lambda} \mathcal{K}.
\end{equation}
In other words, any cube in $P_1 \mathcal{N}_\lambda \cup P_2 \mathcal{N}_\lambda$ is contained in exactly one cube $\mathcal{H} \in P \mathcal{N}_\lambda$. The following Lemma plays a crucial role in the mentioned combinatorial argument. It shows that a cube of class $\mathcal{N}_\lambda$ is not only far away from the diagonal in terms of its own sidelength as shown in Lemma \ref{faraway}, but also in terms of the sidelength of the larger cube $\mathcal{H} \in P \mathcal{N}_\lambda$ in which its projection onto the diagonal is contained.

\begin{lem} \label{farawayH}
	Let $\mathcal{K}=K_1 \times K_2 \in \mathcal{N}_\lambda$, so that for some $h \in \{1,2\}$ $P_h \mathcal{K}$ belongs to $P_h \mathcal{N}_\lambda$. Moro ever, let $\mathcal{H}=H \times H$ be the unique cube that belongs to $P \mathcal{N}_\lambda$ and contains $P_h \mathcal{K}$. Then we have $\textnormal{dist}(K_1,K_2) \geq 2^{-k(\mathcal{H})}$.
\end{lem}
\begin{proof}
	Without loss of generality we can assume that $h=1$, since the case when $h=2$ can be treated in the same way. We prove by contradiction. Assume that 
	\begin{equation} \label{contr}
	\textnormal{dist}(K_1,K_2) < 2^{-k(\mathcal{H})}
	\end{equation}
	and denote by $x_\mathcal{H}$ the center of the cube $H$. Choose points $x_1 \in \overline H$ and $y_1 \in \overline K_2$ such that $\textnormal{dist}(\overline H, \overline K_2)=|x_1-y_1|$ and denote by $y_0$ the center of $K_2$. Then for any $y \in K_2$, we have
	\begin{align*}
	|x_\mathcal{H}-y| \leq |x_1-y_1|+|x_\mathcal{H}-x|+|y_1-y_0|+|y_0-y| < & 2^{-k(\mathcal{H})} + \frac{\sqrt{n}}{2} 2^{-k(\mathcal{H})} + \sqrt{n} 2^{-k(\mathcal{K})} \\
	\leq & 3 \sqrt{n} 2^{-k(\mathcal{H})},
	\end{align*}
	so that $K_2 \subset B_{3 \sqrt{n} 2^{-k(\mathcal{H})}}(x_\mathcal{H})$. Since by assumption $K_1 \subset H$, we arrive at 
	\begin{equation} \label{Kcont}
	\mathcal{K} \subset \mathcal{B}_{3 \sqrt{n} 2^{-k(\mathcal{H})}}(x_\mathcal{H}).
	\end{equation}
	Since $\mathcal{H}$ belongs to $P \mathcal{N}_\lambda$ and thus to $P_1 \mathcal{N}_\lambda \cup P_2 \mathcal{N}_\lambda$, we have 
	\begin{equation} \label{HexitU}
	\begin{aligned}
	& \mu \left (\left \{ (x,y) \in \mathcal{H} \mid  {\mathcal{M}}_{\mathcal{B}_{4n}} (U^2)(x,y) > N_d^2 \lambda^2 \right \} \right ) \\ & +  \mu \left (\left \{ (x,y) \in \mathcal{H} \mid  {\mathcal{M}}_{\mathcal{B}_{4n}} (G^{q_0})(x,y) > \lambda^{q_0} \right \} \right )  > \varepsilon \mu(\mathcal{H}).
	\end{aligned}
	\end{equation}
	(\ref{HexitU}) implies
	\begin{align*}
	 \Psi_\lambda(x_\mathcal{H},3 \sqrt{n} 2^{-k(\mathcal{H})}) = &\mu \left (\left \{ (x,y) \in \mathcal{B}_{3 \sqrt{n} 2^{-k(\mathcal{H})}}(x_\mathcal{H}) \mid  {\mathcal{M}}_{\mathcal{B}_{4n}} (U^2)(x,y) > N_d^2 \lambda^2 \right \} \right ) \\ & + \mu \left (\left \{ (x,y) \in \mathcal{B}_{3 \sqrt{n} 2^{-k(\mathcal{H})}}(x_\mathcal{H}) \mid  {\mathcal{M}}_{\mathcal{B}_{4n}} (G^{q_0})(x,y) > \lambda^{q_0} \right \} \right ) \\ \geq & 
	 \mu \left (\left \{ (x,y) \in \mathcal{H} \mid  {\mathcal{M}}_{\mathcal{B}_{4n}} (U^2)(x,y) > N_d^2 \lambda^2 \right \} \right ) \\ &+
	 \mu \left (\left \{ (x,y) \in \mathcal{H} \mid  {\mathcal{M}}_{\mathcal{B}_{4n}} (G^{q_0})(x,y) > \lambda^{q_0} \right \} \right ) \\ > & \varepsilon \mu(\mathcal{H}) \geq \varepsilon \mu \left (\mathcal{B}_{2^{-(k(\mathcal{H})+1)}}(x_\mathcal{H}) \right ) \\ \geq & \frac{1}{(6\sqrt{n})^{n+2\theta}} \varepsilon \mu \left (\mathcal{B}_{3 \sqrt{n} 2^{-k(\mathcal{H})}}(x_\mathcal{H}) \right ) \geq \kappa \varepsilon \mu \left (\mathcal{B}_{3 \sqrt{n} 2^{-k(\mathcal{H})}}(x_\mathcal{H}) \right ) ,
	\end{align*}
	where we also used (\ref{kappadef}) in order to obtain the last inequality.
	Therefore, in view of (\ref{lowerexit}) and (\ref{upperexit}) we have
	$3 \sqrt{n} 2^{-k(\mathcal{H})} \leq r_{x_\mathcal{H}},$ where $r_{x_\mathcal{H}}$ is the exit radius at the point $x_\mathcal{H}$. In particular, we have
	$$ \mathcal{B}_{3 \sqrt{n} 2^{-k(\mathcal{H})}}(x_\mathcal{H}) \subset \mathcal{B}_{r_{x_\mathcal{H}}}(x_\mathcal{H}),$$
	which together with (\ref{Kcont}) implies
	$$ \mathcal{K} \subset \mathcal{B}_{r_{x_\mathcal{H}}}(x_\mathcal{H}).$$
	But then by definition of $\mathcal{U}^d_\lambda$ (see (\ref{Ud})), we obtain that $\mathcal{K} \in \mathcal{U}^d_\lambda$, which is a contradiction since $\mathcal{K} \in \mathcal{N}_\lambda \subset \mathcal{U}^{nd}_\lambda$ and $\mathcal{U}^d_\lambda \cap \mathcal{U}^{nd}_\lambda = \emptyset$. Thus, the proof is finished.
\end{proof}

We now estimate the measures of the cubes of class $\mathcal{N}_\lambda$. The main tools are (\ref{nondiag}) and the above Lemma \ref{farawayH}, which allows to classify the projected diagonal cubes in a way that enables us to control the diagonal terms in (\ref{nondiag}).

\begin{lem} \label{hardest}
We have
	\begin{equation} \label{hard}
	\begin{aligned}
	\sum_{\mathcal{K} \in \mathcal{N}_{\lambda}} \mu(\mathcal{K}) \leq & \frac{C}{\lambda^2} \int_{\mathcal{Q}_1 \cap \left \{ {\mathcal{M}}_{\mathcal{B}_{4n}} (U^2) > \lambda^2 \right \}} {\mathcal{M}}_{\mathcal{B}_{4n}} (U^2 ) d\mu \\ & + \frac{C}{\lambda^{q_0}} \int_{\mathcal{Q}_1 \cap \left \{ {\mathcal{M}}_{\mathcal{B}_{4n}} (G^{q_0}) > \lambda^{q_0} \right \}} {\mathcal{M}}_{\mathcal{B}_{4n}} (G^{q_0} ) d\mu,
	\end{aligned}
	\end{equation}
	where $C=C(n,s,\theta)>0$.
\end{lem}

\begin{proof}
\emph{Step 1: A combinatorial argument.}
For any $\mathcal{H} \in P \mathcal{N}_{\lambda}$ and $h \in \{1,2\}$, we define the family of cubes
$$ \mathcal{N}_\lambda^{h}(\mathcal{H}) := \left \{\mathcal{K} \in \mathcal{N}_\lambda \mid P_h \mathcal{K} \subset \mathcal{H} \right \} .$$
Since the family $P \mathcal{N}_{\lambda}$ is a disjoint covering of the family $P_1 \mathcal{N}_\lambda \cup P_2 \mathcal{N}_\lambda$, we can decompose $\mathcal{N}^h_{\lambda}$ into mutually disjoint subfamilies as follows
\begin{equation} \label{disdecomp}
\mathcal{N}^h_{\lambda} = \bigcup_{\mathcal{H} \in P \mathcal{N}_{\lambda}} \mathcal{N}^h_{\lambda} (\mathcal{H}),
\end{equation}
in the sense that we have $\mathcal{N}^h_{\lambda} (\mathcal{H}_1) \cap \mathcal{N}^h_{\lambda} (\mathcal{H}_2) = \emptyset$ whenever $\mathcal{H}_1 \neq \mathcal{H}_2$. Since for any $\mathcal{H} \in P \mathcal{N}_{\lambda}$ and any $\mathcal{K} \in \mathcal{N}^h_{\lambda} (\mathcal{H})$ we have $k(\mathcal{K})=k(P_1\mathcal{K}) \geq k(\mathcal{H})$, we can decompose $\mathcal{N}^h_{\lambda} (\mathcal{H})$ into the following classes
$$ [\mathcal{N}^h_{\lambda} (\mathcal{H})]_i := \left \{\mathcal{K} \in \mathcal{N}^h_{\lambda} (\mathcal{H}) \mid k(\mathcal{K}) = i+k(\mathcal{H}) \right \}, \quad i \in \mathbb{N}_0, \quad h\in \{1,2\}.$$
More precisely, we have the decomposition into mutually disjoint subfamilies
\begin{equation} \label{disdecompi}
\mathcal{N}^h_{\lambda} (\mathcal{H}) = \bigcup_{i \geq 0} [\mathcal{N}^h_{\lambda} (\mathcal{H})]_i,
\end{equation}
in the sense that $[\mathcal{N}^h_{\lambda} (\mathcal{H})]_i \cap [\mathcal{N}^h_{\lambda} (\mathcal{H})]_j = \emptyset$ whenever $i \neq j$. \par 
Next, for  $h\in \{1,2\}$ and $i,j \in \mathbb{N}_0$ we define further subfamilies by
$$ [\mathcal{N}^h_{\lambda} (\mathcal{H})]_{i,j} := \left \{\mathcal{K}=K_1 \times K_2 \in [\mathcal{N}^h_{\lambda} (\mathcal{H})]_i \mid 2^{j-k(\mathcal{H})} \leq \textnormal{dist}(K_1,K_2) < 2^{j+1-k(\mathcal{H})} \right \}. $$
Since by Lemma \ref{farawayH} for any $\mathcal{K}=K_1 \times K_2 \in \mathcal{N}^h_{\lambda} (\mathcal{H})$ we have $2^{-k(\mathcal{H})} \leq \textnormal{dist}(K_1,K_2)$, we have the disjoint decomposition
\begin{equation} \label{disdecompij}
\mathcal{N}^h_{\lambda} (\mathcal{H}) = \bigcup_{i,j \geq 0} [\mathcal{N}^h_{\lambda} (\mathcal{H})]_{i,j},
\end{equation}
in the sense that $[\mathcal{N}^h_{\lambda} (\mathcal{H})]_{i_1,j_1} \cap [\mathcal{N}^h_{\lambda} (\mathcal{H})]_{i_2,j_2} = \emptyset$ whenever $(i_1,j_1) \neq (i_2,j_2)$. \par 
Therefore, by combining (\ref{disdecomp}) and (\ref{disdecompij}), we arrive at the following decomposition of mutually disjoint subfamilies
\begin{equation} \label{dd}
\mathcal{N}^h_{\lambda} = \bigcup_{\mathcal{H} \in P \mathcal{N}_{\lambda}} \bigcup_{i,j \geq 0} [\mathcal{N}^h_{\lambda} (\mathcal{H})]_{i,j}.
\end{equation}
Now fix some $\mathcal{H}=H \times H \in P \mathcal{N}_{\lambda}$. Our next goal is to prove that for $h \in \{1,2\}$ the following inequality holds
\begin{equation} \label{combineq}
\begin{aligned}
&\sum_{\mathcal{K} \in \mathcal{N}^h_{\lambda} (\mathcal{H})} \phi(\mathcal{K})^{n-2 \theta} \mu \left (\left \{ (x,y) \in P_h \mathcal{K} \mid  {\mathcal{M}}_{\mathcal{B}_{4n}} (U^2)(x,y) > N_d^2 \phi(\mathcal{K})^{-2(\theta+s)} \lambda^2 \right \} \right ) \\
&+ \sum_{\mathcal{K} \in \mathcal{N}^h_{\lambda} (\mathcal{H})} \phi(\mathcal{K})^{n-2 \theta} \mu \left (\left \{ (x,y) \in P_h\mathcal{K} \mid  {\mathcal{M}}_{\mathcal{B}_{4n}} (G^{q_0})(x,y) > \phi(\mathcal{K})^{-{q_0}(\theta+s)} \lambda^{q_0} \right \} \right ) \\
\leq & C_0 \left (\frac{1}{\lambda^2} \int_{\mathcal{H} \cap \left \{ {\mathcal{M}}_{\mathcal{B}_{4n}} (U^2) > \lambda^2 \right \}} {\mathcal{M}}_{\mathcal{B}_{4n}} (U^2 ) d\mu + \frac{1}{\lambda^{q_0}} \int_{\mathcal{H} \cap \left \{ {\mathcal{M}}_{\mathcal{B}_{4n}} (G^{q_0}) > \lambda^{q_0} \right \}} {\mathcal{M}}_{\mathcal{B}_{4n}} (G^{q_0} ) d\mu \right ),
\end{aligned}
\end{equation}
where $C_0=C_0(n,s)>0$. By definition of the class $[\mathcal{N}^h_{\lambda} (\mathcal{H})]_{i,j}$, for any $ \mathcal{K}=K_1 \times K_2 \in [\mathcal{N}^h_{\lambda} (\mathcal{H})]_{i,j}$ we have
$$ \phi(\mathcal{K})= \frac{2^{-k(\mathcal{K})}}{\textnormal{dist}(K_1,K_2)} = \frac{1}{2^i} \frac{2^{-k(\mathcal{H})}}{\textnormal{dist}(K_1,K_2)} \leq \frac{1}{2^{i+j}} .$$
By using Chebychev's inequality, (\ref{disdecompij}) and then the last display, we obtain
\begin{equation} \label{chebappU}
\begin{aligned}
&\sum_{\mathcal{K} \in \mathcal{N}^h_{\lambda} (\mathcal{H})} \phi(\mathcal{K})^{n-2 \theta} \mu \left (\left \{ (x,y) \in P_h \mathcal{K} \mid  {\mathcal{M}}_{\mathcal{B}_{4n}} (U^2)(x,y) > N_d^2 \phi(\mathcal{K})^{-2(\theta+s)} \lambda^2 \right \} \right ) \\
\leq & \frac{1}{N_d^2 \lambda^2} \sum_{\mathcal{K} \in \mathcal{N}^h_{\lambda} (\mathcal{H})} \phi(\mathcal{K})^{n+2s} \int_{P_h \mathcal{K} \cap \left \{ {\mathcal{M}}_{\mathcal{B}_{4n}} (U^2) > N_d^2 \lambda^2 \right \}} {\mathcal{M}}_{\mathcal{B}_{4n}} (U^2 ) d\mu \\
\leq & \frac{1}{\lambda^2} \sum_{i,j=0}^\infty \sum_{\mathcal{K} \in [\mathcal{N}^h_{\lambda} (\mathcal{H})]_{i,j}} \phi(\mathcal{K})^{n+2s} \int_{P_h \mathcal{K} \cap \left \{ {\mathcal{M}}_{\mathcal{B}_{4n}} (U^2) > \lambda^2 \right \}} {\mathcal{M}}_{\mathcal{B}_{4n}} (U^2 ) d\mu \\
\leq & \frac{1}{\lambda^2} \sum_{i,j=0}^\infty \left (\frac{1}{2^{i+j}} \right )^{n+2s} \sum_{\mathcal{K} \in [\mathcal{N}^h_{\lambda} (\mathcal{H})]_{i,j}} \int_{P_h \mathcal{K} \cap \left \{ {\mathcal{M}}_{\mathcal{B}_{4n}} (U^2) > \lambda^2 \right \}} {\mathcal{M}}_{\mathcal{B}_{4n}} (U^2 ) d\mu 
\end{aligned}
\end{equation} 
and similarly by additionally using that $n+(q_0-2)\theta+q_0s \geq n+2s$
\begin{equation} \label{chebappG}
\begin{aligned}
&\sum_{\mathcal{K} \in \mathcal{N}^h_{\lambda} (\mathcal{H})} \phi(\mathcal{K})^{n-2 \theta} \mu \left (\left \{ (x,y) \in P_h \mathcal{K} \mid  {\mathcal{M}}_{\mathcal{B}_{4n}} (G^{q_0})(x,y) > \phi(\mathcal{K})^{-2(\theta+s)} \lambda^2 \right \} \right ) \\
\leq & \frac{1}{\lambda^{q_0}} \sum_{\mathcal{K} \in \mathcal{N}^h_{\lambda} (\mathcal{H})} \phi(\mathcal{K})^{n+(q_0-2)\theta+q_0s} \int_{P_h \mathcal{K} \cap \left \{ {\mathcal{M}}_{\mathcal{B}_{4n}} (G^{q_0}) > \lambda^{q_0} \right \}} {\mathcal{M}}_{\mathcal{B}_{4n}} (G^{q_0} ) d\mu \\
\leq & \frac{1}{\lambda^{q_0}} \sum_{i,j=0}^\infty \sum_{\mathcal{K} \in [\mathcal{N}^h_{\lambda} (\mathcal{H})]_{i,j}} \phi(\mathcal{K})^{n+(q_0-2)\theta+q_0s} \int_{P_h \mathcal{K} \cap \left \{ {\mathcal{M}}_{\mathcal{B}_{4n}} (G^{q_0}) > \lambda^{q_0} \right \}} {\mathcal{M}}_{\mathcal{B}_{4n}} (G^{q_0} ) d\mu \\
\leq & \frac{1}{\lambda^{q_0}} \sum_{i,j=0}^\infty \left (\frac{1}{2^{i+j}} \right )^{n+2s} \sum_{\mathcal{K} \in [\mathcal{N}^h_{\lambda} (\mathcal{H})]_{i,j}} \int_{P_h \mathcal{K} \cap \left \{ {\mathcal{M}}_{\mathcal{B}_{4n}} (G^{q_0}) > \lambda^{q_0} \right \}} {\mathcal{M}}_{\mathcal{B}_{4n}} (G^{q_0} ) d\mu. 
\end{aligned}
\end{equation}
In order to proceed, we need to further decompose the families $[\mathcal{N}^h_{\lambda} (\mathcal{H})]_{i,j}$. For any $i \geq 0$, $\mathcal{H}$ contains exactly $2^{ni}$ disjoint diagonal cubes $\mathcal{H}^m_i=H^m_i \times H^m_i$, $m \in \{1,...,2^{ni}\}$, with sidelength $2^{-i-k(\mathcal{H})}$. In particular, the disjointness of these cubes implies
\begin{equation} \label{subaddU}
\sum_{m=1}^{2^{ni}} \int_{\mathcal{H}^m_i \cap \left \{ {\mathcal{M}}_{\mathcal{B}_{4n}} (U^2) >\lambda^2 \right \}} {\mathcal{M}}_{\mathcal{B}_{4n}} (U^2 ) d\mu \leq \int_{\mathcal{H} \cap \left \{ {\mathcal{M}}_{\mathcal{B}_{4n}} (U^2) >\lambda^2 \right \}} {\mathcal{M}}_{\mathcal{B}_{4n}} (U^2 ) d\mu
\end{equation}
and 
\begin{equation} \label{subaddG}
\sum_{m=1}^{2^{ni}} \int_{\mathcal{H}^m_i \cap \left \{ {\mathcal{M}}_{\mathcal{B}_{4n}} (G^{q_0}) >\lambda^{q_0} \right \}} {\mathcal{M}}_{\mathcal{B}_{4n}} (G^{q_0} ) d\mu \leq \int_{\mathcal{H} \cap \left \{ {\mathcal{M}}_{\mathcal{B}_{4n}} (G^{q_0}) >\lambda^{q_0} \right \}} {\mathcal{M}}_{\mathcal{B}_{4n}} (G^{q_0}) d\mu.
\end{equation}
For $h \in \{1,2\}$, $i,j \geq 0$ and $m \in \{1,...,2^{ni}\}$, define the families
$$ [\mathcal{N}^h_{\lambda} (\mathcal{H})]_{i,j,m} := \left \{\mathcal{K} \in [\mathcal{N}^h_{\lambda} (\mathcal{H})]_{i,j} \mid P_h \mathcal{K} = \mathcal{H}^m_i \right \} $$ and observe that we have the disjoint decomposition
\begin{equation} \label{decomp3}
[\mathcal{N}^h_{\lambda} (\mathcal{H})]_{i,j} = \bigcup_{m=1}^{2^{ni}} [\mathcal{N}^h_{\lambda} (\mathcal{H})]_{i,j,m}.
\end{equation}
For a moment, let us focus on the case when $h=1$. Note that since $\mathcal{N}_\lambda^1$ is a family of dyadic cubes, we have 
$P_2\mathcal{K}_1 \cap P_2\mathcal{K}_2 = \emptyset$ for all cubes $\mathcal{K}_1,\mathcal{K}_2 \in [\mathcal{N}^1_{\lambda} (\mathcal{H})]_{i,j,m}$ with $ \mathcal{K}_1 \neq \mathcal{K}_2$, since otherwise $\mathcal{K}_1$ and $\mathcal{K}_2$ would coincide.
Therefore, we observe that the cubes in $[\mathcal{N}^1_{\lambda} (\mathcal{H})]_{i,j,m}$ are all contained in the family $\mathcal{F}^1_{i,j,m}(\mathcal{H})$ consisting of all distinct dyadic cubes of the form $\mathcal{K} = H^m_i \times K$ with $K \subset \mathcal{Q}_1$ and sidelength $2^{-i-k(\mathcal{H})}$ that additionally satisfy 
\begin{equation} \label{dist}
2^{j-k(\mathcal{H})} \leq \textnormal{dist}(H^m_i,K) < 2^{j+1-k(\mathcal{H})}.
\end{equation}
Then in view of a combinatorial consideration, we have the estimate
\begin{equation} \label{cardest}
\# [\mathcal{N}^1_{\lambda} (\mathcal{H})]_{i,j,m} \leq \#\mathcal{F}^1_{i,j,m}(\mathcal{H}) \leq C_1 2^{n(i+j)},
\end{equation}
where $C_1=C_1(n)>0$.
Thus, in view of (\ref{decomp3}), (\ref{cardest}), (\ref{subaddU}) and (\ref{subaddG}), we obtain
\begin{align*}
& \sum_{\mathcal{K} \in [\mathcal{N}^1_{\lambda} (\mathcal{H})]_{i,j}} \int_{P_h \mathcal{K} \cap \left \{ {\mathcal{M}}_{\mathcal{B}_{4n}} (U^2) > \lambda^2 \right \}} {\mathcal{M}}_{\mathcal{B}_{4n}} (U^2 ) d\mu \\
= & \sum_{m=1}^{2^{ni}} \sum_{\mathcal{K} \in [\mathcal{N}^1_{\lambda} (\mathcal{H})]_{i,j}} \int_{\mathcal{H}^m_i \cap \left \{ {\mathcal{M}}_{\mathcal{B}_{4n}} (U^2) > \lambda^2 \right \}} {\mathcal{M}}_{\mathcal{B}_{4n}} (U^2 ) d\mu \\
\leq & C_1 2^{n(i+j)} \sum_{m=1}^{2^{ni}} \int_{\mathcal{H}^m_i \cap \left \{ {\mathcal{M}}_{\mathcal{B}_{4n}} (U^2) > \lambda^2 \right \}} {\mathcal{M}}_{\mathcal{B}_{4n}} (U^2 ) d\mu \\
\leq & C_1 2^{n(i+j)} \int_{\mathcal{H} \cap \left \{ {\mathcal{M}}_{\mathcal{B}_{4n}} (U^2) > \lambda^2 \right \}} {\mathcal{M}}_{\mathcal{B}_{4n}} (U^2 ) d\mu
\end{align*}
and by the same reasoning 
\begin{align*}
& \sum_{\mathcal{K} \in [\mathcal{N}^1_{\lambda} (\mathcal{H})]_{i,j}} \int_{P_h \mathcal{K} \cap \left \{ {\mathcal{M}}_{\mathcal{B}_{4n}} (G^{q_0}) > \lambda^{q_0} \right \}} {\mathcal{M}}_{\mathcal{B}_{4n}} (G^{q_0} ) d\mu \\ \leq & C_1 2^{n(i+j)} \int_{\mathcal{H} \cap \left \{ {\mathcal{M}}_{\mathcal{B}_{4n}} (G^{q_0}) > \lambda^{q_0} \right \}} {\mathcal{M}}_{\mathcal{B}_{4n}} (G^{q_0} ) d\mu.
\end{align*}
In addition, by arguing similarly, the last two displays clearly also hold for $[\mathcal{N}^1_{\lambda} (\mathcal{H})]_{i,j}$ replaced by $[\mathcal{N}^2_{\lambda} (\mathcal{H})]_{i,j}$.
Therefore, for $h \in \{1,2\}$ we deduce
\begin{align*}
& \sum_{i,j=0}^\infty \left (\frac{1}{2^{i+j}} \right )^{n+2s} \sum_{\mathcal{K} \in [\mathcal{N}^h_{\lambda} (\mathcal{H})]_{i,j}} \int_{P_h \mathcal{K} \cap \left \{ {\mathcal{M}}_{\mathcal{B}_{4n}} (U^2) > \lambda^2 \right \}} {\mathcal{M}}_{\mathcal{B}_{4n}} (U^2 ) d\mu \\
\leq & C_1 \sum_{i,j=0}^\infty \left (\frac{1}{2^{i+j}} \right )^{2s} \int_{\mathcal{H} \cap \left \{ {\mathcal{M}}_{\mathcal{B}_{4n}} (U^2) > \lambda^2 \right \}} {\mathcal{M}}_{\mathcal{B}_{4n}} (U^2 ) d\mu \\
\leq & C_0 \int_{\mathcal{H} \cap \left \{ {\mathcal{M}}_{\mathcal{B}_{4n}} (U^2) > \lambda^2 \right \}} {\mathcal{M}}_{\mathcal{B}_{4n}} (U^2 ) d\mu,
\end{align*}
where $C_0=C_1 \left (\frac{1}{1-2^{-2s}} \right )^2<\infty$. Similarly, we also have
\begin{align*}
& \sum_{i,j=0}^\infty \left (\frac{1}{2^{i+j}} \right )^{n+2s} \sum_{\mathcal{K} \in [\mathcal{N}^h_{\lambda} (\mathcal{H})]_{i,j}} \int_{P_h \mathcal{K} \cap \left \{ {\mathcal{M}}_{\mathcal{B}_{4n}} (G^{q_0}) > \lambda^{q_0} \right \}} {\mathcal{M}}_{\mathcal{B}_{4n}} (G^{q_0} ) d\mu \\
\leq & C_0 \int_{\mathcal{H} \cap \left \{ {\mathcal{M}}_{\mathcal{B}_{4n}} (G^{q_0}) > \lambda^{q_0} \right \}} {\mathcal{M}}_{\mathcal{B}_{4n}} (G^{q_0} ) d\mu.
\end{align*}
By combining the last two displays with (\ref{chebappU}) and (\ref{chebappG}), we finally arrive at the estimate (\ref{combineq}) with respect to $C_0$. \par
\emph{Step 2: Summation.}
For any $\mathcal{K} \in \mathcal{N}_\lambda$, we either have $\mathcal{K} \in \mathcal{M}^1_\lambda \cap \mathcal{N}^2_\lambda$, $\mathcal{K} \in \mathcal{M}^2_\lambda \cap \mathcal{N}^1_\lambda$ or $\mathcal{K} \in \mathcal{N}^1_\lambda \cap \mathcal{N}^2_\lambda$.
If $\mathcal{K} \in \mathcal{M}^1_\lambda \cap \mathcal{N}^2_\lambda$, then in a similar way as in the proof of Lemma \ref{softest}, by using (\ref{nondiag}) and taking into account (\ref{phiesta}), we have
\begin{align*}
\mu(\mathcal{K})
\leq & (\sqrt{n})^{n+2\theta} \bigg( \mu \left (\left \{ (x,y) \in \mathcal{K} \mid  {\mathcal{M}}_{\mathcal{B}_{4n}} (U^2)(x,y) > \lambda^2 \right \} \right ) \\ 
& + C_2 \varepsilon \mu(\mathcal{K})\\
&+ \phi(\mathcal{K})^{n-2 \theta} \mu \left (\left \{ (x,y) \in P_2 \mathcal{K} \mid  {\mathcal{M}}_{\mathcal{B}_{4n}} (U^2)(x,y) > N_d^2 \phi(\mathcal{K})^{-2(\theta+s)} \lambda^2 \right \} \right ) \\
&+ \phi(\mathcal{K})^{n-2 \theta} \mu \left (\left \{ (x,y) \in P_2\mathcal{K} \mid  {\mathcal{M}}_{\mathcal{B}_{4n}} (G^{q_0})(x,y) > \phi(\mathcal{K})^{-{q_0}(\theta+s)} \lambda^{q_0} \right \} \right ) \bigg ) ,
\end{align*}
so that in view of the restriction (\ref{epsrestr}) imposed on $\varepsilon$, reabsorbing the second term on the right-hand side into the left-hand side of the previous display yields
\begin{align*}
\mu(\mathcal{K})
\leq & C_3 \bigg( \mu \left (\left \{ (x,y) \in \mathcal{K} \mid  {\mathcal{M}}_{\mathcal{B}_{4n}} (U^2)(x,y) > \lambda^2 \right \} \right ) \\ 
&+ \phi(\mathcal{K})^{n-2 \theta} \mu \left (\left \{ (x,y) \in P_2 \mathcal{K} \mid  {\mathcal{M}}_{\mathcal{B}_{4n}} (U^2)(x,y) > N_d^2 \phi(\mathcal{K})^{-2(\theta+s)} \lambda^2 \right \} \right ) \\
&+ \phi(\mathcal{K})^{n-2 \theta} \mu \left (\left \{ (x,y) \in P_2\mathcal{K} \mid  {\mathcal{M}}_{\mathcal{B}_{4n}} (G^{q_0})(x,y) > \phi(\mathcal{K})^{-{q_0}(\theta+s)} \lambda^{q_0} \right \} \right ) \bigg ) ,
\end{align*}
where $C_3=C_3(n,\theta)>0$.
By a similar argument, we also obtain that for any $\mathcal{K} \in \mathcal{M}^2_\lambda \cap \mathcal{N}^1_\lambda$, we have
\begin{align*}
\mu(\mathcal{K})
\leq & C_3 \bigg( \mu \left (\left \{ (x,y) \in \mathcal{K} \mid  {\mathcal{M}}_{\mathcal{B}_{4n}} (U^2)(x,y) > \lambda^2 \right \} \right ) \\ 
&+ \phi(\mathcal{K})^{n-2 \theta} \mu \left (\left \{ (x,y) \in P_1 \mathcal{K} \mid  {\mathcal{M}}_{\mathcal{B}_{4n}} (U^2)(x,y) > N_d^2 \phi(\mathcal{K})^{-2(\theta+s)} \lambda^2 \right \} \right ) \\
&+ \phi(\mathcal{K})^{n-2 \theta} \mu \left (\left \{ (x,y) \in P_1\mathcal{K} \mid  {\mathcal{M}}_{\mathcal{B}_{4n}} (G^{q_0})(x,y) > \phi(\mathcal{K})^{-{q_0}(\theta+s)} \lambda^{q_0} \right \} \right ) \bigg ) .
\end{align*}
By combining the last two displays with the fact that for any $\mathcal{K} \in \mathcal{N}^1_\lambda \cap \mathcal{N}^2_\lambda$ we have the estimate (\ref{nondiag}), we arrive at
\begin{align*}
& \sum_{\mathcal{K} \in \mathcal{N}_{\lambda}} \mu(\mathcal{K}) \\
\leq & C_4 \Bigg( \sum_{\mathcal{K} \in \mathcal{N}_{\lambda}} \mu \left (\left \{ (x,y) \in \mathcal{K} \mid  {\mathcal{M}}_{\mathcal{B}_{4n}} (U^2)(x,y) > \lambda^2 \right \} \right ) \\ 
& + \sum_{\mathcal{K} \in \mathcal{N}^1_{\lambda}} \phi(\mathcal{K})^{n-2 \theta} \mu \left (\left \{ (x,y) \in P_1\mathcal{K} \mid  {\mathcal{M}}_{\mathcal{B}_{4n}} (U^2)(x,y) > N_d^2 \phi(\mathcal{K})^{-2(\theta+s)} \lambda^2 \right \} \right ) \\
& + \sum_{\mathcal{K} \in \mathcal{N}^2_{\lambda}} \phi(\mathcal{K})^{n-2 \theta} \mu \left (\left \{ (x,y) \in P_2\mathcal{K} \mid  {\mathcal{M}}_{\mathcal{B}_{4n}} (U^2)(x,y) > N_d^2 \phi(\mathcal{K})^{-2(\theta+s)} \lambda^2 \right \} \right ) \\
& + \sum_{\mathcal{K} \in \mathcal{N}^1_{\lambda}} \phi(\mathcal{K})^{n-2 \theta} \mu \left (\left \{ (x,y) \in P_1 \mathcal{K} \mid  {\mathcal{M}}_{\mathcal{B}_{4n}} (G^{q_0})(x,y) > \phi(\mathcal{K})^{-{q_0}(\theta+s)} \lambda^{q_0} \right \} \right ) \\
& + \sum_{\mathcal{K} \in \mathcal{N}^2_{\lambda}} \phi(\mathcal{K})^{n-2 \theta} \mu \left (\left \{ (x,y) \in P_2\mathcal{K} \mid  {\mathcal{M}}_{\mathcal{B}_{4n}} (G^{q_0})(x,y) > \phi(\mathcal{K})^{-{q_0}(\theta+s)}\lambda^{q_0} \right \} \right ) \Bigg ),
\end{align*}
where $C_4=C_4(n,s,\theta)>0$.
Using the disjointness of the cubes $\mathcal{K} \in \mathcal{N}_{\lambda}$ and then Chebychev's inequality, for the first term on the right-hand side of the previous display, we deduce
\begin{align*}
\sum_{\mathcal{K} \in \mathcal{N}_{\lambda}} \mu \left (\left \{ (x,y) \in \mathcal{K} \mid  {\mathcal{M}}_{\mathcal{B}_{4n}} (U^2)(x,y) > \lambda^2 \right \} \right ) \leq & \mu \left (\left \{ (x,y) \in \mathcal{Q}_1 \mid  {\mathcal{M}}_{\mathcal{B}_{4n}} (U^2)(x,y) > \lambda^2 \right \} \right ) \\
\leq & \frac{1}{\lambda^2} \int_{\mathcal{Q}_1 \cap \left \{ {\mathcal{M}}_{\mathcal{B}_{4n}} (U^2) > \lambda^2 \right \}} {\mathcal{M}}_{\mathcal{B}_{4n}} (U^2 ) d\mu.
\end{align*}
Moreover, in view of (\ref{disdecomp}), (\ref{combineq}) and the disjointness of the cubes $\mathcal{H} \in P \mathcal{N}_{\lambda}$, for $h \in \{1,2\}$ we obtain that
\begin{align*}
& \sum_{\mathcal{K} \in \mathcal{N}^h_{\lambda}} \phi(\mathcal{K})^{n-2 \theta} \mu \left (\left \{ (x,y) \in P_h\mathcal{K} \mid  {\mathcal{M}}_{\mathcal{B}_{4n}} (U^2)(x,y) > N_d^2 \phi(\mathcal{K})^{-2(\theta+s)} \lambda^2 \right \} \right ) \\
&+ \sum_{\mathcal{K} \in \mathcal{N}^h_{\lambda}} \phi(\mathcal{K})^{n-2 \theta} \mu \left (\left \{ (x,y) \in P_h\mathcal{K} \mid  {\mathcal{M}}_{\mathcal{B}_{4n}} (G^{q_0})(x,y) > \phi(\mathcal{K})^{-{q_0}(\theta+s)} \lambda^{2} \right \} \right ) \\
= & \sum_{\mathcal{H} \in P \mathcal{N}_{\lambda}} \sum_{\mathcal{K} \in \mathcal{N}^h_{\lambda}(\mathcal{H})} \phi(\mathcal{K})^{n-2 \theta} \mu \left (\left \{ (x,y) \in P_h\mathcal{K} \mid  {\mathcal{M}}_{\mathcal{B}_{4n}} (U^2)(x,y) > N_d^2 \phi(\mathcal{K})^{-2(\theta+s)} \lambda^2 \right \} \right ) \\
&+ \sum_{\mathcal{H} \in P \mathcal{N}_{\lambda}} \sum_{\mathcal{K} \in \mathcal{N}^h_{\lambda}(\mathcal{H})} \phi(\mathcal{K})^{n-2 \theta} \mu \left (\left \{ (x,y) \in P_h\mathcal{K} \mid  {\mathcal{M}}_{\mathcal{B}_{4n}} (G^{q_0})(x,y) > \phi(\mathcal{K})^{-{q_0}(\theta+s)} \lambda^{q_0} \right \} \right ) \\
\leq & C_0 \sum_{\mathcal{H} \in P \mathcal{N}_{\lambda}} \left (\frac{1}{\lambda^2}  \int_{\mathcal{H} \cap \left \{ {\mathcal{M}}_{\mathcal{B}_{4n}} (U^2) > \lambda^2 \right \}} {\mathcal{M}}_{\mathcal{B}_{4n}} (U^2 ) d\mu + \frac{1}{\lambda^{q_0}} \int_{\mathcal{H} \cap \left \{ {\mathcal{M}}_{\mathcal{B}_{4n}} (G^{q_0}) > \lambda^{q_0} \right \}} {\mathcal{M}}_{\mathcal{B}_{4n}} (G^{q_0} ) d\mu \right ) \\
= & C_0 \left (\frac{1}{\lambda^2} \int_{\mathcal{Q}_1 \cap \left \{ {\mathcal{M}}_{\mathcal{B}_{4n}} (U^2) > \lambda^2 \right \}} {\mathcal{M}}_{\mathcal{B}_{4n}} (U^2 ) d\mu +\frac{1}{\lambda^{q_0}}  \int_{\mathcal{Q}_1 \cap \left \{ {\mathcal{M}}_{\mathcal{B}_{4n}} (G^{q_0}) > \lambda^{q_0} \right \}} {\mathcal{M}}_{\mathcal{B}_{4n}} (G^{q_0} ) d\mu \right ).
\end{align*}
The estimate (\ref{hard}) now follows directly by combining the last three displays.
\end{proof}

\subsection{Level set estimate}

By combining the above results, we are finally able to estimate the measure of the level set of ${\mathcal{M}}_{\mathcal{B}_{4n}} (U^2)$ in the whole cube $\mathcal{Q}_1$.

\begin{cor} \label{levelsetest}
Under all the assumptions made above, for any $\lambda \geq \lambda_0$ we have
\begin{align*}
& \mu \left (\left \{(x,y) \in \mathcal{Q}_{1} \mid  {\mathcal{M}}_{\mathcal{B}_{4n}} (U^2)(x,y) > N_{\varepsilon,q}^2 \lambda^2 \right \} \right ) \\ \leq & C \left (\frac{\varepsilon}{\lambda^2} \int_{\mathcal{Q}_1 \cap \left \{ {\mathcal{M}}_{\mathcal{B}_{4n}} (U^2) > \lambda^2 \right \}} {\mathcal{M}}_{\mathcal{B}_{4n}} (U^2 ) d\mu + \frac{1}{\delta^{q_0} \lambda^{q_0}} \int_{\mathcal{Q}_1 \cap \left \{ {\mathcal{M}}_{\mathcal{B}_{4n}} (G^{q_0}) > \delta^{q_0} \lambda^{q_0} \right \}} {\mathcal{M}}_{\mathcal{B}_{4n}} (G^{q_0} ) d\mu \right ),
\end{align*}
where $C=C(n,s,\theta)>0$ and $\delta = \delta(\varepsilon,n,s,\theta,\Lambda) \in (0,1)$ is given by Lemma \ref{mfuse}.
\end{cor}

\begin{proof}
In view of Lemma \ref{decompcover}, Lemma \ref{ondiagsum}, Lemma \ref{softest}, Lemma \ref{hardest} and Chebychev's inequality, we obtain
\begin{align*}
& \mu \left (\left \{(x,y) \in \mathcal{Q}_{1} \mid  {\mathcal{M}}_{\mathcal{B}_{4n}} (U^2)(x,y) > N_{\varepsilon,q}^2 \lambda^2 \right \} \right ) \\
\leq & C_1 \varepsilon \left(\kappa \sum_{(x,x) \in J_D} \mu(\mathcal{B}_{r_x}(x) \cap \mathcal{Q}_1) + \sum_{\mathcal{K} \in \mathcal{G}_{\lambda}} \mu(\mathcal{K}) + \sum_{\mathcal{K} \in \mathcal{N}_{\lambda}} \mu(\mathcal{K}) \right ) \\
\leq & C \left (\frac{\varepsilon}{\lambda^2} \int_{\mathcal{Q}_1 \cap \left \{ {\mathcal{M}}_{\mathcal{B}_{4n}} (U^2) > \lambda^2 \right \}} {\mathcal{M}}_{\mathcal{B}_{4n}} (U^2 ) d\mu + \frac{1}{\delta^{q_0} \lambda^{q_0}} \int_{\mathcal{Q}_1 \cap \left \{ {\mathcal{M}}_{\mathcal{B}_{4n}} (G^{q_0}) > \delta^{q_0} \lambda^{q_0} \right \}} {\mathcal{M}}_{\mathcal{B}_{4n}} (G^{q_0} ) d\mu \right ),
\end{align*}
where all constants depend only on $n,s$ and $\theta$. This finishes the proof.
\end{proof}

\section{$L^p$ estimates for $U$} \label{ape}
We first prove the estimate we are interested in on a fixed scale in the form of an priori estimate and under the additional assumption that $U$ satisfies the estimate (\ref{J}). In order to do this, we use the following standard alternative characterization of the $L^p$ norm which follows from Fubini's theorem in a straightforward way.
\begin{lem} \label{altchar}
	Let $\nu$ be a $\sigma$-finite measure on $\mathbb{R}^n$ and let $h:\Omega \rightarrow [0,+\infty]$ be a $\nu$-measurable function in a domain $\Omega \subset \mathbb{R}^n$. Then for any $0< \beta < \infty$, we have 
	$$ \int_{\Omega} h^\beta d\nu = \beta \int_0^{\infty} \lambda^{\beta-1} \nu \left ( \{x \in \Omega \mid h(x)>\lambda \} \right ) d\lambda.$$
\end{lem}
\begin{prop} \label{aprioriest}
	Let $q \in [2,p)$ and $\widetilde q \in (q_0,q^\star)$, where $q_0$ is given by (\ref{q0}). Then there exists some small enough $\delta = \delta(n,s,\theta,\Lambda,q,\widetilde q) > 0$ such that if $A \in \mathcal{L}_0(\Lambda)$ is $\delta$-vanishing in $\mathcal{B}_{4n}$ and $g \in W^{s,2}(\mathbb{R}^n)$ satisfies $G \in L^{\widetilde q}(\mathcal{B}_{4n},\mu)$, then for any weak solution $u \in W^{s,2}(\mathbb{R}^n)$ of the equation $L_{A}^\Phi u = (-\Delta)^s g$ in $B_{4n}$ that satisfies $U \in L^{\widetilde q}(\mathcal{B}_{4n},\mu)$ and the estimate (\ref{J}) in any ball contained in $B_{4n}$ with respect to $q$, we have
\begin{align*}
\left (\dashint_{\mathcal{B}_{1/2}} U^{\widetilde q} d\mu \right )^{\frac{1}{\widetilde q}} \leq & C \Bigg (\sum_{k=1}^\infty 2^{-k(s-\theta)} \left ( \dashint_{\mathcal{B}_{2^k 4n}} U^2 d\mu \right )^\frac{1}{2} \\
	& + \left ( \dashint_{\mathcal{B}_{4n}} G^{\widetilde q} d\mu \right )^\frac{1}{\widetilde q} + \sum_{k=1}^\infty 2^{-k(s-\theta)} \left ( \dashint_{\mathcal{B}_{2^k 4n}} G^2 d\mu \right )^\frac{1}{2} \Bigg ),
	\end{align*}
	where $C=C(n,s,\theta,\Lambda,q,\widetilde q,p)>0$.
\end{prop}

\begin{proof}
Let $\varepsilon$ to be chosen small enough and consider the corresponding $\delta = \delta(\varepsilon,n,s,\theta,\Lambda) > 0$ given by Lemma \ref{mfuse}. Then by using Lemma \ref{altchar} multiple times, first with $\beta=\widetilde q$, $h={\mathcal{M}}_{\mathcal{B}_{4n}} (U^2 )^\frac{1}{2}$ and $d\nu=d\mu$, then with $\beta=\widetilde q-2$, $h={\mathcal{M}}_{\mathcal{B}_{4n}} (U^2 )^\frac{1}{2}$ and $d\nu={\mathcal{M}}_{\mathcal{B}_{4n}} (U^2 )d\mu$, and also with $\beta=\widetilde q-q_0$,  $h={\mathcal{M}}_{\mathcal{B}_{4n}} (G^{q_0} )^\frac{1}{q_0}$ and $d\nu={\mathcal{M}}_{\mathcal{B}_{4n}} (G^{q_0} )d\mu$, a change of variables, Corollary \ref{levelsetest} and the definition of $N_{\varepsilon,q}$ from (\ref{Neps}), we obtain
\begin{align*}
& \int_{\mathcal{Q}_1} \left ({\mathcal{M}}_{\mathcal{B}_{4n}} (U^2 ) \right )^\frac{\widetilde q}{2} d\mu \\ = & \widetilde q \int_0^{\infty} \lambda^{\widetilde q-1} \mu \left ( \mathcal{Q}_1 \cap \left \{{\mathcal{M}}_{\mathcal{B}_{4n}} (U^2 ) > \lambda^2 \right \} \right ) d\lambda \\
= & \widetilde q N_{\varepsilon,q}^{\widetilde q} \int_0^{\infty} \lambda^{\widetilde q-1} \mu \left ( \mathcal{Q}_1 \cap \left \{{\mathcal{M}}_{\mathcal{B}_{4n}} (U^2 ) > N_{\varepsilon,q}^2 \lambda^2 \right \} \right ) d\lambda \\
= & \widetilde q N_{\varepsilon,q}^{\widetilde q} \int_0^{\lambda_0} \lambda^{\widetilde q-1} \mu \left ( \mathcal{Q}_1 \cap \left \{{\mathcal{M}}_{\mathcal{B}_{4n}} (U^2 ) > N_{\varepsilon,q}^2 \lambda^2 \right \} \right ) d\lambda \\ & + \widetilde q N_{\varepsilon,q}^{\widetilde q} \int_{\lambda_0}^{\infty} \lambda^{\widetilde q-1} \mu \left ( \mathcal{Q}_1 \cap \left \{{\mathcal{M}}_{\mathcal{B}_{4n}} (U^2 ) > N_{\varepsilon,q}^2 \lambda^2 \right \} \right ) d\lambda \\
\leq & \widetilde q N_{\varepsilon,q}^{\widetilde q} \mu(\mathcal{Q}_1) \lambda_0^{\widetilde q} \\
&+ C_1 \widetilde q N_{\varepsilon,q}^{\widetilde q} \varepsilon \int_{0}^{\infty} \lambda^{\widetilde q-3} \int_{\mathcal{Q}_1 \cap \left \{ {\mathcal{M}}_{\mathcal{B}_{4n}} (U^2) > \lambda^2 \right \}} {\mathcal{M}}_{\mathcal{B}_{4n}} (U^2 ) d\mu d\lambda \\
&+ C_1 \widetilde q N_{\varepsilon,q}^{\widetilde q}\delta^{-q_0} \int_{0}^{\infty} \lambda^{\widetilde q-q_0-1} \int_{\mathcal{Q}_1 \cap \left \{ {\mathcal{M}}_{\mathcal{B}_{4n}} (G^{q_0}) > \delta^{q_0} \lambda^{q_0} \right \}} {\mathcal{M}}_{\mathcal{B}_{4n}} (G^{q_0} ) d\mu d\lambda \\
= & \widetilde q N_{\varepsilon,q}^{\widetilde q} \mu(\mathcal{Q}_1) \lambda_0^{\widetilde q} \\
&+ C_1 \widetilde q C_{nd} C_{s,\theta} N_d 10^{10n} \varepsilon^{1-\widetilde q/q^\star} \int_{\mathcal{Q}_1} \left ({\mathcal{M}}_{\mathcal{B}_{4n}} (U^2 ) \right )^\frac{\widetilde q}{2} d\mu \\
&+ C_1 \widetilde q N_{\varepsilon,q}^{\widetilde q} \delta^{-q_0} \int_{\mathcal{Q}_1} \left ({\mathcal{M}}_{\mathcal{B}_{4n}} (G^{q_0} ) \right )^\frac{\widetilde q}{q_0} d\mu,
\end{align*}
where $C_1=C_1(n,s,\theta) \geq 1$. Now we set
$$ \varepsilon := \min \left \{ \left ( 4 (\sqrt{n})^{n+2\theta} C_2 \right )^{-1}, \left (2 C_1 \widetilde q C_{nd} C_{s,\theta} N_d 10^{10n} \right )^{-\frac{q^\star}{q^\star-\widetilde q}} \right \}, $$
so that $\varepsilon$ satisfies the restriction (\ref{epsrestr}) and moreover, we have
$$ C_1 \widetilde q C_{nd} C_{s,\theta} N_d 10^{10n} \varepsilon^{1-\widetilde q/q^\star} \leq \frac{1}{2}.$$ Since in addition by assumption we have $U \in L^{\widetilde q}(\mathcal{B}_{4n},\mu)$, by Proposition \ref{Maxfun} we have $$\int_{\mathcal{Q}_1} \left ({\mathcal{M}}_{\mathcal{B}_{4n}} (U^2 ) \right )^\frac{\widetilde q}{2} d\mu < \infty,$$ so that we can reabsorb the second to last term on the right-hand side of the first display of the proof in the the left-hand side, which yields
\begin{align*}
\int_{\mathcal{Q}_1} \left ({\mathcal{M}}_{\mathcal{B}_{4n}} (U^2 ) \right )^\frac{\widetilde q}{2} d\mu
\leq & 2 \widetilde q N_{\varepsilon,q}^{\widetilde q} \mu(\mathcal{Q}_1) \lambda_0^{\widetilde q}+ 2 C_1 \widetilde q N_{\varepsilon,q}^{\widetilde q} \delta^{-q_0} \int_{\mathcal{Q}_1} \left ({\mathcal{M}}_{\mathcal{B}_{4n}} (G^{q_0} ) \right )^\frac{\widetilde q}{q_0} d\mu.
\end{align*}
Now in view of Corollary \ref{maxdominate} and Proposition \ref{Maxfun}, taking into account the definition of $\lambda_0$ from (\ref{tailcontrol}) and using H\"older's inequality, we obtain
\begin{align*}
& \dashint_{\mathcal{B}_{1/2}} U^{\widetilde q} d\mu
\leq \frac{1}{\mu \left (\mathcal{B}_{1/2} \right )} \int_{\mathcal{Q}_1} \left ({\mathcal{M}}_{\mathcal{B}_{4n}} (U^2 ) \right )^\frac{\widetilde q}{2} d\mu \\ 
\leq & C_2 \left ( \lambda_0^{\widetilde q}+ \int_{\mathcal{Q}_1} \left ({\mathcal{M}}_{\mathcal{B}_{4n}} (G^{q_0} ) \right )^\frac{\widetilde q}{q_0} d\mu \right ) \\
\leq & C_3 \Bigg (\sum_{k=1}^\infty 2^{-k(s-\theta)} \left (\dashint_{\mathcal{B}_{2^k 4n}} U^2 d\mu \right )^\frac{1}{2} + \sum_{k=1}^\infty 2^{-k(s-\theta)} \left (\dashint_{\mathcal{B}_{2^k 4n}} G^2 d\mu \right )^\frac{1}{2} +\left (\dashint_{\mathcal{B}_{4n}} G^{q_0} d\mu \right )^\frac{1}{q_0} \Bigg )^{\widetilde q} \\
&+ C_3 \int_{\mathcal{B}_{4n}} G^{\widetilde q} d\mu \\
\leq & C_4 \Bigg (\sum_{k=1}^\infty 2^{-k(s-\theta)} \left (\dashint_{\mathcal{B}_{2^k 4n}} U^2 d\mu \right )^\frac{1}{2} + \sum_{k=1}^\infty 2^{-k(s-\theta)} \left (\dashint_{\mathcal{B}_{2^k 4n}} G^2 d\mu \right )^\frac{1}{2} \Bigg )^{\widetilde q} + C_4 \dashint_{\mathcal{B}_{4n}} G^{\widetilde q} d\mu,
\end{align*}
where all constants depend only on $n,s,\theta,\Lambda,q,\widetilde q$ and $p$. This proves the desired estimate with $C=C_4^{1/\widetilde q}$.
\end{proof}

\begin{cor} \label{aprioriestcor}
	Consider some $q \in [2,p)$ and some $\widetilde q \in (q_0,q^\star)$. Then there exists some small enough $\delta = \delta(n,s,\theta,\Lambda,q,\widetilde q) > 0$ such that if $A \in \mathcal{L}_0(\Lambda)$ is $\delta$-vanishing in $B_1$ and $g \in W^{s,2}(\mathbb{R}^n)$ satisfies $G \in L^{\widetilde q}(\mathcal{B}_1,\mu)$, then for any weak solution $u \in W^{s,2}(\mathbb{R}^n)$ of the equation $L_{A}^\Phi u = (-\Delta)^s g$ in $B_1$ that satisfies $U \in L^{\widetilde q}(\mathcal{B}_1,\mu)$ and the estimate (\ref{J}) in any ball contained in $B_1$ with respect to $q$, we have the estimate
	\begin{align*}
	\left (\dashint_{\mathcal{B}_{1/2}} U^{\widetilde q} d\mu \right )^{\frac{1}{\widetilde q}} \leq & C \Bigg (\sum_{k=1}^\infty 2^{-k(s-\theta)} \left ( \dashint_{\mathcal{B}_{2^k}} U^2 d\mu \right )^\frac{1}{2} \\
	& + \left ( \dashint_{\mathcal{B}_{1}} G^{\widetilde q} d\mu \right )^\frac{1}{\widetilde q} + \sum_{k=1}^\infty 2^{-k(s-\theta)} \left ( \dashint_{\mathcal{B}_{2^k}} G^2 d\mu \right )^\frac{1}{2} \Bigg ),
	\end{align*}
	where $C=C(n,s,\theta,\Lambda,q,\widetilde q,p)>0$.
\end{cor}

\begin{proof}
	There exists some small enough radius $r_1 \in \left (0,1 \right)$ such that
	\begin{equation} \label{rz}
	B_{4n r_1}(z) \Subset B_1
	\end{equation}
	for any $z \in B_{1/2}$. Now fix some $z \in B_{1/2}$ and
	consider the scaled functions $u_z,g_z \in W^{s,2}(\mathbb{R}^n)$ given by
	$$ u_z(x):=u(r_1 x+z), \quad g_z(x):=g(r_1 x+z), \quad A_z(x,y):= A(r_1 x+z,r_1 y+z).$$
	Since $A$ is $\delta$-vanishing in $\mathcal{B}_1$, we see that $A_z$ clearly is $\delta$-vanishing in $B_{\frac{1}{4nr_1}}(-z) \supset B_{4n}$.
	Furthermore, in view of (\ref{rz}), $u_z$ is a weak solution of $L_{A_z}^\Phi u_z = (-\Delta)^s g_z$ in $B_{\frac{1}{4nr_1}}(-z) \supset B_{4n}$. Now fix some $r>0$ and some $x_0 \in \mathbb{R}^n$ such that $B_r(x_0) \subset B_{4n}$. Then again in view of (\ref{rz}), we clearly have $$B_{r_1 r}(r_1 x_0+z) \subset B_1,$$ so that by the assumption that the estimate (\ref{J}) holds for any ball contained in $B_1$, the estimate (\ref{J}) holds with respect to the ball $B_{r_1 r}(r_1 x_0+z)$. Together with changes of variables, we see that for the functions
	\begin{align*}
	U_z(x,y):=\frac{|u_z(x)-u_z(y)|}{|x-y|^{s+\theta}}, \quad G_z(x,y):=\frac{|g_z(x)-g_z(y)|}{|x-y|^{s+\theta}},
	\end{align*}
	we have
	\begin{equation} \label{sca1}
	\begin{aligned}
	& \left (\dashint_{\mathcal{B}_{r/2}(x_0)} U_z^{q} d\mu \right )^{\frac{1}{q}} = \left ( \frac{1}{\mu(\mathcal{B}_{r/2}(x_0))} \int_{B_{r/2}(x_0)} \int_{B_{r/2}(x_0)} \frac{|u_z(x)-u_z(y)|^{q}}{|x-y|^{n-2\theta+q(s+\theta)}}dydx \right )^{\frac{1}{q}} \\
	= & C_1 \left (\frac{r_1^{n-2\theta+q(s+\theta)}}{r_1^{2n}} \frac{1}{\mu(\mathcal{B}_{r/2})} \int_{B_{\frac{r_1r}{2}}(r_1 x_0+z)} \int_{B_{\frac{r_1r}{2}}(r_1 x_0+z)} \frac{|u(x)-u(y)|^{q}}{|x-y|^{n-2\theta+q(s+\theta)}}dydx \right )^{\frac{1}{q}} \\
	= & C_1 \Bigg (r_1^{-n-2\theta+q(s+\theta)} \underbrace{\frac{\mu \left (\mathcal{B}_{r_1r/2} \right )}{\mu(\mathcal{B}_{r/2})}}_{=r_1^{n+2\theta}} \dashint_{\mathcal{B}_{\frac{r_1r}{2}}(r_1 x_0+z)} U^{q} d\mu \Bigg )^{\frac{1}{q}} \\
	= & C_1 r_1^{s+\theta} \left (\dashint_{\mathcal{B}_{\frac{r_1r}{2}}(r_1 x_0+z)} U^{q} d\mu \right )^{\frac{1}{q}} \\ \leq & C_2 r_1^{s+\theta} \Bigg ( \sum_{k=1}^\infty 2^{-k(s-\theta)} \left ( \dashint_{\mathcal{B}_{2^kr_1 r}(r_1 x_0+z x_0)} U^2 d\mu \right )^\frac{1}{2} \\ & + \left (\dashint_{\mathcal{B}_{r_1 r}(r_1 x_0+z)} G^{q_0} d\mu \right )^\frac{1}{q_0} + \sum_{k=1}^\infty 2^{-k(s-\theta)} \left ( \dashint_{\mathcal{B}_{2^kr_1 r}(r_1 x_0+z x_0)} G^2 d\mu \right )^\frac{1}{2} \Bigg ) \\
	= & C_3 r_1^{s+\theta} \Bigg ( \sum_{k=1}^\infty 2^{-k(s-\theta)} \left (\frac{r_1^{2n}}{r_1^{n+2s}} \frac{\mu(\mathcal{B}_{2^kr})}{\mu(\mathcal{B}_{2^kr_1 r})} \dashint_{\mathcal{B}_{2^kr}(x_0)} U_z^2 d\mu \right )^\frac{1}{2} \\ & + \left (\frac{r_1^{2n}}{r_1^{n-2\theta+q_0(s+\theta)}} \frac{\mu(\mathcal{B}_{r})}{\mu(\mathcal{B}_{r_1 r})} \dashint_{\mathcal{B}_{r}(x_0)} G_z^{q_0} d\mu \right )^\frac{1}{q_0} \\ & + \sum_{k=1}^\infty 2^{-k(s-\theta)} \left (\frac{r_1^{2n}}{r_1^{n+2s}} \frac{\mu(\mathcal{B}_{2^kr})}{\mu(\mathcal{B}_{2^kr_1 r})} \dashint_{\mathcal{B}_{2^kr}(x_0)} G_z^2 d\mu \right )^\frac{1}{2} \Bigg ) \\
	= & C_3 \Bigg ( \sum_{k=1}^\infty 2^{-k(s-\theta)} \left ( \dashint_{\mathcal{B}_{2^kr}(x_0)} U_z^2 d\mu \right )^\frac{1}{2} \\ & + \left (\dashint_{\mathcal{B}_{r}(x_0)} G_z^{q_0} d\mu \right )^\frac{1}{q_0} + \sum_{k=1}^\infty 2^{-k(s-\theta)} \left ( \dashint_{\mathcal{B}_{2^kr}(x_0)} G_z^2 d\mu \right )^\frac{1}{2} \Bigg ),
	\end{aligned}
	\end{equation}
	where all constants depend only on $q,n,s,\theta$ and $\Lambda$.
	Therefore, we see that $U_z$ and $G_z$ satisfy the estimate (\ref{J}) in any ball that is contained in $B_{4n}$. Since in addition the assumption that $U \in L^{\widetilde q}(\mathcal{B}_1,\mu)$ clearly implies that $U_z \in L^{\widetilde q} \left (\mathcal{B}_{{\frac{1}{4nr_1}}(-z)},\mu \right ) \subset L^{\widetilde q}\left (\mathcal{B}_{4n} ,\mu \right )$, by Proposition \ref{aprioriest} we obtain that
	\begin{align*}
	\left (\dashint_{\mathcal{B}_{1/2}} U_z^{\widetilde q} d\mu \right )^{\frac{1}{\widetilde q}} \leq & C_4 \Bigg (\sum_{k=1}^\infty 2^{-k(s-\theta)} \left ( \dashint_{\mathcal{B}_{2^k 4n}} U_z^2 d\mu \right )^\frac{1}{2} \\
	& + \left ( \dashint_{\mathcal{B}_{4n}} G_z^{\widetilde q} d\mu \right )^\frac{1}{\widetilde q} + \sum_{k=1}^\infty 2^{-k(s-\theta)} \left ( \dashint_{\mathcal{B}_{2^k 4n}} G_z^2 d\mu \right )^\frac{1}{2} \Bigg ),
	\end{align*}
	where $C_4=C_4(n,s,\theta,\Lambda,q,\widetilde q,p)>0$. By combining the last display with changes of variables, we deduce
	\begin{equation} \label{scalingarg}
	\begin{aligned}
	& \left (\dashint_{\mathcal{B}_{r_1/2}(z)} U^{\widetilde q} d\mu \right )^{\frac{1}{\widetilde q}} = \left ( \frac{1}{\mu(\mathcal{B}_{r_1/2}(z))} \int_{B_{r_1/2}(z)} \int_{B_{r_1/2}(z)} \frac{|u(x)-u(y)|^{\widetilde q}}{|x-y|^{n-2\theta+q(s+\theta)}}dydx \right )^{\frac{1}{\widetilde q}} \\ 
	= & C_5 \left (r_1^{-n-2\theta} \frac{r_1^{2n}}{r_1^{n-2\theta+q(s+\theta)}} \int_{B_{1/2}} \int_{B_{1/2}} \frac{|u_z(x)-u_z(y)|^{\widetilde q}}{|x-y|^{n-2\theta+q(s+\theta)}}dydx \right )^{\frac{1}{\widetilde q}} \\
	= & C_6 r_1^{-s-\theta} \left (\dashint_{\mathcal{B}_{1/2}} U_z^{\widetilde q} d\mu \right )^{\frac{1}{\widetilde q}} \\
	\leq & C_7 r_1^{-s-\theta} \Bigg (\sum_{k=1}^\infty 2^{-k(s-\theta)} \left ( \dashint_{\mathcal{B}_{2^k 4n}} U_z^2 d\mu \right )^\frac{1}{2} \\
	& + \left ( \dashint_{\mathcal{B}_{4n}} G_z^{\widetilde q} d\mu \right )^\frac{1}{\widetilde q} + \sum_{k=1}^\infty 2^{-k(s-\theta)} \left ( \dashint_{\mathcal{B}_{2^k 4n}} G_z^2 d\mu \right )^\frac{1}{2} \Bigg ) \\
	= & C_7 r_1^{-s-\theta} \Bigg (\sum_{k=1}^\infty 2^{-k(s-\theta)} \left ( \frac{r_1^{n+2s}}{r_1^{2n}} \frac{\mu(\mathcal{B}_{2^k 4n r_1})}{\mu(\mathcal{B}_{2^k 4n})} \dashint_{\mathcal{B}_{2^k 4nr_1(z)}} U^2 d\mu \right )^\frac{1}{2} \\
	& + \left ( \frac{r_1^{n-2\theta+\widetilde q(s+\theta)}}{r_1^{2n}} \frac{\mu(\mathcal{B}_{4n r_1})}{\mu(\mathcal{B}_{4n})} \dashint_{\mathcal{B}_{4n r_1(z)}} G^{\widetilde q} d\mu \right )^\frac{1}{\widetilde q} \\ & + \sum_{k=1}^\infty 2^{-k(s-\theta)} \left ( \frac{r_1^{n+2s}}{r_1^{2n}} \frac{\mu(\mathcal{B}_{2^k 4n r_1})}{\mu(\mathcal{B}_{2^k 4n})} \dashint_{\mathcal{B}_{2^k 4nr_1(z)}} G^2 d\mu \right )^\frac{1}{2} \Bigg ) \\
	= & C_7 \Bigg (\sum_{k=1}^\infty 2^{-k(s-\theta)} \left ( \dashint_{\mathcal{B}_{2^k 4nr_1(z)}} U^2 d\mu \right )^\frac{1}{2} \\
	& + \left ( \dashint_{\mathcal{B}_{4nr_1(z)}} G^{\widetilde q} d\mu \right )^\frac{1}{\widetilde q} + \sum_{k=1}^\infty 2^{-k(s-\theta)} \left ( \dashint_{\mathcal{B}_{2^k 4nr_1(z)}} G^2 d\mu \right )^\frac{1}{2} \Bigg ),
	\end{aligned}
	\end{equation}
	where all constants depend only on $q,\widetilde q,p,n,s,\theta$ and $\Lambda$. Since $\left \{B_{r_1/2}(z) \right \}_{z \in B_{1/2}}$ is an open covering of $\overline B_{1/2}$ and $\overline B_{1/2}$ is compact, there is a finite subcover $\left \{B_{r_1/2}(z_j) \right \}_{j=1}^m$ of $\overline B_{1/2}$ and hence also of $B_{1/2}$. In particular, $\left \{\mathcal{B}_{r/2}(z_j) \right \}_{j=1}^m$ is a finite subcover of $\mathcal{B}_{1/2}$. Therefore, by summing up the above estimates, we obtain
	\begin{align*}
	& \left (\dashint_{\mathcal{B}_{1/2}} U^{\widetilde q} d\mu \right )^{\frac{1}{\widetilde q}} \leq \sum_{j=1}^m 
	\left (\dashint_{\mathcal{B}_{r_1/2}(z_j)} U^{\widetilde q} d\mu \right )^{\frac{1}{\widetilde q}} \\
	\leq & C_7 \sum_{j=1}^m \Bigg (\sum_{k=1}^\infty 2^{-k(s-\theta)} \left ( \dashint_{\mathcal{B}_{2^k 4nr_1({z_j})}} U^2 d\mu \right )^\frac{1}{2} \\
	& + \left ( \dashint_{\mathcal{B}_{4nr_1({z_j})}} G^{\widetilde q} d\mu \right )^\frac{1}{\widetilde q} + \sum_{k=1}^\infty 2^{-k(s-\theta)} \left ( \dashint_{\mathcal{B}_{2^k 4nr_1({z_j})}} G^2 d\mu \right )^\frac{1}{2} \Bigg ) \\
	\leq & C_8 \sum_{j=1}^m \Bigg (\sum_{k=1}^\infty 2^{-k(s-\theta)} \left ( \frac{\mu(\mathcal{B}_{2^k4n})}{\mu(\mathcal{B}_{2^k 4nr_1})} \dashint_{\mathcal{B}_{2^k}} U^2 d\mu \right )^\frac{1}{2} \\
	& + \left ( \frac{\mu(\mathcal{B}_{4n})}{\mu(\mathcal{B}_{4nr_1})} \dashint_{\mathcal{B}_{1}} G^{\widetilde q} d\mu \right )^\frac{1}{\widetilde q} + \sum_{k=1}^\infty 2^{-k(s-\theta)} \left ( \frac{\mu(\mathcal{B}_{2^k})}{\mu(\mathcal{B}_{2^k 4nr_1})} \dashint_{\mathcal{B}_{2^k 4n}} G^2 d\mu \right )^\frac{1}{2} \Bigg ) \\
	\leq & C_8 m r_1^{-n/2-\theta} \Bigg (\sum_{k=1}^\infty 2^{-k(s-\theta)} \left ( \dashint_{\mathcal{B}_{2^k}} U^2 d\mu \right )^\frac{1}{2} \\
	& + \left ( \dashint_{\mathcal{B}_{1}} G^{\widetilde q} d\mu \right )^\frac{1}{\widetilde q} + \sum_{k=1}^\infty 2^{-k(s-\theta)} \left ( \dashint_{\mathcal{B}_{2^k}} G^2 d\mu \right )^\frac{1}{2} \Bigg ),
	\end{align*}
	where $C_8=C_8(n,s,\theta,\Lambda,q,\widetilde q,p)>0$. Since in addition $m$ and $r_1$ depend only on $B_{1/2}$ and thus only on $n$, the proof is finished.
\end{proof}

\begin{cor} \label{aprioriestcorscaled}
	Let $r>0$ and $z \in \mathbb{R}^n$ and consider some $q \in [2,p)$ and some $\widetilde q \in (q_0,q^\star)$. Then there exists some small enough $\delta = \delta(n,s,\theta,\Lambda,q,\widetilde q) > 0$ such that if $A \in \mathcal{L}_0(\Lambda)$ is $\delta$-vanishing in $\mathcal{B}_r(z)$ and $g \in W^{s,2}(\mathbb{R}^n)$ satisfies $G \in L^{\widetilde q}(\mathcal{B}_r(z),\mu)$, then for any weak solution $u \in W^{s,2}(\mathbb{R}^n)$ of the equation $L_{A}^\Phi u = (-\Delta)^s g$ in $B_r(z)$ that satisfies $U \in L^{\widetilde q}(\mathcal{B}_r(z),\mu)$ and the estimate (\ref{J}) in any ball contained in $B_r(z)$ with respect to $q$, we have the estimate
	\begin{align*}
	\left (\dashint_{\mathcal{B}_{r/2}(z)} U^{\widetilde q} d\mu \right )^{\frac{1}{\widetilde q}} \leq & C \Bigg (\sum_{k=1}^\infty 2^{-k(s-\theta)} \left ( \dashint_{\mathcal{B}_{2^k r(z)}} U^2 d\mu \right )^\frac{1}{2} \\
	& + \left ( \dashint_{\mathcal{B}_{r}(z)} G^{\widetilde q} d\mu \right )^\frac{1}{\widetilde q} + \sum_{k=1}^\infty 2^{-k(s-\theta)} \left ( \dashint_{\mathcal{B}_{2^k r}(z)} G^2 d\mu \right )^\frac{1}{2} \Bigg ),
	\end{align*}
	where $C=C(n,s,\theta,\Lambda,q,\widetilde q,p)>0$.
\end{cor}
\begin{rem} \normalfont
	In the above estimate, it is essential that the constant $C$ does not depend on $r$ and $z$.
\end{rem}

\begin{proof}
Consider the scaled functions $u_1,g_1 \in W^{s,2}(\mathbb{R}^n)$ given by
$$ u_1(x):=u(r x+z), \quad g_1(x):=g(r x+z)$$
and also 
$$A_1(x,y):= A(r x+z,r y+z).$$
Since $A$ is $\delta$-vanishing in $\mathcal{B}_r(z)$, $A_1$ clearly is $\delta$-vanishing in $B_1$. Also, in view of a change of variables, for $U_1(x,y):=\frac{u_1(x)-u_1(y)}{|x-y|^{s+\theta}}$ we clearly have $U_1 \in L^{\widetilde q}(\mathcal{B}_1,\mu)$. Moreover, since $u$ satisfies the estimate (\ref{J}) in any ball contained in $B_r(z)$, by a scaling argument similar to the first one in the proof of Corollary \ref{aprioriestcor} (see (\ref{sca1})), we deduce that $u_1$ satisfies the estimate (\ref{J}) in any ball contained in $B_1$.
In addition, $u_1$ is a weak solution of $L_{A_1}^\Phi u_1 = (-\Delta)^s g_1$ in $B_1$. Therefore, $u_1$ and $A_1$ satisfy all assumptions from Corollary \ref{aprioriestcor}. The desired estimate now follows by applying essentially the same scaling argument as the second one in the proof of Corollary \ref{aprioriestcor} (see (\ref{scalingarg})), by replacing $u_z$ by $u_1$, $g_z$ by $g_1$, $A_z$ by $A_1$, replacing the number $4n$ by $1$ and applying Corollary \ref{aprioriestcor} to $u_1$ along the way.
\end{proof}

\begin{prop} \label{aprioriestx}
	Let $r>0$, $z \in \mathbb{R}^n$, $s \in (0,1)$ and $p \in (2,\infty)$. Then there exists some small enough $\delta = \delta(n,s,\theta,\Lambda,p) > 0$ such that if $A \in \mathcal{L}_0(\Lambda)$ is $\delta$-vanishing in $\mathcal{B}_{r_1}(z)$ and $g \in W^{s,2}(\mathbb{R}^n)$ satisfies $G \in L^{p}(\mathcal{B}_{r}(z),\mu)$, then for any weak solution $u \in W^{s,2}(\mathbb{R}^n)$ of the equation $L_{A}^\Phi u = (-\Delta)^s g$ in $B_{r}(z)$ that satisfies $U \in L^{p}(\mathcal{B}_{r}(z),\mu)$, we have the estimate
	\begin{equation} \label{estp}
	\begin{aligned}
	\left (\dashint_{\mathcal{B}_{r/2}(z)} U^{p} d\mu \right )^{\frac{1}{p}} \leq & C \Bigg (\sum_{k=1}^\infty 2^{-k(s-\theta)} \left ( \dashint_{\mathcal{B}_{2^k r}(z)} U^2 d\mu \right )^\frac{1}{2} \\
	& + \left ( \dashint_{\mathcal{B}_{r}(z)} G^{p} d\mu \right )^\frac{1}{p} + \sum_{k=1}^\infty 2^{-k(s-\theta)} \left ( \dashint_{\mathcal{B}_{2^k r}(z)} G^2 d\mu \right )^\frac{1}{2} \Bigg ),
	\end{aligned}
	\end{equation}
	where $C=C(n,s,\theta,\Lambda,p)>0$.
\end{prop}

\begin{proof}
Define iteratively a sequence $\{q_i\}_{i=0}^\infty$ of real numbers by
$$ q_0:=2, \quad q_{i+1}:= \min \{(q_i+(q_i)^\star)/2,p\},$$
where as in (\ref{qstar}) we let
\begin{align*}
(q_i)^\star=\begin{cases} 
\frac{nq_i}{n-sq_i}, & \text{if } n>sq_i \\
2p, & \text{if } n \leq sq_i.
\end{cases}
\end{align*}
Since for any $i$ with $n>sq_{i+1}$ we have
$$ \left (q_i+\frac{nq_i}{n-sq_i} \right )/2 -q_i =\frac{nq_{i}}{2(n-sq_{i})} - \frac{q_i}{2} \geq \frac{4s}{2(n-s)}>0,$$
there clearly exists some $i_p \in \mathbb{N}$ such that $q_{i_p} = p$. \newline
Since the estimate (\ref{J}) is trivially satisfied for $q=q_0=2$, and in view of the additional assumption that $U \in L^{p}(\mathcal{B}_{r}(z),\mu)$ we in particular have $U \in L^{q_1}(\mathcal{B}_{r}(z),\mu)$, if we choose $\delta$ small enough such that Corollary \ref{aprioriestcorscaled} is applicable with $q=2$ and $\widetilde q = q_1$, then all assumptions of Corollary \ref{aprioriestcorscaled} are satisfied with respect to $q=q_0=2$ and $\widetilde q = q_1\in (q_0,(q_0)^\star)$, so that we obtain
\begin{equation} \label{q1est}
\begin{aligned}
\left (\dashint_{\mathcal{B}_{r/2}(z)} U^{q_1} d\mu \right )^{\frac{1}{q_1}} \leq & C \Bigg (\sum_{k=1}^\infty 2^{-k(s-\theta)} \left ( \dashint_{\mathcal{B}_{2^k r(z)}} U^2 d\mu \right )^\frac{1}{2} \\
& + \left ( \dashint_{\mathcal{B}_{r}(z)} G^{q_1} d\mu \right )^\frac{1}{q_1} + \sum_{k=1}^\infty 2^{-k(s-\theta)} \left ( \dashint_{\mathcal{B}_{2^k r}(z)} G^2 d\mu \right )^\frac{1}{2} \Bigg ),
\end{aligned}
\end{equation}
where $C_1=C_1(n,s,\theta,\Lambda,p)>0$. If $i_p=1$, then $q_1=p$ and the proof is finished. Otherwise, we note that since $r$ and $z$ are arbitrary, the estimate (\ref{q1est}) holds also in any ball that is contained in $B_r(z)$, which means that the estimate (\ref{J}) is satisfied with respect to $q=q_1$ in any ball contained in $B_r(z)$. Since also $U \in L^{p}(\mathcal{B}_{r}(z),\mu) \subset L^{q_2}(\mathcal{B}_{r}(z),\mu)$, if we choose $\delta$ smaller if necessary such that Corollary \ref{aprioriestcorscaled} is applicable with $q=q_1$ and $\widetilde q = q_2$, then all assumptions of Corollary \ref{aprioriestcorscaled} are satisfied with respect to $q=q_1$ and $\widetilde q = q_2 = (q_i+(q_i)^\star)/2 \in (q_1,(q_1)^\star)$, so that we obtain the estimate 
\begin{align*}
\left (\dashint_{\mathcal{B}_{r/2}(z)} U^{q_2} d\mu \right )^{\frac{1}{q_2}} \leq & C_2 \Bigg (\sum_{k=1}^\infty 2^{-k(s-\theta)} \left ( \dashint_{\mathcal{B}_{2^k r(z)}} U^2 d\mu \right )^\frac{1}{2} \\
& + \left ( \dashint_{\mathcal{B}_{r}(z)} G^{q_2} d\mu \right )^\frac{1}{q_2} + \sum_{k=1}^\infty 2^{-k(s-\theta)} \left ( \dashint_{\mathcal{B}_{2^k r}(z)} G^2 d\mu \right )^\frac{1}{2} \Bigg ),
\end{align*}
where $C_2=C_2(n,s,\theta,\Lambda,p)>0$.
By iterating this procedure $i_p$ times and using that $q_{i_p}=p$, we finally arrive at the estimate (\ref{estp}).
\end{proof}
\section{Proofs of the main results} \label{pmr}
We are now in the position to prove our main results.
\begin{thm} \label{mainint5}
	Let $\Omega \subset \mathbb{R}^n$ be a domain, $s \in (0,1)$, $\Lambda \geq 1$, $R>0$ and $p \in (2,\infty)$. Moreover, fix some $t$ such that
	\begin{equation} \label{trangex}
	s < t < \min  \left \{2s \left (1-\frac{1}{p} \right) ,1-\frac{2-2s}{p} \right \}.
	\end{equation} Then there exists some small enough $\delta=\delta(p,n,s,t,\Lambda)>0$, such that if $A \in \mathcal{L}_0(\Lambda)$ is $(\delta,R)$-BMO in $\Omega$ and if $\Phi$ satisfies the conditions (\ref{PhiLipschitz}) and (\ref{PhiMonotone}) with respect to $\Lambda$, 
	then for any weak solution $u \in W^{s,2}(\mathbb{R}^n)$
	of the equation
	$$
	L_A^\Phi u = f \text{ in } \Omega,
	$$
	we have the implication $$f \in L^\frac{np}{n+(2s-t)p}_{loc}(\Omega) \implies u \in W^{t,p}_{loc}(\Omega).$$ Moreover, for all relatively compact bounded open sets ${\Omega^\prime} \Subset {\Omega^{\prime \prime}} \Subset \Omega$, we have the estimate
	\begin{equation} \label{Wstest}
	[u]_{W^{t,p}(\Omega^\prime)} \leq C \left ([u]_{W^{s,2}(\mathbb{R}^n)} + ||f||_{L^{\frac{np}{n+(2s-t)p}}(\Omega^{\prime \prime})} \right ),
	\end{equation}
	where $C=C(n,s,t,\Lambda,R,p,\Omega^\prime,\Omega^{\prime \prime})>0$.
\end{thm}

\begin{rem} \normalfont
Note that in view of Proposition \ref{Sobcont}, the conclusion that $u \in W^{t,p}_{loc}(\Omega)$ for some $p \in (2,\infty)$ and for any $t$ in the range (\ref{trangex}) also implies that $u \in W^{t,p_0}_{loc}(\Omega)$ for any $p_0 \in (1,p]$.
\end{rem}

\begin{proof}[Proof of Theorem \ref{mainint5}]
Fix relatively compact bounded open sets ${\Omega^\prime} \Subset {\Omega^{\prime \prime}} \Subset \Omega$.
Let $\delta=\delta(p,n,s,\theta,\Lambda)>0$ be given by Proposition \ref{aprioriestx}. There exists some small enough $r_1 \in (0,1)$ such that $2r_1 \leq R$ and $B_{2r_1}(z) \Subset \Omega^{\prime \prime}$ for any $z \in \Omega^\prime$. Now fix some $z \in \Omega^{\prime \prime}$. Since $A$ is $(\delta,R)$-BMO in $\Omega$, we obtain that $A$ is $\delta$-vanishing in $B_{2r_1}(z)$.
Let $\{\psi_m\}_{m=1}^\infty$ be a sequence of standard mollifiers in $\mathbb{R}^{n}$ with the properties 
\begin{equation} \label{molprop}
\psi_m \in C_0^\infty(B_{1/m}), \quad \psi_m \geq 0, \quad \int_{\mathbb{R}^{n}} \psi_m(x)dx=1 \quad \text{for all } m \in \mathbb{N}.
\end{equation}
In addition, for any $m \in \mathbb{N}$ we define $$A_m(x,y):= \int_{\mathbb{R}^{n}} \int_{\mathbb{R}^{n}} A(x-x^\prime,y-y^\prime) \psi_m(x^\prime) \psi_m (y^\prime)dy^\prime dx^\prime .$$ 
Clearly, $A_m$ is symmetric and belongs to $\mathcal{L}_0(\Lambda)$ for any $m \in\mathbb{N}$. In addition, there exists some large enough $m_0 \in \mathbb{N}$ such that $$ \frac{1}{m_0} < \min\{r_1,\textnormal{dist}(B_{2r_1}(z),\Omega)\}.$$
Fix $r>0$ and $x_0,y_0 \in B_{r_1}(z)$ with $B_r(x_0) \subset B_{r_1}(z)$, $B_r(y_0) \subset B_{r_1}(z)$. Then for any $m \geq m_0$ and all $x^\prime,y^\prime \in B_{1/m}$, we have $B_{r}(x_0-x^\prime) \subset B_{2r_1}(z)$ and $B_{r}(y_0-y^\prime) \subset B_{2r_1}(z)$. Therefore, since $A$ is $\delta$-vanishing in $B_{2r_1}(z)$, for any $m \geq m_0$ and $x^\prime,y^\prime \in B_{1/m}$, we have $$ \dashint_{B_{r}(x_0-x^\prime)} \dashint_{B_{r}(y_0-y^\prime)} |A(x,y)- \overline A_{r,x_0-x^\prime,y_0-y^\prime}|dydx \leq \delta .$$ 
Therefore, together with changes of variables, Fubini's theorem and (\ref{molprop}),
we obtain
\begin{align*}
& \dashint_{B_r(x_0)} \dashint_{B_r(y_0)} |A_m(x,y)-\overline {(A_m)}_{r,x_0,y_0}|dydx \\
\leq & \int_{B_\frac{1}{m}} \int_{B_\frac{1}{m}} \dashint_{B_r(x_0)} \dashint_{B_r(y_0)} \bigg |A(x-x^\prime,y-y^\prime) - \dashint_{B_r(x_0)} \dashint_{B_r(y_0)} A(x_1-x^\prime,y_1-y^\prime)dy_1dx_1 \bigg |dydx \\ & \times \psi_m(x^\prime) \psi_m (y^\prime) dy^\prime dx^\prime \\
= & \int_{B_\frac{1}{m}} \int_{B_\frac{1}{m}} \left ( \dashint_{B_r(x_0-x^\prime)} \dashint_{B_r(y_0-y^\prime)} |A(x,y)-\overline A_{r,x_0-x^\prime,y_0-y^\prime}|dydx \right ) \psi_m(x^\prime) \psi_m (y^\prime) dy^\prime dx^\prime \\ \leq & \delta \int_{B_\frac{1}{m}} \int_{B_\frac{1}{m}} \psi_m(x^\prime) \psi_m (y^\prime) dy^\prime dx^\prime = \delta ,
\end{align*}
so that we conclude that $A_m$ is $\delta$-vanishing in $B_{r_1}(z)$ for any $m \geq m_0$.
Now define
$$ \widetilde A_m(x,y) := \begin{cases} \normalfont
A_m(x,y) & \text{if } (x,y) \in \mathcal{B}_{2r_1}(z) \\
A(x,y) & \text{if } (x,y) \notin \mathcal{B}_{2r_1}(z).
\end{cases} $$
Since $A \in L^\infty(\mathbb{R}^n) \subset L^1_{loc}(\mathbb{R}^n)$, by standard properties of mollifiers we have
\begin{equation} \label{Aapr}
\widetilde A_m \xrightarrow{m \to \infty} A \text{ in } L^1(\mathcal{B}_{2r_1}(z))
\end{equation}
and also $\widetilde A_m(x,y) \in C^\infty(\mathcal{B}_{2r_1}(z))$. In particular, $\widetilde A_m$ is continuous inside of $B_{2r_1}(z) \times B_{2r_1}(z)$. \par
Next, for any $m \in \mathbb{N}$ and $x \in \Omega_m := \left \{x \in \Omega \mid \textnormal{dist}(x, \partial \Omega) > 1/m \right \}$, define
$$ f_m(x):= \int_{\Omega} f(y)\psi_m(x-y)dy.$$
Since $f \in L^\frac{np}{n+(2s-t)p}_{loc}(\Omega)$ and $B_{2r_1}(z) \Subset \Omega$, again by standard properties of mollifiers we have
\begin{equation} \label{appf}
f_m \xrightarrow{m \to \infty} f \text{ in } L^\frac{np}{n+(2s-t)p}(B_{2r_1}(z))
\end{equation}
and $f_m \in C^\infty(B_{2r_1}(z)) \subset L^\infty(B_{\frac{3}{2}r_1}(z))$. Next, for any $m \geq m_0$ we let $u_m \in W^{s,2}(\mathbb{R}^n)$ be the unique weak solution of the Dirichlet problem 
\begin{equation} \label{constcof3k}
\begin{cases} \normalfont
L_{\widetilde A_m}^\Phi u_m = f_m & \text{ in } B_{2r_1}(z) \\
u_m = u & \text{ a.e. in } \mathbb{R}^n \setminus B_{2r_1}(z).
\end{cases}
\end{equation}
Since $2_\star=\frac{2n}{n+2s} < \frac{np}{n+(2s-t)p}$, we can choose the number $\sigma_0>0$ from Theorem \ref{KMS} small enough such that
$$2_\star +\sigma_0 \leq \frac{np}{n+(2s-t)p}.$$
Then by Proposition \ref{appplxy}, (\ref{Cst}), H\"older's inequality, (\ref{Aapr}) and (\ref{appf}), for $w_m:=u-u_m$ we obtain
\begin{align*}
& \int_{\mathbb{R}^n} \int_{\mathbb{R}^n} \frac{(w_m(x)-w_m(y))^2}{|x-y|^{n+2s}}dydx \\
\leq & C_1 \omega(A- \widetilde A_m,2r_1,x_0)^\frac{\gamma}{n-\gamma} \mu(\mathcal{B}_{r_1}(z)) \left (\sum_{k=1}^\infty 2^{-k(s-\theta)} \left ( \dashint_{\mathcal{B}_{2^kr_1}(z)} U^2 d\mu \right )^\frac{1}{2} \right )^2 \\ & + C_1 \omega(A- \widetilde A_m,2r_1,z)^\frac{\gamma}{n-\gamma} r_1^{2(s-\theta)}\mu(\mathcal{B}_{r_1}(z)) \left ( \dashint_{B_{2r_1}(z)}|f|^{2_\star+\sigma_0} dx \right )^\frac{2}{2_\star+\sigma_0} \\
& + C_1 r_1^{2(s-\theta)}\mu(\mathcal{B}_{r_1}(z)) \left (\dashint_{B_{2r_1}(z)}|f-f_m|^{2_\star} dx \right )^\frac{2}{2_\star} \\
\leq & C_2 \bigg ( ||A-\widetilde A_m||_{L^1(\mathcal{B}_{2r_1}(z))}^\frac{\gamma}{n-\gamma} [u]_{W^{s,2}(\mathbb{R}^n)}^2 + ||A-\widetilde A_m||_{L^1(\mathcal{B}_{2r_1}(z))}^\frac{\gamma}{n-\gamma} ||f||_{L^{\frac{np}{n+(2s-t)p}}(B_{2r_1}(z))}^2 \\
& + ||f-f_m||_{L^{\frac{np}{n+(2s-t)p}}(B_{2r_1}(z))}^2 \bigg) \xrightarrow{k \to \infty} 0,
\end{align*}
where $C_1=C_1(n,s,\theta,\Lambda,\sigma_0)>0$ and $C_2=C_2(n,s,\theta,\Lambda,\sigma_0,r_1)>0$.
Thus, we deduce
\begin{equation} \label{UktoU}
\lim_{m \to \infty} [u_m]_{W^{s,2}(\mathbb{R}^n)}
= [u]_{W^{s,2}(\mathbb{R}^n)}.
\end{equation}
Next, for any $m \in \mathbb{N}$ let $g_m \in W^{s,2}(\mathbb{R}^n)$ be the unique weak solution of the Dirichlet problem
\begin{equation} \label{constcof3gk}
\begin{cases} \normalfont
(-\Delta)^s g_m = f_m & \text{ in } B_{2r_1}(z) \\
g_m = 0 & \text{ a.e. in } \mathbb{R}^n \setminus B_{2r_1}(z).
\end{cases}
\end{equation}
Then by (\ref{basic}), we have the estimate
\begin{equation} \label{fkestimate}
[g_m]_{W^{s,2}(\mathbb{R}^n)} \leq C_3 ||f_m||_{L^{2_\star}(B_{2r_1}(z))} \leq C_4 ||f_m||_{L^{\frac{np}{n+(2s-t)p}}(B_{2r_1}(z))},
\end{equation}
where $C_3=C_3(n,s,r_1)>0$ and $C_4=C_4(n,s,t,p,r_1)>0$.
In addition, by Theorem \ref{H2spest} we have the estimate
\begin{equation} \label{H2spa}
||g_m||_{H^{2s,\frac{np}{n+(2s-t)p}}(B_{r_1}(z))} \leq C_5 ||f_m||_{L^{\frac{np}{n+(2s-t)p}}(B_{2r_1}(z))},
\end{equation}
where $C_5=C_5(n,s,t,p)>0$.
Also, by Proposition \ref{BesselTr}, we have 
\begin{equation} \label{embend}
[g_m]_{W^{t,p}(B_{r_1}(z))} \leq C_6 ||g_m||_{H^{2s,\frac{np}{n+(2s-t)p}}(B_{r_1}(z))},
\end{equation}
where $C_6=C_6(n,s,t,p)>0$. In view of (\ref{constcof3k}) and (\ref{constcof3gk}), $u_m$ is a weak solution of the equation
$$ L_{\widetilde A_m}^\Phi u_m= (-\Delta)^s g_m \text{ in } B_{2r_1}(z).$$
Define $$U_m(x,y):=\frac{|u_m(x)-u_m(y)|}{|x-y|^{s+\theta}}, \quad G_m(x,y):=\frac{|g_m(x)-g_m(y)|}{|x-y|^{s+\theta}}.$$ 
Since $\widetilde A_m$ is continuous in $B_{2r_1}(z) \times B_{2r_1}(z)$ and $f_m \in L^\infty(B_{\frac{3}{2}r_1}(z))$, by Theorem \ref{HiHol} we have $u_m \in C^{s+\theta}(B_{r_1}(z))$ and therefore $U_m \in L^\infty(B_{r_1}(z),\mu) \subset L^p(B_{r_1}(x_0),\mu)$. Therefore, by Proposition \ref{aprioriestx}, (\ref{Cst}), (\ref{fkestimate}), (\ref{embend}) and (\ref{H2spa}), we have
\begin{align*}
\left (\dashint_{\mathcal{B}_{r_1/2}(z)} U_m^{p} d\mu \right )^{\frac{1}{p}} \leq & C_7 \Bigg (\sum_{k=1}^\infty 2^{-k(s-\theta)} \left ( \dashint_{\mathcal{B}_{2^k r_1}(z)} U_m^2 d\mu \right )^\frac{1}{2} \\
& + \left ( \dashint_{\mathcal{B}_{r_1}(z)} G_m^{p} d\mu \right )^\frac{1}{p} + \sum_{k=1}^\infty 2^{-k(s-\theta)} \left ( \dashint_{\mathcal{B}_{2^k r_1}(z)} G_m^2 d\mu \right )^\frac{1}{2} \Bigg ) \\
\leq & C_8 \left ([u_m]_{W^{s,2}(\mathbb{R}^n)} + [g_m]_{W^{t,p}(B_{r_1}(z)} + [g_m]_{W^{s,2}(\mathbb{R}^n)} \right ) \\
\leq & C_9 \left ([u_m]_{W^{s,2}(\mathbb{R}^n)} + ||f_m||_{L^{\frac{np}{n+(2s-t)p}}(B_{2r_1}(z))} \right ) ,
\end{align*}
where all constants depend only on $n,s,t,\theta,\Lambda,p$ and $r_1$.
Combining the previous display with Fatou's Lemma (which is applicable after passing to a subsequence if necessary), (\ref{UktoU}) and (\ref{appf}), we conclude that
\begin{equation} \label{liminfest}
\begin{aligned}
\left (\dashint_{\mathcal{B}_{r_1/2}(z)} U^{p} d\mu \right )^{\frac{1}{p}} \leq & \liminf_{m \to \infty} \left (\dashint_{\mathcal{B}_{r_1/2}(z)} U_m^{p} d\mu \right )^{\frac{1}{p}} \\
\leq & C_{10} \lim_{m \to \infty} \left ([u_m]_{W^{s,2}(\mathbb{R}^n)} + ||f_m||_{L^{\frac{np}{n+(2s-t)p}}(B_{2r_1}(z))} \right ) \\
= & C_{10} \left ([u]_{W^{s,2}(\mathbb{R}^n)} + ||f||_{L^{\frac{np}{n+(2s-t)p}}(B_{2r_1}(z))} \right ) ,
\end{aligned}
\end{equation}
where $C_{10}=C_{10}(n,s,t,\theta,\Lambda,p,r_1)>0$. \par 
Since $\left \{B_{r_1/2}(z) \right \}_{z \in {\Omega^\prime}}$ is an open covering of $\overline {\Omega^\prime}$ and $\overline {\Omega^\prime}$ is compact, there exists a finite subcover $\left \{B_{r_1/2}(z_i) \right \}_{i=1}^N$ of $\overline {\Omega^\prime}$ and hence of $\Omega^\prime$. Now summing over $i=1,...,N$ and using the estimate (\ref{liminfest}) for any $i$, we arrive at
\begin{equation} \label{cest}
\begin{aligned}
\left (\int_{\Omega^\prime \times \Omega^\prime} U^{p} d\mu \right )^{\frac{1}{p}} \leq & \sum_{i=1}^N \left (\int_{\mathcal{B}_{r_1/2}(z_i)} U^{p} d\mu \right )^{\frac{1}{p}} \\
\leq & \sum_{i=1}^N C_{11} \left ([u]_{W^{s,2}(\mathbb{R}^n)} + ||f||_{L^{\frac{np}{n+(2s-t)p}}(B_{2r_1}(z_i))} \right ) \\
\leq & C_{11}N \left ([u]_{W^{s,2}(\mathbb{R}^n)} + ||f||_{L^{\frac{np}{n+(2s-t)p}}(\Omega^{\prime \prime})} \right ),
\end{aligned}
\end{equation}
where $C_{11}=C_{11}(n,s,t,\theta,\Lambda,p,r_1)>0$.
Clearly, for any $t$ in the range (\ref{trangex}), there exists some $0<\theta <\min \{s,1-s\}$ such that $t= s+\theta \left (1-\frac{2}{p} \right)$,
so that by choosing this $\theta$ in our definition of $\mu$, we arrive at
\begin{align*}
[u]_{W^{t,p}(\Omega^\prime)} = \left (\int_{\Omega^\prime \times \Omega^\prime} U^{p} d\mu \right )^{\frac{1}{p}} \leq C \left ([u]_{W^{s,2}(\mathbb{R}^n)} + ||f||_{L^{\frac{np}{n+(2s-t)p}}(\Omega^{\prime \prime})} \right ),
\end{align*}
where $C=C(n,s,t,\Lambda,p,\Omega^\prime,\Omega^{\prime \prime})>0$. Here we also used that $r_1$ depends only on $R,\Omega^\prime$ and $\Omega^{\prime \prime}$ and that $\theta$ depends only on $n,s$ and $t$. This proves the estimate (\ref{Wstest}). \par 
Next, let us prove that $u \in L^p_{loc}(\Omega)$. For any $q \in (2,p]$, we fix some 
$$s< t_q < \min  \left \{2s \left (1-\frac{1}{q} \right) ,1-\frac{2-2s}{q} \right \}$$ and define
$$ q^\star := \begin{cases} 
\min \left \{\frac{nq}{n-t_q q},p \right \}, & \text{if } t_q q < n \\
p, & \text{if } t_q q \geq n .
\end{cases}$$
Since $u \in W^{s,2}(\mathbb{R}^n)$, by the Sobolev embedding (Proposition \ref{Sobemb}) we have $u \in L^{2^\star}_{loc}(\Omega)$, where $2^\star:=\min \left \{\frac{2n}{n-2s},p \right \}$. If $p=2^\star$, the proof is finished. Otherwise, together with the estimate (\ref{Wstest}) with $p$ replaced by $2^\star$, we conclude that $u \in W^{t_{2^\star},2^\star}_{loc}(\Omega)$. Again by the Sobolev embedding, we then obtain that $u \in L^{2^{{\star}^\star}}_{loc}(\Omega)$. If $p=2^{{\star}^\star}$, the proof is finished. Otherwise, iterating this procedure also leads to the conclusion that $u \in L^p_{loc}(\Omega)$ at some point, so that we conclude that $u \in W^{t,p}_{loc}(\Omega)$. This finishes the proof.
\end{proof}

\begin{proof}[Proof of Theorem \ref{mainint5z}]
Let us first handle the case when $t>s$. Since $A$ is assumed to be VMO in $\Omega$, for any $\delta>0$, there exists some $R>0$ such that is $(\delta,R)$-BMO in $\Omega$. Therefore, in this case Theorem \ref{mainint5z} follows directly from Theorem \ref{mainint5}. This finishes the proof in the case then $t>s$. \par 
In the case when $t=s$, fix some small enough $\varepsilon>0$ such that $\widetilde s := s+\varepsilon$ belongs to the range (\ref{trangex}) and $\widetilde p := \frac{np}{n+\varepsilon p}>2$.
Then by assumption and an elementary computation, we have $f \in L^\frac{np}{n+sp}_{loc}(\Omega) = L^\frac{n \widetilde p}{n+ (2s-\widetilde s) \widetilde p}_{loc}(\Omega)$. By applying the previous case when $t>s$ with $t=\widetilde s$ and with $p$ replaced by $\widetilde p$, we obtain that $u \in W^{\widetilde s,\widetilde p}_{loc}(\Omega)$, which by Proposition \ref{BesselTr} leads to $u \in W^{s,p}_{loc}(\Omega)$. Thus, the proof is finished.
\end{proof}

\begin{proof}[Proof of Theorem \ref{mainint5zx}]
	Fix some $t$ such that $s \leq t < 1$. First, we assume that $t$ satisfies (\ref{trangexy}). Then we have $n>(2s-t)q$ and set $p:=\frac{nq}{n-(2s-t)q}$, so that we have $q=\frac{np}{n+(2s-t)p}$ and thus $f \in L^\frac{np}{n+(2s-t)p}_{loc}(\Omega)$. Then in view of (\ref{trangexy}) and elementary computations, we obtain that $p>2$ and
	$$
		s\leq t < \min  \left \{2s \left (1-\frac{1}{p} \right) ,1-\frac{2-2s}{p} \right \}
	$$
	so that by Theorem \ref{mainint5z} we obtain $u \in W^{t,p}_{loc}(\Omega) = W^{t,\frac{nq}{n-(2s-t)q}}_{loc}(\Omega)$. \par Next, suppose that $t = 2s-\frac{n}{q}$. Since $t<1$, in this case we have $2s-\frac{n}{q} < 1$. Using the latter inequality, a direct computation shows that
	\begin{equation} \label{simpcompx}
		2s-\frac{n}{q} < \min \left \{2s \left (1- \frac{n}{(n+2s)q} \right ),1-\frac{(2-2s)(n+q-2sq)}{(n+2-2s)q} \right \},
	\end{equation} so that there exists some $t^\prime \geq s$ such that 	
$$ 2s-\frac{n}{q} < t^\prime< \min \left \{2s \left (1- \frac{n}{(n+2s)q} \right ),1-\frac{(2-2s)(n+q-2sq)}{(n+2-2s)q} \right \}.$$
Then by the previous case, we obtain that $u \in W^{t^\prime,\frac{nq}{n-(2s-t^\prime)q}}_{loc}(\Omega)$, which by Proposition \ref{BesselTr} and Proposition \ref{Sobcont} implies that $u \in W^{2s-\frac{n}{q},p}_{loc}(\Omega)=W^{t,p}_{loc}(\Omega)$ for any $p \in (1,\infty)$. \par
Finally, if we have $2s-\frac{n}{q} >t$, then there exists some $\varepsilon>0$ such that $2s-\frac{n}{q} > t + \varepsilon$. Then for any $p >1$, we have $q \geq \frac{np}{n+(2s-t-\varepsilon)p}$ and therefore $f \in L^\frac{np}{n+(2s-t-\varepsilon)p}_{loc}(\Omega)$. Furthermore, for $p>\max \left \{\frac{2s}{2s-t},\frac{2-2s}{1-t} \right \}$, we see that $$s \leq t+\varepsilon < \min  \left \{2s \left (1-\frac{1}{p} \right) ,1-\frac{2-2s}{p} \right \},$$ so that by Theorem \ref{mainint5z} we obtain that $u \in W^{t+\varepsilon,p}_{loc}(\Omega)$ for any $p>\max \left \{\frac{2s}{2s-t},\frac{2-2s}{1-t} \right \}$, which by Proposition \ref{Sobcont} implies that $u \in W^{t,p}_{loc}(\Omega)$ for any $p \in (1,\infty)$, which finishes the proof.
\end{proof}

\begin{proof}[Proof of Theorem \ref{C2sreg1}]
	First of all, we remark that by the assumption that $q>\frac{n}{2s}$, we always have $2s-\frac{n}{q}>0.$ Moreover, by a simple computation we have $q> \frac{2n}{n+2(2s-t)}$ for any $t <1$. Now consider the case when $0<2s-\frac{n}{q}<1$.
	Then in view of (\ref{simpcompx}), there exists some $t \geq s$ such that
	$$
	2s-\frac{n}{q} < t< \min \left \{2s \left (1- \frac{n}{(n+2s)q} \right ),1-\frac{(2-2s)(n+q-2sq)}{(n+2-2s)q} \right \},
	$$
	so that by Theorem \ref{mainint5zx}, for $p:=\frac{nq}{n-(2s-t)q}$ we obtain that $u \in W^{t,p}_{loc}(\Omega)$. Since in addition we have $$t-\frac{n}{p}=t-\frac{n(n-(2s-t)q)}{nq}=2s-\frac{n}{q}>0, $$
	by the Sobolev embedding (Proposition \ref{Sobemb}) we conclude that $u \in C^{2s-\frac{n}{q}}_{loc}(\Omega)$. \par
	Next, consider the case when $2s-\frac{n}{q} \geq 1$. In this case, by Theorem \ref{mainint5zx} we obtain $u \in W^{t,p}_{loc}(\Omega)$ for any $s \leq t<1$ and any $p \in (1,\infty)$, which by the Sobolev embedding implies that $u \in C^\alpha_{loc}(\Omega)$ for any $\alpha \in (0,1)$. This finishes the proof.
\end{proof}

\begin{rem} \label{endremark} \normalfont
Our main results remain valid for another large class of coefficients $A$ that in general might not be VMO. Namely, the conclusions of Theorem \ref{mainint5z}, Theorem \ref{mainint5zx} and Theorem \ref{C2sreg1} remain true if we replace the assumption that $A$ is VMO with the following assumption used in for example \cite{MeH}:
Namely, our main results remain true if there exists some small $\varepsilon>0$ such that
\begin{equation} \label{contkernel}
	\lim_{h \to 0} \sup_{\substack{_{x,y \in K}\\{|x-y| \leq \varepsilon}}} |A(x+h,y+h)-A(x,y)| =0 \quad \text{for any compact set } K \subset \Omega.
\end{equation}
This is because by \cite[Theorem 1.1]{MeH}, the H\"older estimate from Theorem \ref{HiHol} remains valid under the assumption (\ref{contkernel}). Therefore, in contrast to the case when $A$ is VMO, under the assumption (\ref{contkernel}) the above proof can be executed without the need to freeze the coefficient $A$, so that the proof actually simplifies in this case. The condition (\ref{contkernel}) is for example satisfied in the case when $A$ is translation invariant in $\Omega$, that is, if there exists a measurable function $a: \mathbb{R}^n \to \mathbb{R}$ such that $A(x,y)=a(x-y)$ for all $x,y \in \Omega$. Since in this case $A$ is otherwise not required to satisfy any additional smoothness assumption, $A$ might not be VMO in $\Omega$ but still satisfies (\ref{contkernel}).
\end{rem}

\bibliographystyle{amsplain}

\end{document}